\documentclass[11pt]{article}
\usepackage{amsfonts,amsmath,latexsym}
\usepackage{pgfplots}
\DeclareUnicodeCharacter{2212}{−}
\usepgfplotslibrary{groupplots,dateplot}
\usetikzlibrary{patterns,shapes.arrows}
\pgfplotsset{compat=newest}
\usepackage{bm}
\usepackage[T1]{fontenc}
\usepackage[utf8]{inputenc}
\usepackage{verbatim}
\usepackage{algorithm}
\usepackage{algpseudocode}
\usepackage{eurosym}
\usepackage{enumitem}
\usepackage{dsfont}
\usepackage{hyperref}
\usepackage{graphicx}
\usepackage{epsf}
\usepackage{amsthm}
\usepackage{color}
\usepackage{amssymb}

\usepackage{comment}

\usepackage{caption}
\usepackage{subfigure}

\usepackage{fancyhdr} 

\usepackage{palatino}
\def\R{{\mathbb R}}
\def\E{{\mathbb E}}
\def\P{{\mathbb P}}
\def\N{{\mathbb N}}
\def\F{{\mathcal F}}
\def\B{{\mathcal B}}
\def\Pma{{\mathcal P}}

\def\1{{\mathds 1}}
\def\vphi{{\varphi}}
\newtheorem{theorem}{Theorem}[section]
\newtheorem{prop}[theorem]{Proposition}
\newtheorem{coro}[theorem]{Corollary}
\newtheorem{remark}[theorem]{Remark}

\newtheorem{lemma}[theorem]{Lemma}
\newtheorem{hyp}[theorem]{Hypothesis}
\newtheorem{definition}[theorem]{Definition}

\renewcommand{\theequation}{\arabic{section}.\arabic{equation}}
\newcommand{\argmin}{\mathop{\mathrm{arg\,min}}}

\def\R{\mathbb R}
\def\N{\mathbb N}

\def\E{\mathbb E}
\def\P{\mathbb P}
\def\Q{\mathbb Q}
\def\U{\mathbb U}
\def\sha{{\cal A}}
\def\shb{{\cal B}}
\def\shc{{\cal C}}

\def\shf{{\cal F}}

\def\shj{{\cal J}}
\def\shk{{\cal K}}

\def\shl{{\cal L}}
\def\shp{{\cal P}}

\def\shv{{\cal V}}
\def\shy{{\cal Y}}
\def\shz{{\cal Z}}

\author{
{\sc Thibaut BOURDAIS}
\thanks{ENSTA Paris, Institut Polytechnique de Paris.
Unit\'e de Math\'ematiques Appliqu\'ees (UMA).
 E-mail:{ \tt thibaut.bourdais@ensta-paris.fr}} 
{\sc,}\ {\sc Nadia OUDJANE}
\thanks{EDF R\&D,   and FiME (Laboratoire de Finance des March\'es de l'Energie
(Dauphine, CREST,  EDF R\&D) www.fime-lab.org). 
E-mail:{\tt  
nadia.oudjane@edf.fr}}
\ {\sc and}\ {\sc Francesco RUSSO} 
\thanks{ENSTA Paris, Institut Polytechnique de Paris.
Unit\'e de Math\'ematiques Appliqu\'ees (UMA). 
 E-mail:{\tt  francesco.russo@ensta-paris.fr}.
 }}

\date{July 2025}

\title{An entropy penalized approach for stochastic control problems. Complete version.}

\newcommand{\MBFigure}[6]{
$\left. \right.$ \\
\refstepcounter{figure}
\addcontentsline{lof}{figure}{\numberline{\thefigure}{\ignorespaces #5}}
\begin{center}
\begin{minipage}{#1cm}
\centerline{\includegraphics[width=#2cm,angle=#3]{#4}}
\begin{center}
\upshape{F\textsc{ig} \normal
\end{center}
size{\thefigure}. $-$} #5
\end{center}
\label{#6}
\end{minipage}
\end{center}
$\left. \right.$ \\}

\begin{document}
\maketitle
\begin{abstract}
  In this paper, we propose an original approach to stochastic control problems. We consider a weak formulation that is written as an optimization (minimization) problem on the space of probability measures. We then introduce a penalized version of this problem obtained by splitting the minimization variables and penalizing the discrepancy 
  between the two variables via an  entropy term.
  We show that the penalized problem provides a good approximation of
  the original problem when the weight of the entropy
  penalization term is large enough. Moreover, the penalized problem
  has the advantage of giving rise to two optimization subproblems that are easy to solve in each of the two optimization
  variables when the other is fixed. We take advantage of this property to propose an alternating optimization procedure
  that converges to the infimum of the penalized problem with a rate $O(1/k)$, where $k$ is the number of iterations.
  The relevance of this approach is illustrated by solving a high-dimensional stochastic control problem aimed at controlling consumption in electrical systems.

 \end{abstract}
\medskip\noindent {\bf Key words and phrases:}  
Stochastic control;  optimization; Donsker-Varadhan representation;
exponential twist; relative entropy; demand-side management.

\medskip\noindent  {\bf 2020  AMS-classification}: 49M99; 49J99; 60H10; 
 60J60; 65C05.

\section{Introduction}

{\it General framework.} Stochastic control problems appear in many fields of application such as robotics \cite{RobotPathIntegral}, economics and finance \cite{touzibook}. These problems are, either tackled using the Pontryagin's optimality principle or
the dynamic programming principle,
which allows
the representation of the value function via  nonlinear Hamilton-Jacobi-Bellman PDEs or Backward Stochastic Differential
Equations (BSDEs).
The idea of this paper is to propose a radically different approach based on a weak reformulation of the stochastic control problem as an optimization problem on the space of probability measures. We propose
an entropic penalization of this optimization problem which suitably approximates the original control problem.
We prove the convergence of an alternating optimization procedure to the infimum of the penalized problem: the interest of this procedure is demonstrated in simulation compared with classical techniques relying on dynamic programming.
The proof of the convergence of our algorithm relies on geometric arguments rather than classical convex optimization techniques.

{\it Problem formulation.} On some filtered probability space $(\Omega, \shf,\P)$, we are interested in a problem of the type
\begin{equation}
	\label{eq:strongControlIntro}
	J^*_{strong} := \inf_{\nu } \E\left[\int_0^T f(r, X_r^{\nu}, \nu_r) dr
	+ g(X_T^\nu)\right],
\end{equation}
where  $ \nu $ is a progressively measurable process taking values
in some fixed
convex domain $\U \subset \R^d$. 
$X = X^\nu$ will be a controlled diffusion process taking values in $\R^d$ of the form
\begin{equation} \label{eq:nu}
	X_t^\nu = x + \int_0^t b(r, X_r^\nu, \nu_r)dr + \int_0^t \sigma(r, X_r^\nu) dW_r.
\end{equation}
  Under some mild supplementary assumptions,
Problem \eqref{eq:strongControlIntro} can be reformulated as an optimization program on the space of probability measures, in the  form
\begin{equation}
	\label{eq:controlProblemIntro}
	J^*:= \inf_{\P \in \Pma_{\U}} J(\P),\quad\textrm{with}\quad 
	J(\P):=\E^\P\left[\int_0^T f(r, X_r, \nu_r^\P)dr + g(X_T)\right],     
\end{equation}
 $\Pma_{\U}$ being a set of probability measures, defined in Definition \ref{def:PU}, such that, under $\P \in \Pma_\U$, the canonical process
$X$ is decomposed as
\begin{equation} \label{eq:nu_u}
  X_t = x + \int_0^t b(r, X_r, \nu_r^\P)dr + \int_0^t \sigma(r,X_r) dW_r,
\end{equation}
where $\nu^\P$ is a progressively measurable process  with respect to the
canonical filtration $\F^X$  of $X$ taking values in $\U$ and $W$ is some standard Brownian motion.
In particular we will have $J^*_{strong} = J^*$.
In the sequel, to insist on the path-dependence of $\nu$, we will write
$\nu_t  = \nu(t,X)$.
We refer to Appendix \ref{app:equiControl} for the precise link between the different formulations of stochastic control problems \eqref{eq:strongControlIntro} and \eqref{eq:controlProblemIntro}.

One major difficulty in analyzing Problem \eqref{eq:controlProblemIntro} is the lack of convexity of the functional $J$ in~\eqref{eq:controlProblemIntro} with respect to $\P$,
even though the literature includes some techniques to transform the original problem into a minimization
of a convex functional, see e.g. \cite{benamou}. 
For that reason, we cannot rely on classical convex analysis techniques, see e.g. \cite{ekeland}, in order to perform related algorithms, see e.g. \cite{Bonnans_Gilbert}.
As announced above, our method consists in replacing Problem~\eqref{eq:controlProblemIntro}  with the penalized version
\begin{equation}
	\label{eq:penalizedProblemIntro}
	\shj_\epsilon^*:=
	\inf_{(\P, \Q) \in \mathcal{A}} \shj_\epsilon(\Q,\P),\quad \textrm{with}\quad 
	\shj_\epsilon (\Q,\P):=
	\E^\Q \left[\int_0^T f(r, X_r, \nu_r^\P)dr
	+ g(X_T)\right] + \frac{1}{\epsilon}H(\Q | \P),
\end{equation}
where $\mathcal{A}$ is a subset of elements $(\P, \Q) \in \Pma(\Omega)^2$ defined in Definition \ref{def:A}, $H$ is the relative entropy, see Definition \ref{def:klDiv}, and the penalization parameter $\epsilon > 0$ is intended to vanish to zero, in order to impose $\Q = \P$. 

{\it Main contributions.} In Theorem \ref{th:existenceSolutionRegProb} one shows that
the infimum in \eqref{eq:penalizedProblemIntro} is indeed a minimum
$\shj_\epsilon^*=\shj_\epsilon (\Q^*_\epsilon,\P^*_\epsilon)$,
attained
on some admissible couple of probability measures $(\P^*_\epsilon, \Q^*_\epsilon)\in \mathcal{A}$.
Given  one solution $(\P^*_\epsilon, \Q^*_\epsilon)$ of Problem \eqref{eq:penalizedProblemIntro}, 
Proposition \ref{prop:approximateControl}
shows that $\P_\epsilon^*$
is an approximate solution of Problem \eqref{eq:controlProblemIntro}, in the sense 
that the infimum $J^*$ can indeed be approached by $J(\P^*_\epsilon)$ where $\P^*_\epsilon \in \Pma_\U$ when $\epsilon \rightarrow 0$, and more precisely $J(\P^*_\epsilon) - J^* = O(\epsilon)$.
The interest of the penalized Problem  \eqref{eq:penalizedProblemIntro}, 
with respect to the original Problem \eqref{eq:controlProblemIntro}, is that the minimization of the functional $\shj_\epsilon$, with respect to one variable $\Q$ or $\P$ (the other variable being fixed), can be provided quasi-explicitly. This is the object of Section \ref{sec:subProblems}. Indeed, Proposition \ref{prop:pointwiseMinimization} states that the minimization with respect to $\P$ can be reduced to a pointwise minimization, provided that $\Q$ has a Markovian decomposition. In this situation, there exists a function
  $(t, x) \mapsto u(t, x) \in \U$, such that, for all $(t, x) \in [0, T] \times \R^d$, $u(t, x)$ is independently obtained as the minimum of a strictly convex function and such that the infimum of the minimization $\underset{\P \in \shp_\U}{\inf}{\shj_\epsilon(\Q, \P)}$ is attained by the unique probability measure $\P \in \shp_\U$ verifying $\nu_t^\P = u(t, X_t)$.
Concerning the minimization with respect to $\Q$, Proposition \ref{prop:markovianDrift} characterizes the explicit solution of the subproblem. In fact, this is a well-known problem
in the area of large deviations, see \cite{DupuisEllisLargeDeviations}.
It gives rise to a variational representation formula relating log-Laplace transform of the costs and relative entropy,
which is linked to a specific case of stochastic optimal control,
for which it is possible to linearize the HJB equation by an exponential transform, see \cite{FlemingExpoTrans, FlemingPathIntegral}. This type of problem is known as {\it path integral control} and it has been extensively studied with many applications,
see \cite{ControlFlexibility,PathIntegralControl, RobotPathIntegral}.

In Section \ref{sec:algo} we introduce an alternating minimization procedure \eqref{eq:alternateMinimizationProcedure}, which consists in solving
sequentially
each subproblem, in $\Q$ and $\P$, alternatively.
In Theorem \ref{th:convAlgo}, we prove that the iterated values generated by this procedure converge to the minimum value $\shj^*_\epsilon$.
We insist again on the fact
that $\shj_\epsilon $
is not jointly convex with respect to $(\Q, \P)$, so, the proof of Theorem \ref{th:convAlgo} relies on geometric arguments developed in \cite{CsiszarAlternating}. In Section \ref{sec:example}, we show the relevance of this algorithm compared to classical dynamic programming techniques, by performing an application to the control of thermostatic loads in power systems.

{\it Link to the literature.} Interest in optimization problems on the space
of probability measures has increased strongly during the  recent years with the Monge-Kantorovitch optimal transport problem,
which, for two fixed Borel probability measures on $\R^d$, $\nu_1$ and $\nu_2$
consists in
determining a joint law, whose marginals are precisely  $\nu_1$ and $\nu_2$,
minimizing an expected given cost. Benamou and Brenier in \cite{benamou}
propose a dynamical formulation of
this problem: it consists in an optimal control problem where the aim is to minimize the integrated kinetic energy of a deterministic dynamical system over a given time horizon, in order to go from the initial law $\nu_1$ to $\nu_2$ as terminal law.
In \cite{ThieullenMikami}, the authors replace the deterministic dynamical system with a diffusion,
introducing the so called stochastic mass transportation problem. This consists in controlling the drift of the diffusion  to minimize,
over a given finite horizon, a mean integrated cost depending on the drift and the state of the process, while imposing the initial and final distribution of the diffusion.
Those authors formulate their problem as an optimization on a space of probability measures, for which they make use of convex duality techniques.
In \cite{TanTouzi}, the authors generalize these techniques
controlling the volatility as well. Those authors also propose a numerical scheme in order to approximate the dual formulation of their stochastic mass transport problem. In the same spirit as in \cite{ThieullenMikami}, in this paper, we formulate a stochastic optimal control problem 
as a minimization on the
space of probability measures.
However our approach is based, on the one hand, on an entropy correction and, on the other hand, on an alternating procedure.
	
  Similar ideas based on an entropy correction and an alternating procedure were introduced in the context of
  optimal transport, see e.g.
  \cite{IterativeBregman, LightspeedComputation}. In \cite{IterativeBregman},
  the authors are interested in the discrete optimal mass transport problem.
  To approach that problem, they
  introduce an entropic regularization
  which consists in minimizing a relative entropy $H(\gamma | \xi)$ over a subset $\shk := \shk_1 \cap \shk_2$ of joint probability measures on $\R^{d \times d}$, where $\xi$ is a reference probability measure on $\R^{d \times d}$, $\shk_1$ is a subset of $\shp(\R^{d \times d})$ with a given first marginal, while $\shk_2$ imposes the second marginal. The solution to this new problem is approximated by a sequence $\left(\gamma^{(n)}\right)_{n \ge 1}$, where
  $\gamma^{(n + 1)}$ is the entropic projection of $\gamma^{(n)}$ on the set $\shc_n$,
where $\shc_{2p} := \shk_2$ and $\shc_{2p + 1} := \shk_1$ for $p \in \N$.
  This means
  $\gamma_n  = \underset{\gamma \in \shc_n}{\argmin} ~ H(\gamma | \gamma^{(n)})$.

This type of methods and their generalization to continuous states distributions are commonly referred to as Sinkhorn algorithms and are widely used in optimal transport and related fields, such as the Schr\"odinger Bridge problem, see e.g. \cite{CattiauxLeonard, SchrodingerBridgeStochasticControl, LeonardSchrodinger, SchrodingerBridgeData} for detailed accounts. However, the approach we propose here is resolutely different and differs from these classical methods in two aspects. Firstly, our approach is based on a duplication of the optimization variables, and the entropy correction term we introduce is a {\it penalty term}, designed to impose equality on the duplicated variables.
Furthermore, our alternating procedure aims to sequentially optimize the penalized objective function, in the first and then second variable involved in the entropy penalty, whereas the Sinkhorn alternating projection algorithm is always driven to minimize the cross entropy term with respect to the first variable.

	Our reformulation offers both numerical and theoretical advantages. From a numerical point of view, our algorithm relies on two standard optimization sub-problems that are simpler than the original stochastic control problem and that can be tackled by specific numerical schemes.
        For example, one of the two sub-problems (called exponential twist problem)
has an explicit solution, which can be evaluated at each time step by parallel computations of conditional expectations.
        On the other hand, that sub-problem
        corresponds to a stochastic control problem with no constraints on the control, and can therefore efficiently
        be tackled by regression methods \cite{bender10, gobet10bis, gobet17} or deep learning methods as in \cite{MLPDEWarin, Germain22, HighDimensionalPDEDL, Hure20}. Hence, our algorithm constitutes a complementary approach to existing regression or machine learning techniques developed to solve stochastic control problems.
	
From a theoretical point of view, the entropy penalization approach offers new perspectives for reformulating complex stochastic control problems including, for example, constraints on the marginal laws of the controlled process, see e.g. Schr\"odinger bridge. This is the subject of
a the recent paper 
\cite{BORTargetLaw}.

The paper is organized as follows. After this Introduction, Section \ref{S2}
is devoted to the basic definitions and notations. In Section \ref{S3} we introduce an entropy 
penalized optimization problem approaching the original stochastic control problem, see in particular
Proposition \ref{prop:approximateControl}.
The subsequent Section \ref{sec:algo} proposes an alternating minimization procedure
to approximate the solution of the entropy penalized problem:
in particular Theorem \ref{th:convAlgo}  establishes that convergence.
Section \ref{sec:example} illustrates the interest of the approach on a specific application
to demand-side management       in power systems.
Some new perspectives of our method are sketched in Section \ref{S6}.

The paper is organized as follows. After this Introduction Section \ref{S2}
is devoted to the basic definitions and notations. In Section \ref{S3} we introduce an entropy 
penalized optimization problem approaching the original stochastic control problem, see in particular
Proposition \ref{prop:approximateControl}.
The subsequent Section \ref{sec:algo} proposes an alternating minimization procedure
to approximate the solution of the entropy penalized problem.
In particular Theorem \ref{th:convAlgo}  establishes that convergence.
Section \ref{sec:example} illustrates the interest of the approach on a specific application
 demand-side management       in power systems.
 Some new perspectives of our method are sketched in Section \ref{S6}.
 We conclude the paper with the Appendices,
 which contain the proof of most the technical intermediate results.

 \section{Notations and definitions}

\label{S2}
\setcounter{equation}{0}

In this section we introduce the basic notions and notations used throughout this document. In what follows, $T > 0$ will be a fixed time horizon.
\begin{itemize}
\item All vectors $x \in \R^d$ are column vectors. Given $x \in \R^d$, $|x|$ will denote its Euclidean norm.
	
\item Given a matrix $A \in \R^{d \times d}$, $\|A\| := \sqrt {Tr[AA^\top]}$ will denote its Frobenius norm.
  
	\item Given $\phi \in C^{1, 2}([0, T] \times \R^d, \R)$, $\partial_t\phi$, $\nabla_x \phi$ and $\nabla_x^2\phi$ will denote respectively the partial derivative of $\phi$ with respect to (w.r.t.) $t \in [0, T]$, its gradient and its Hessian matrix  w.r.t. $x \in \R^d$.
	\item Given any bounded function $\Phi: \R^d \rightarrow  \R$,
                    we denote by $\vert \Phi \vert_\infty$ its supremum.

	\item $\U$ will denote a closed subset of $\R^p$ for some $p \in \N^*$.
	
	\item For any topological spaces $E$ and $F,$ $\mathcal{B}(E)$ will denote the Borel $\sigma$-field of $E;$ $C(E, F)$ (resp. $\B(E, F)$) will denote the linear space of functions from $E$ to $F$ that are continuous (resp. Borel). $\Pma(E)$ will denote the Borel probability measures on $E$. Given $\P \in \Pma(E),$ $\E^\P$ will denote the expectation with respect to (w.r.t.) $\P.$
	
	\item Except if differently specified, $\Omega$ will denote the space of continuous functions from $[0, T]$ to $\R^d.$ For any $t \in [0, T] $ we denote by $X_t : \omega \in \Omega \mapsto \omega_t$ the coordinate mapping on $\Omega.$ We introduce the $\sigma$-field $\shf := \sigma(X_r, 0 \le r \le T)$. On the measurable space $(\Omega, \F),$ we introduce the \textit{canonical process} $X : \omega \in ([0, T] \times \Omega, \B([0, T])\otimes \F) \mapsto X_t(\omega) = \omega_t \in (\R^d, \B(\R^d))$.\\
          We endow $(\Omega, \F)$ with the right-continuous filtration
          $\F_t := \underset{t \le s \le T}{\bigcap}
          \sigma(X_r, 0 < r \le s),  \ t \in [0,T].$
          The filtered space
        $(\Omega, \F, (\F_t))$ will be called the \textit{canonical space} (for the sake of brevity, we denote $(\F_t)_{t \in [0, T]}$ by $(\F_t)$).
\item Given a function $\lambda: [0,T] \times \R^d \rightarrow \R$,
  $p \ge 1$
  and a Borel probability $\Q$ on $\Omega$, we will improperly say that
  $\lambda \in L^q(dt \otimes \Q)$ if the map
  $(t,\omega) \mapsto \lambda(t,X_t(\omega)) \in  L^q(dt \otimes \Q)$.

      \item Given a continuous local martingale
        $M$, $[ M ]$ will denote its \textit{quadratic variation}.
	
	\item Equality between stochastic processes are in the sense of \textit{indistinguishability}.
	
	\item Except if specified otherwise, all properties of processes (e.g. measurability, martingale) are with respect to the canonical filtration $(\shf_t)_{t \in [0, T]}$.
	
\end{itemize}
\begin{definition} (Relative entropy).
	\label{def:klDiv}
	Let $E$ be a topological space. Let $\P, \Q \in \Pma(E).$ The relative entropy $H(\Q | \P)$ between the measures $\P$ and $\Q$ is defined by
	\begin{equation}
		\label{eq:relativeEntropy}
		H(\Q | \P) :=
		\left\{
		\begin{aligned}
			& \E^{\Q}\left[\log \frac{d\Q}{d\P}\right] &\text{if $\Q \ll \P$}\\
			& + \infty &\text{otherwise,}
		\end{aligned}
		\right.
	\end{equation}
	with the convention $\log(0/0) = 0$.
\end{definition}
\begin{remark}
  \label{rmk:relativeEntropy}
  Let $E$ be a Polish space.
  The relative entropy $H: \shp(E) \times \shp(E)$ is
  non-negative and jointly convex,
    that is for all $\P_1, \P_2, \Q_1, \Q_2 \in \Pma(E)$, for all $\lambda \in [0, 1]$, $H(\lambda \Q_1 + (1 - \lambda) \Q_2 | \lambda \P_1 + (1 - \lambda)\P_2) \le \lambda H(\Q_1 | \P_1) + (1 - \lambda)H(\Q_2| \P_2)$. Moreover, $(\P, \Q) \mapsto H(\Q | \P)$ is lower semicontinuous with respect to the weak-star topology on $E^*$.
  We refer to \cite{DupuisEllisLargeDeviations} Lemma 1.4.3 for a proof of these properties.
\end{remark}

\begin{definition} \label{def:MinSeq} (Minimizing sequence, solution and
  $\epsilon$-solution). 
  Let $E$ be a generic set. Let $J : E \mapsto \R$ be a function. Let $J^* := \underset{x \in E}{\inf} J(x)$, which can be finite or not.
  \begin{enumerate}
  	\item A {\bf minimizing sequence} for $J$ is a sequence $(x_n)_{n \ge 0}$ of elements of $E$ such that $J(x_n) \underset{n \rightarrow + \infty}{\longrightarrow} J^*$.
  	
  	\item We will say that $x^* \in E$ is a solution to the optimization problem
  	\begin{equation} \label{eq:MinSeq}
  		\underset{x \in E}{\inf} J(x),
  	\end{equation}
  	if $J(x^*) = J^*.$
  	In this case,
  	$J^* = \underset{x \in E}{\min} J(x)$.
  	
  	\item For $\epsilon \ge 0$, we will say that $x^\epsilon \in E$ is an $\epsilon$-{\bf solution} to the
  	optimization Problem
  	\eqref{eq:MinSeq}
  	if $0 \le J(x^\epsilon) - J^* \le \epsilon$. We also say that $x^\epsilon$ is $\epsilon$-{\bf optimal} for the (optimization) Problem
  	\eqref{eq:MinSeq}.
  \end{enumerate}
   
  \end{definition}
  We remark that a $0$-solution is a solution of the optimization Problem
  \eqref{eq:MinSeq}.
  
\section{From the stochastic optimal control problem to a penalized optimization problem}

      \label{S3}

      \setcounter{equation}{0}

      In this section we consider a stochastic control problem that we reformulate in terms of an optimization problem on a space of probability measures.
      Later we propose a penalized version of that problem whose solutions are $\varepsilon$-optimal for the original problem.

      \subsection{The stochastic optimal control problem}
      We specify the assumptions and the formulation of the stochastic optimal control Problem \eqref{eq:controlProblemIntro}, stated in the
      Introduction. Let us first consider a drift $b \in \B([0, T] \times \R^d \times \U, \R^d)$ and a diffusion matrix $\sigma \in \B([0, T] \times \R^d, \R^{d \times d})$, following the assumptions below.
\begin{hyp} (Diffusion coefficients).
	\label{hyp:coefDiffusion}
	\begin{enumerate}
		\item $b$ is continuous in $(t, x, u)$.
		\label{item:bContinuous}
		
		\item There exists a constant $C_{b, \sigma} > 0$ such that, for all $(t, x) \in [0, T] \times \R^d, u \in \U, $
		\begin{equation}
			\label{eq:linearGrowthbSigma}
		|b(t, x, u)| + \|\sigma(t, x)\| \le C_{b, \sigma}(1 + |x|).
		\end{equation}
		\item There exists $c_\sigma > 0$ such that for all $(t, x) \in [0, T] \times \R^d, \xi \in \R^d$,
		\begin{equation}
			\label{eq:sigmaElliptic}
			\xi^\top \sigma\sigma^\top(t, x)\xi \ge c_\sigma |\xi|^2.
		\end{equation}
		\item For all $x \in \R^d$,
		$$
		\lim_{y \rightarrow x} \sup_{0 \le r \le T} \|\sigma(r, x) - \sigma(r, y)\| = 0.
		$$
	\end{enumerate}
\end{hyp}
Let us define the admissible set of probability measures $\Pma_\U$ for Problem \eqref{eq:controlProblemIntro}.
 \begin{definition}
 	\label{def:PU}
 	Let $\Pma_\U$ be the set of probability measures on $(\Omega, \F)$ such that, for all $\P \in \Pma_\U$, under $\P$, the canonical process decomposes as
 	\begin{equation} \label{eq:decompP}
 		X_t = x + \int_0^t b(r, X_r, \nu_r^\P)dr + M_t^\P,
 	\end{equation}
 	where $x \in \R^d$,
 $M^\P$ is a $(\P, \shf_t)$-local martingale such that
 $[M^\P] = \int_0^\cdot \sigma\sigma^\top(r, X_r)dr$, $\nu^\P$ is a progressively measurable process with values in $\U$. If in addition there exists $u^\P \in \B([0, T] \times \R^d, \U)$ such that
 $\nu_t^\P = u^\P(t, X_t)$ $dt \otimes d\P$-a.e, we will denote $\P \in \Pma_\U^{Markov}$.
\end{definition}

\begin{remark} \label{rmk:volatility}
  The admissible set of probability measures $\shp_\U$ for Problem \eqref{eq:controlProblemIntro}
  imposes an uncontrolled volatility. Indeed the approach developed in the present paper relies on Girsanov's
  theorem and can not be easily extended to the case of controlled volatility.
\end{remark}

\begin{remark}
	\label{rmk:stroockVaradhan}
        By classical stochastic calculus arguments, see e.g. Proposition 5.4.6 in \cite{karatshreve},
        we can state the following.
	If $\P \in \Pma_\U^{Markov}$ in the sense of Definition \ref{def:PU},
	then the following equivalent properties hold.
	\begin{enumerate}
		\item One has
                 \begin{equation} \label{eq:decompPBis}
			X_t = x + \int_0^t b(r, X_r, u^\P(r, X_r))dr + M_t^\P,
		\end{equation}
		with $x \in \R^d$, $[M^\P] = \int_0^\cdot \sigma\sigma^\top(r, X_r)dr$.
              \item $\P$ is solution of the \textit{martingale problem} (in the sense of Stroock and Varadhan in \cite{stroock}) associated with the initial condition $(0, x)$ and the operator $\shl_{u^\P}$ defined, for all
                bounded functions $\phi \in C_b^{1, 2}([0, T] \times \R^d, \R)$, $(t, y) \in [0, T] \times \R^d$, by
		\begin{equation}
			\label{eq:generatorU}
			\shl_{u^\P}\phi(t, y) = \partial_t \phi(t, y) + \langle \nabla_x\phi(t, y), b(t, y, u^\P(t, y))\rangle + \frac{1}{2}Tr[\sigma\sigma^\top(t, y)\nabla_x^2\phi(t, y)],
		\end{equation}
		with $\nu^\P := u^\P(\cdot, X_\cdot)$
                in \eqref{eq:decompP}.
		\item $\P$ is a solution (in law) of
		\begin{equation} \label{eq:decompPTer}
			X_t = x + \int_0^t b(r, X_r, u^\P(r, X_r))dr +
			\int_0^t \sigma(r, X_r) dW_r,
		\end{equation}
		for some suitable Brownian motion $W$.
	\end{enumerate}     
      \end{remark}
      
We will often make use of the following proposition.
\begin{prop}
	\label{prop:existencePu}
	Assume Hypothesis \ref{hyp:coefDiffusion} holds. Let $u \in \B([0, T] \times \R^d, \U)$. There exists a unique probability measure $\P^u \in \Pma_\U^{Markov}$ such that under $\P^u$ the canonical process decomposes as \eqref{eq:decompP}, with $\nu^\P_t = u(t, X_t) (= u^\P(t, X_t))$.
      \end{prop}
\begin{remark}
  \label{rmk:uniqDecomp}
  In particular, for a given $u: [0,T] \times \R^d \rightarrow \R$,
  the equation
  	\begin{equation} \label{eq:decompPRem}
 		X_t = x + \int_0^t b(r, X_r, u(r,X_r))dr + M_t^\P,
              \end{equation}
 	where $x \in \R^d$, $X$ is the canonical process, and
 $M^\P$ is a $(\P, \shf_t)$-local martingale such that
 $[M^\P] = \int_0^\cdot \sigma\sigma^\top(r, X_r)dr$,
admits a unique solution $\P $.

\end{remark}
\begin{proof}[Proof of Proposition \ref{prop:existencePu}]
  By Theorem 10.1.3 in \cite{stroock}, the martingale problem,
  associated with the initial condition $(0, x)$ and the operator $\shl_u$ defined by \eqref{eq:generatorU} with $u^\P = u$, admits a unique solution $\P^u$. The result is then a consequence of Remark \ref{rmk:stroockVaradhan}.
 \end{proof}

 Let then $f \in \B([0, T] \times \R^d \times \U, \R), ~g \in \B(\R^d, \R)$, referred to as the
\textit{running cost} and the \textit{terminal cost} respectively, and assume that the following holds.
\begin{hyp} (Cost functions).
	\label{hyp:costFunctionsControl}
	\begin{enumerate}
        \item The functions $f, g$ are positive and there
          exist $C_{f, g} > 0$, $p \ge 1$, such that, for all $(t, x, u)
		\in [0, T] \times \R^d \times \U$,
		\begin{equation}
			\label{eq:polyGrowthfg}
			f(t, x, u) + g(x) \le C_{f, g}(1 + |x|^p).
		\end{equation}
	
		\item $f$ and $g$ are continuous in $(t, x, u) \in [0, T] \times \R^d \times \U$ and $x \in \R^d$ respectively.
		\label{item:fgContinuous}
		
              \item
Let $p \ge 1$ mentioned at item $1.$
               There exist constants $p' > p$
                  and $C' > 0$
                  such that $|u|^{p'} \le C'(1 + f(t, x, u))$, for all $(t, x, u) \in [0, T] \times \R^d \times \U$.
\end{enumerate}
\end{hyp}
For any $(t, x) \in [0, T] \times \R^d$, we introduce the set
\begin{equation}
	\label{eq:Ktx}
	K(t, x) := \left\{\vphantom{e^{||}} (b(t, x, u), z)~\middle | ~u \in \U,~z \ge f(t, x, u)\right\}.
\end{equation}

\begin{remark}
  \label{rmk:KtxClosed}
  \begin{enumerate}
  \item Item $3.$ of Hypothesis \ref{hyp:costFunctionsControl} is of course verified
    if $\U$ is bounded. 
    \item Whenever $\U$ is unbounded, the same hypothesis implies that
      $ \vert f(t,x,u) \vert \rightarrow +\infty$ if
      $\vert u \vert$ goes to infinity.
 \item
      Under Hypothesis \ref{hyp:costFunctionsControl}, the set $K(t, x)$ is closed.
      Let indeed $((y_n, z_n))_{n \ge 0}$ be a sequence of elements of $K(t, x)$, which converges toward $(y^*, z^*) \in \R^{d + 1}$. Let $(u_n)_{n \ge 0}$ be a sequence of elements of $\U$ such that for all $n \in \N$,
	\begin{equation}
		\label{eq:ynzn}
		y_n = b(t, x, u_n) \quad \text{and} \quad z_n \ge f(t, x, u_n).
	\end{equation}
	Then, by item 3. of Hypothesis \ref{hyp:costFunctionsControl},
	$$
	\sup_{n \in \N}|u_n|^{p'} \le C'\left(1 + \sup_{n \in \N}f(t, x, u_n)\right) \le C'\left(1 + \sup_{n \in \N}z_n \right).
	$$
	Since $(z_n)_{n \in \N}$ converges, it is bounded and the previous inequality implies that $(u_n)_{n \in \N}$ is also bounded. Up to a subsequence, we can thus assume that $(u_n)_{n \in \N}$ converges towards a limit $u^* \in \U$ (recall that $\U$ is closed). Since $b(t, x, \cdot)$ and $f(t, x, \cdot)$ are continuous, letting $n \rightarrow + \infty$ in \eqref{eq:ynzn}, yields $y^* = b(t, x, u^*)$ and $z^* \ge f(t, x, u^*)$. Hence $(y^*, z^*) \in K(t, x)$, and $K(t, x)$ is closed.
\end{enumerate}
      \end{remark}

We will require the following convexity assumption.
\begin{hyp}
  \label{hyp:convexSet}(Convex).
 	For all $(t, x) \in [0, T] \times \R^d$, the set $K(t, x)$ is convex.
\end{hyp}

\begin{remark}
	\begin{enumerate}
        \item If $\U$ is convex, $b$ is linear w.r.t. to $u$ and $f$ is convex w.r.t. $u$, then Hypothesis \ref{hyp:convexSet} holds.
		\item Hypothesis \ref{hyp:convexSet} is a classical convexity assumption when one wants to prove existence of optimal Markovian control to Problem \eqref{eq:controlProblemIntro} in the weak sense, by using compactness arguments, see e.g. \cite{HaussmannCompactification, HaussmanLepeltierOptimal, ElKarouiCompactification, LackerMarkovianControl}.
	\end{enumerate}
\end{remark}

We conclude this section providing a moment estimate, see  e.g.
Corollary 5.12 of Chapter 2 in \cite{krylov}.
\begin{lemma}
  \label{lemma:classicalEstimates}
  Let $b, \sigma$ fulfill Hypothesis \ref{hyp:coefDiffusion}
  and $q \ge 1$.
  Then there is a constant $C(q)$, which depends  on $T$ and $C_{b, \sigma}$ (and $q$), such that the following holds.
  
	Let $(\Omega, \F, (\F_t), \P)$ be a filtered probability space. Let $\nu : [0, T] \times \Omega \rightarrow \U$ be an $(\F_t)$-progressively measurable process. Let $X$ be an Itô process on $(\Omega, \F, \P)$, which decomposes as
	$$
	X_t = x + \int_0^t b(r, X_r, \nu_r)dr + M_t^\P,
	$$
	where $M^\P$ is a $\P$-local martingale such that $[M^\P] = \int_0^\cdot \sigma\sigma^\top(r, X_r)dr$.
        Then we have
	$$
	\E^{\P}\left[\sup_{0 \le t \le T} |X_t|^q\right] \le C(q).
	$$
\end{lemma}
Under Hypotheses \ref{hyp:coefDiffusion} and \ref{hyp:costFunctionsControl},
the function $J$ introduced in \eqref{eq:controlProblemIntro} is well-defined
on the set $\Pma_\U$, set out in Definition \ref{def:PU}.
Indeed,
by the moment estimate given by Lemma \ref{lemma:classicalEstimates} one has 
	$$
	\E^{\P}\left[\int_0^T f(r, X_r, \nu_r^\P)dr + g(X_T)\right] < + \infty,
	$$
	for all $\P \in \Pma_{\U}$.

\subsection{The penalized optimization problem}

As mentioned in the Introduction, we reformulate Problem \eqref{eq:controlProblemIntro}, by doubling the decision variables and by adding a relative entropy term, in the objective function. The modified penalized  Problem is precisely
\eqref{eq:penalizedProblemIntro},
where $\mathcal{A}$ is the subset of elements $(\P, \Q) \in \Pma(\Omega)^2$ defined below.
\begin{definition}
	\label{def:A}
 $\mathcal{A}$ will denote the set of probability measure couples $(\P, \Q) \in \Pma(\Omega)^2$, such that
	\begin{enumerate}
		\item $\P \in \mathcal{P}_{\U}$,
		\item $H(\Q | \P) < + \infty.$
	\end{enumerate}
\end{definition}
In the perspective of solving the penalized optimization Problem \eqref{eq:penalizedProblemIntro}, we will introduce in Sections \ref{sec:pointwiseMinimization} and \ref{sec:markovDrift} two subproblems.
 The interest of the penalized formulation \eqref{eq:penalizedProblemIntro} relies on the fact that each of the subproblems, $\underset{\Q \in \shp(\Omega)}{\inf} \shj_{\epsilon}(\Q, \P)$ and $\underset{\P \in \shp_{\U}}{\inf} \shj_{\epsilon}(\Q, \P)$, can be solved by classical techniques described in the literature:
 those resolutions will constitute the two steps of our alternating minimization algorithm.

 The first subproblem, considered in Section \ref{sec:markovDrift}, is a minimization on $\Q$, the probability $\P$ being fixed, and it is related to a variational representation formula, whose solution is denominated exponential twist, see e.g. \cite{DupuisEllisLargeDeviations}.
In particular  the following result will intervene.
\begin{prop}
	\label{prop:markovExistenceMinimizer}
	Let $\vphi : \Omega \rightarrow \R$ be a Borel function and $\P \in \Pma(\Omega)$. Assume that $\vphi$ is bounded below. Then
	\begin{equation}
		\label{eq:klOpti}
	\inf_{\Q \in \Pma(\Omega)} \E^{\Q}[\vphi(X)] + \frac{1}{\epsilon} H(\Q | \P) = - \frac{1}{\epsilon}\log \E^{\P}\left[\exp(-\epsilon\vphi(X))\right].
      \end{equation}
      Moreover the problem \eqref{eq:klOpti} admits a unique solution
      (minimizer)
      $\Q^* \in \Pma(\Omega)$, given by
	$$
	d\Q^* = \frac{\exp(-\epsilon\vphi(X))}{\E^{\P}[\exp(-\epsilon\vphi(X))]}d\P.
	$$
\end{prop}
\begin{proof}
	The random variable $\vphi(X)$ is bounded below, hence satisfies condition $(FE)$ of \cite{EntropyWeighted}. The statement then follows from Proposition 2.5 in \cite{EntropyWeighted}.
\end{proof}
Applying Proposition \ref{prop:markovExistenceMinimizer} to our framework for $\P \in \Pma_\U$ and $\vphi(X) := \int_0^T f(r, X_r, \nu_r^\P)dr + g(X_T)$, we get that, under Hypothesis \ref{hyp:costFunctionsControl}, the subproblem $\underset{\Q \in \Pma(\Omega)}{\inf} \shj_\epsilon(\Q, \P)$ admits a unique solution $\Q^*$
given by
\begin{equation}
	\label{eq:minimizerGeneralOptimizationProblem}
	d\Q^* = \frac{\exp\left(-\epsilon\int_0^T f(r, X_r, \nu_r^\P)dr - \epsilon g(X_T)\right)}{\E^{\P}\left[\exp\left(-\epsilon\int_0^T  f(r, X_r, \nu_r^\P)dr - \epsilon g(X_T)\right)\right]}d\P,
\end{equation}
and that its optimal value is
\begin{equation}
	\label{eq:optimalValueMinQ}
	\shj_\epsilon(\Q^*, \P) = - \frac{1}{\epsilon}\log \E^{\P}\left[\exp\left(-\epsilon\int_0^T  f(r, X_r, \nu_r^\P)dr - \epsilon g(X_T)\right)\right].
\end{equation}
This subproblem is further analyzed in Section \ref{sec:markovDrift}. In particular Proposition \ref{prop:markovianDrift} allows to identify $\Q^*$ as the law of a semimartingale with Markovian drift.

\begin{remark} \label{rmk:OptVal}
  Suppose the validity of Hypothesis \ref{hyp:costFunctionsControl}.
  Then $ \|d\Q^*/d\P\|_\infty < + \infty.$
  \end{remark}

Let us discuss now about the second problem, i.e.
the subproblem  $\underset{\P \in \Pma_{\U}}{\inf} \shj_\epsilon(\Q, \P)$,
which will be the object of
Section \ref{sec:pointwiseMinimization}. This is a minimization on $\P$, the probability $\Q$ remaining unchanged.
The optimal solution of this subproblem is a probability measure $\P^u$, in the sense of Proposition \ref{prop:existencePu}, where the function $u \in \B([0, T] \times \R^d, \U)$, is provided by a pointwise minimization.

The next theorem proves that the penalized Problem \eqref{eq:penalizedProblemIntro} has  a Markovian solution.

\begin{theorem}
  \label{th:existenceSolutionRegProb}
  Assume Hypotheses \ref{hyp:coefDiffusion}, \ref{hyp:costFunctionsControl} and \ref{hyp:convexSet} hold.
  Then the penalized Problem \eqref{eq:penalizedProblemIntro} has a {\it solution} $(\P^*_\epsilon, \Q^*_\epsilon) \in \sha$, in the sense that $\shj^*_\epsilon = \shj(\Q^*_\epsilon, \P^*_\epsilon).$ Moreover, under $\P^*_\epsilon$, the canonical process is a Markov process and $\nu^{\P_\epsilon^*}$, related to $\P_\epsilon^*$ by Definition \ref{def:PU}, is such that $\nu^{\P^*_\epsilon}_t  (= u^{\P_\epsilon^*}_\epsilon(t, X_t)) = u^*_\epsilon(t, X_t)$, for some function
  $u^*_\epsilon \in \shb([0, T] \times \R^d, \U)$ and we also have
	\begin{equation}
		\label{eq:optimalSolutionDensity}
		d\Q_{\epsilon}^* = \frac{\exp\left(-\epsilon \int_0^T f(r, X_r, u_\epsilon^*(r, X_r))dr - \epsilon g(X_T)\right)}{\E^{\Q_\epsilon^*}\left[\exp\left(-\epsilon \int_0^T f(r, X_r, u_\epsilon^*(r, X_r))dr - \epsilon g(X_T)\right)\right]}d\P_\epsilon^*.
	\end{equation}
\end{theorem}

\noindent
The proof of this result relies on several technical lemmas. For the convenience of the reader, it is postponed to
Appendix \ref{app:proofThEx}.

The following proposition justifies the use of the penalized Problem \eqref{eq:penalizedProblemIntro} to
approximately solve the initial stochastic optimal control Problem \eqref{eq:controlProblemIntro}.
Indeed, the next result states that one can build an approximate solution of Problem
\eqref{eq:controlProblemIntro},
based on an approximate solution of Problem \eqref{eq:penalizedProblemIntro}

\begin{prop}
  \label{prop:approximateControl} We suppose the validity of Hypothesis \ref{hyp:coefDiffusion} and item $1.$ of Hypothesis
  \ref{hyp:costFunctionsControl}. Let $\epsilon >0, \epsilon' \ge 0$ and let   $\P_\epsilon^{\epsilon'}$ be the first component of an $\epsilon'$-solution of Problem \eqref{eq:penalizedProblemIntro}, in the sense of Definition \ref{def:MinSeq}, with $E = \sha.$
  We set $Y^{\epsilon'}_\epsilon := \int_0^T f(r, X_r, \nu^{\epsilon'}_\epsilon) dr
  + g(X_T),$ where $\nu^{\epsilon'}_\epsilon$ corresponds to the $\nu^{\P_\epsilon^{\epsilon'}}$, appearing in decomposition \eqref{eq:decompP}. Then the following holds.
	\begin{enumerate}
				\item There is a constant $C^*$ depending only on $C_{b, \sigma}, C_{f, g}, p, d, T$ of Hypothesis \ref{hyp:costFunctionsControl} $1.$ such that\\ $\max(\E[Y_\epsilon^{\epsilon'}], Var^{\P^{\epsilon'}_\epsilon}[Y_\epsilon^{\epsilon'}]) \le C^*,$ where  $Var^{\P^{\epsilon'}_\epsilon}[Y^{\epsilon'}_\epsilon]$ denotes the variance of $Y^{\epsilon'}_\epsilon$ under ${\P^{\epsilon'}_\epsilon}$.
	
		\item We have
		$$
		0 \le J(\P_\epsilon^{\epsilon'}) - J^* \le \epsilon e^{\epsilon \E[Y_\epsilon^{\epsilon'}]} Var^{\P^{\epsilon'}_\epsilon}[Y^{\epsilon'}_\epsilon]  + \epsilon',
		$$
	\end{enumerate}
	where we recall that $J$ and $J^*$ are defined in \eqref{eq:controlProblemIntro}.
\end{prop}
\begin{remark} \label{remark:approximateControl}
	\begin{enumerate}
		\item Let $(\P_\epsilon^*, \Q_\epsilon^*)$ be a solution of Problem \eqref{eq:penalizedProblemIntro} given by Theorem \ref{th:existenceSolutionRegProb}. Applying item 2. of Proposition \ref{prop:approximateControl} with $\epsilon' = 0$ implies that $\P_\epsilon^*$ is an $\epsilon e^{\epsilon \E[Y_\epsilon^{0}]} Var^{\P^*_\epsilon}[Y^{0}_\epsilon]$-solution of the original Problem \eqref{eq:controlProblemIntro}.
		\item By definition of infimum, for $\epsilon' > 0$, the existence of an $\epsilon'$-solution is always guaranteed without any convex assumption on the running cost $f$ w.r.t. the control variable.
		\item In the sequel, assuming Hypotheses \ref{hyp:costFunctionsControl} and \ref{hyp:convexSet} to hold, we will propose an algorithm providing a sequence of $\epsilon'_n$-solutions of the penalized Problem \eqref{eq:penalizedProblemIntro}, where $\epsilon'_n \rightarrow 0$ as $n \rightarrow + \infty$. This will also provide a sequence of
                  $(\epsilon e^{\epsilon \E[Y_\epsilon^{\epsilon_n'}]} Var^{\P^{\epsilon'_n}_\epsilon}[Y^{\epsilon'_n}_\epsilon] + \epsilon'_n)$-solutions to the original Problem \eqref{eq:controlProblemIntro} (with a fixed $\epsilon > 0$).
                \end{enumerate}
                
      \end{remark}

\begin{proof}[Proof of Proposition \ref{prop:approximateControl}.]
	We first prove item $1.$
Let $(\P_{\epsilon}^{\epsilon'},\Q_\epsilon^{\epsilon'})$ be an $\epsilon'$-solution of Problem  
	\eqref{eq:penalizedProblemIntro}.
  By Hypothesis \ref{hyp:costFunctionsControl}, for all $\epsilon > 0$,
         on the one hand one has
        $$
        \E[Y_\epsilon^{\epsilon'}] \le C_{f, g}(T + 1)\left(1 + \E^{\P_\epsilon^{\epsilon'}}\left[\underset{0 \le t \le T}{\sup} |X_t|^p\right]\right),
        $$
        and on the other hand,
	$$
	Var^{\P_\epsilon^{\epsilon'}}[Y^{\epsilon'}_\epsilon] \le \E^{\P_\epsilon^{\epsilon'}}\left[(Y^{\epsilon'}_\epsilon )^2\right] \le 4C_{f, g}^2(T^2 + 1)\left(1 + \E^{\P_\epsilon^*}\left[\sup_{0 \le t \le T}|X_t|^{2p}\right]\right).
	$$  
	Combining these inequalities with Lemma \ref{lemma:classicalEstimates} implies the  existence of a constant $C^*$, depending only on $C_{b, \sigma}, C_{f, g}, p, d, T$,
        such that $\max(\E[Y_\epsilon^{\epsilon'}], Var^{\P_\epsilon^{\epsilon'}}[Y_\epsilon^{\epsilon'}]) \le C^*$, which is the statement of item $1.$

 We go on with the proof of item $2.$ First, a direct application of Lemma \ref{lemma:squareIntVar} with $\eta = Y_{\epsilon}^{\epsilon'}$, yields    
	\begin{equation}
		\label{eq:upperBoundVar}
		0 \le \E^{\P_\epsilon^{\epsilon'}}[Y^{\epsilon'}_\epsilon ] - \left(-\frac{1}{\epsilon}\log \E^{\P_\epsilon^{\epsilon'}}[\exp(-\epsilon Y^{\epsilon'}_\epsilon )]\right) \le \frac{\epsilon}{2}Var^{\P_\epsilon^{\epsilon'}}[Y^{\epsilon'}_\epsilon ]
                e^{\epsilon \E^{\P_\epsilon^{\epsilon'}}[Y^{\epsilon'}_\epsilon]}
                .
              \end{equation}
             Let then $\tilde \Q$ be the solution of $\underset{\Q \in \Pma(\Omega)}{\inf} \shj_\epsilon(\Q, \P_\epsilon^{\epsilon'})$, given by
              \eqref{eq:minimizerGeneralOptimizationProblem}, replacing $\P$ with $\P_\epsilon^{\epsilon'}$.
               Consequently, by \eqref{eq:optimalValueMinQ} $\shj_\epsilon(\tilde \Q, \P_\epsilon^{\epsilon'}) = -\frac{1}{\epsilon}\log \E^{\P_\epsilon^{\epsilon'}}[\exp(-\epsilon Y^{\epsilon'}_\epsilon )]$;  replacing the right-hand side of previous expression
              with $\shj_\epsilon(\tilde \Q, \P_\epsilon^{\epsilon'})$
              in \eqref{eq:upperBoundVar}, we get
	\begin{equation}
		\label{eq:interEpsilonOpti1}
	0 \le \E^{\P_\epsilon^{\epsilon'}}[Y^{\epsilon'}_\epsilon ] - \shj_\epsilon(\tilde \Q, \P_\epsilon^{\epsilon'}) \le \epsilon e^{\epsilon \E[Y_\epsilon^{\epsilon'}]} Var^{\P^{\epsilon'}_\epsilon}[Y^{\epsilon'}_\epsilon].
	\end{equation}
	Let $\Q_\epsilon^{\epsilon'}$ be the second component of the $\epsilon'$-solution of Problem
	\eqref{eq:penalizedProblemIntro}, mentioned in the statement of
        the current proposition.
        Observe that $\shj_\epsilon(\tilde \Q, \P_\epsilon^{\epsilon'}) \le \shj_\epsilon(\Q_\epsilon^{\epsilon'}, \P_\epsilon^{\epsilon'}) \le \shj^*_\epsilon + \epsilon'$. Besides,  Problem \eqref{eq:controlProblemIntro}
is equivalent to Problem 	\eqref{eq:penalizedProblemIntro}, under the constraint
$\Q = \P$, therefore $\shj_\epsilon^* \le J^*$.
        Then
	\begin{equation}
		\label{eq:interEpsilonOpti2}
	\shj_\epsilon(\tilde \Q, \P_\epsilon^{\epsilon'}) - J^* \le \shj_\epsilon^* + \epsilon' - J^* \le \epsilon'.
	\end{equation}
	Using \eqref{eq:interEpsilonOpti1} and \eqref{eq:interEpsilonOpti2} finally yields
	\begin{equation}
		\label{eq:ineqVariance}
		0 \le J(\P_\epsilon^{\epsilon'}) - J^* = \E^{\P_\epsilon^{\epsilon'}}[Y^{\epsilon'}_\epsilon ] - \shj_\epsilon(\tilde \Q, \P_\epsilon^{\epsilon'}) + \shj_\epsilon(\tilde \Q, \P_\epsilon^{\epsilon'}) - J^* \le \epsilon e^{\epsilon \E[Y_\epsilon^{\epsilon'}]} Var^{\P^{\epsilon'}_\epsilon}[Y^{\epsilon'}_\epsilon] + \epsilon'.
	\end{equation}
        This concludes the proof of item $2.$
      \end{proof}

\section{Alternating minimization procedure} \label{sec:algo}

\setcounter{equation}{0}
From now on, $\epsilon$ will be implicit in the cost function $\shj_\epsilon$ to alleviate notations.
In this section we will assume Hypotheses \ref{hyp:coefDiffusion}, \ref{hyp:costFunctionsControl} and \ref{hyp:convexSet}. 
We present an alternating procedure for solving 
the penalized Problem \eqref{eq:penalizedProblemIntro}. Let $(\P_0, \Q_0) \in \mathcal{A}$. We will define a sequence $(\P_k, \Q_k)_{k \ge 0}$ satisfying the alternating minimization procedure
\begin{equation}
	\label{eq:alternateMinimizationProcedure}
	\Q_{k + 1} = \underset{\Q \in \Pma(\Omega)}{\argmin}~\shj(\Q, \P_{k}), \quad \P_{k + 1} \in \underset{\P \in \Pma_{\U}}{\argmin}~\shj(\Q_{k + 1}, \P).
\end{equation}

\subsection{Convergence result}
\label{sec:convergenceResult}

The convergence of alternating minimization algorithms has been extensively studied in particular in Euclidean spaces.
In general the proof of convergence results requires joint convexity and smoothness properties of the objective function,
see \cite{BeckFirstOrder}. The major difficulty in our case is that the convexity only holds w.r.t $\Q$ (in fact the set $\Pma_\U$ is not even convex).
To prove the convergence we need to rely on techniques which exploit the properties of the entropic penalization. Let us first assume that the initial probability measure $\P_0 \in \Pma_\U$ is Markovian in the following sense.
\begin{hyp}
	\label{hyp:initialPoint}
	$\P_0 \in \Pma_\U^{Markov}$, see Definition \ref{def:PU}.
          In particular, there exists $u^0 (= u^{\P_0}) \in \B([0, T] \times \R^d, \U)$,
           such that $\P_0 = \P^{u^{0}}$.
      \end{hyp}
      Let $\sigma^{-1}$ be the generalized
                      right-inverse of $\sigma$,
                      i.e.
                $\sigma^\top (\sigma \sigma^\top)^{-1}$.
For a fixed Borel function $\beta: [0,T]\times \R^d \rightarrow \R^d$,
we set 
\begin{equation}
	\label{eq:fBeta}
	F_{\beta} : (t, x, u) \in [0, T] \times \R^d \times \U \mapsto f(t, x, u) +
        \frac{1}{2\epsilon}|\sigma^{-1}(t, x)(\beta(t, x) - b(t, x, u))|^2.
      \end{equation}
Given $(t, x) \in [0, T] \times \R^d$, we furthermore introduce the function
\begin{equation}
	\label{eq:barFbeta}
	(y, z) \in \R^d \times \R  \mapsto \bar F^{t, x}_\beta(y, z) := z + \frac{1}{2 \epsilon}|\sigma^{-1}(t, x)(\beta(t, x) - y)|^2.
\end{equation}
We state a lemma which will be used several times in this article and it will be proved in the Appendix \ref{app:proofEstimate}.
 \begin{lemma}\label{lemma:infKtx}
	Let $(t, x) \in [0, T] \times \R^d$. The following holds.
	\begin{enumerate}
        \item The function $\bar F^{t, x}_\beta$, defined by \eqref{eq:barFbeta}, restricted to $K(t,x)$,
          has a unique minimum  $(y^*, z^*)$, which verifies
          \begin{equation}
          	\label{eq:firstOrder}
          	z - z^* + \frac{1}{\epsilon}\langle (\sigma^{-1})^\top\sigma^{-1}(t, x) (y^* - \beta(t, x)), y - y^* \rangle \ge 0 \quad  \forall (y, z) \in K(t, x).
          \end{equation}

        \item
          \begin{enumerate}
            \item
              Let $u^* \in \U$ such that $y^* = b(t, x, u^*)$ and $z^* \ge f(t, x, u^*)$. Then
              \begin{equation} \label{eq;ustarA}
                u^* \in \underset{a \in \U}{\argmin}~F_\beta(t, x, a),
\end{equation}
                where
$F_\beta$ was defined in  \eqref{eq:fBeta}.
        \item Conversely, if $u^* \in \underset{a \in \U}{\argmin}~F_\beta(t, x, a)$, then
            \begin{equation} \label{eq:ustarB}
              (y^*,z^*):=\left(b(t, x, u^*), f(t, x, u^*)\right) \in \underset{(y, z) \in K(t, x)}{\argmin}~\bar F^{t, x}_\beta(y, z).
              \end{equation}
\end{enumerate}
        \end{enumerate} 
      \end{lemma}

\begin{remark}
	\label{rmk:fBeta}
	Let $\hat \beta: [0, T] \times \Omega \rightarrow \R^d$ be a path-dependent function. We extend the definition \eqref{eq:fBeta} of $F_\beta$, by setting
	\begin{equation} \label{eq:rmkFbeta}
	\hat F_{\hat \beta}(t,X ,u) := f(t, X_t, u) + \frac{1}{2\epsilon}|\sigma^{-1}(t, X_t)(\hat \beta(t, X) - b(t, X_t, u))|^2.
	\end{equation}
    We remark that, whenever
    $\hat u : [0, T] \times \Omega \rightarrow \U$, $\hat u(t,X) =  u(t,X_t)$, and $\hat \beta(t,X) = \beta(t,X_t)$, we have
    $$
    F_\beta(t,X_t,u(t,X_t)) = {\hat F}_{\hat \beta}(t,X,\hat u(t,X)).
    $$
\end{remark}
Let $\P_0$
satisfying Hypothesis \ref{hyp:initialPoint}. We set $\Q_0 = \P_0$. We build a sequence $(\P_k, \Q_k)_{k \ge 0}$ of elements of $\sha$, according to the following procedure. Let $k \ge 1$.
\begin{itemize}
	\item Let
	\begin{equation}
		\label{eq:defQk}
		d\Q_{k + 1} := \frac{\exp\left(-\epsilon\int_0^T f(r, X_r, u^{k}(r, X_r))dr - \epsilon g(X_T)\right)}{\E^{\P_{k}}\left[\exp\left(-\epsilon\int_0^T f(r, X_r, u^{k}(r, X_r))dr - \epsilon g(X_T)\right)\right]}d\P_{k},
              \end{equation}
              where $u^k = u^{\P_k}$. 
	By Proposition \ref{prop:markovianDrift}
        below there exists a measurable
	function $\beta^{k + 1} : [0, T] \times \R^d \rightarrow \R^d$ such that,
        under $\Q_{k + 1}$, the canonical process decomposes as
	\begin{equation}
		\label{eq:decompQn}
		X_t = x + \int_0^t \beta^{k + 1}(r, X_r)dr + M_t^{\Q_{k + 1}},
	\end{equation}
	where $M^{\Q_{k + 1}}$ is a local $\Q_{k + 1}$-martingale such that $[ M^{\Q_{k + 1}}] = \int_0^{\cdot} \sigma\sigma^\top(r, X_r)dr$.
	
      \item 
        By Proposition \ref{prop:pointwiseMinimization} there exists a Borel function $u^{k+1}: [0,T] \times \R^d \rightarrow \U$,
        such that
        \begin{equation} \label{eq:defPk}
          (t, x) \mapsto u^{k + 1}(t, x) \in \argmin_{a \in \U} F_{\beta^{k + 1}}(t, x, a),
          \end{equation}
       where $F_{\beta^{k + 1}}$ is given by \eqref{eq:fBeta}.
We define $\P_{k + 1} := \P^{u^{k + 1}}$ according to 
   Proposition \ref{prop:existencePu}, so that  
under $\P_{k + 1}$ the canonical process decomposes as
	\begin{equation}
		\label{eq:decompPn}
		X_t = x + \int_0^t b(r, X_r, u^{k + 1}(r, X_r))dr + M_t^{\P_{k + 1}},
	\end{equation}
	where $M^{\P_{k + 1}}$ is a local $\P_{k +1}$-martingale such that $[M^{\P_{k + 1}}] = \int_0^{\cdot} \sigma\sigma^\top(r, X_r)dr$. In particular $u^{k+1}= u^{\P_{k +1}}$. 
      \end{itemize}
	The proof of the lemma below is a direct application of Proposition \ref{prop:markovianDrift} for item $1.$ Item $2.$ follows from Propositions \ref{prop:pointwiseMinimization} and \ref{prop:existencePu}.
\begin{lemma}
	\label{lemma:sequenceAlternateDirection}
	Let $\P_0 = \Q_0\in \Pma_\U$ satisfying Hypothesis \ref{hyp:initialPoint}. Let $(\P_k, \Q_k)_{k\ge 0}$ be given by the recursion \eqref{eq:defQk} and
      just before  \eqref{eq:decompPn}. The following holds for $k \ge 0$.
	\begin{enumerate}

		\item $\Q_{k + 1} = \underset{\Q \in \Pma(\Omega)}{\argmin}~\shj(\Q, \P_{k})$, and
                  $$
                  \shj(\Q_{k + 1}, \P_{k}) = -\frac{1}{\epsilon}\log \E^{\P_{k}}\left[\exp\left(-\epsilon\int_0^T f(r, X_r, u^{k}(r, X_r))dr - \epsilon g(X_T)\right)\right],
                  $$
where $u^k = u^{\P_k}$.
                  Moreover, under $\Q_{k + 1}$ the canonical process is a Markov process and $\beta^{k + 1} \in L^q(dt \otimes \Q^{k + 1})$ for all $1 < q < 2$.
				
		\item $\P_{k + 1} \in \underset{\P \in \Pma_{\U}}{\argmin}~\shj(\Q_{k + 1}, \P)$.
      
                \end{enumerate}
                
\end{lemma}

	\begin{remark} \label{rmk:RegMark}
		Let $u^0 \in \shb([0,T] \times \R^d, \U)$.
		We emphasize that the sequence $(u^k)_{k \ge 1}$ of Markovian controls produced by the
                alternating minimization procedure in \eqref{eq:defQk}-\eqref{eq:defPk}, is independent of
                the initial law $\delta_x$. 

                Indeed the function $\beta = \beta^{k+1}$ appearing
                at each step, is provided by
                Proposition \ref{prop:markovianDrift}
 and it  is of the type
  $\beta(t,x) = b(t,x,u(t,x)) + \lambda(t,x)$,
  according to the proof. That Proposition \ref{prop:markovianDrift}
  is a consequence of Corollary 6.12 in
                \cite{BORMarkov2023}, which follows from
                Corollary 6.8 of the same paper.
                We recall that,
                the aforementioned function $\lambda$
was of the form
                $ \frac{\Gamma^v(id)}{v}$,
                where, $v$ was defined in (5.4),
 $\phi \mapsto \Gamma^v(\phi)$
                was a map introduced in Proposition 5.11 of
                \cite{BORMarkov2023}.
                
                The function  $v$, and (taking into account of the step (c) of the proof of that Proposition) $\Gamma^v$  only depend on 
the ''Regularly Markovian property'',
                (Hypothesis 6.2),
                which is always fulfilled in our case,
                see Remark 6.7 of the same paper.
That ''Regularly Markovian property''
is in fact only concerned by the dynamics of $\P$ and not on the initial condition.

                 Besides, the minimization \eqref{eq:pointwiseMinimization} in Proposition \ref{prop:pointwiseMinimization} does not depend on the initial condition provided that $\beta$ is also independent from it.
               \end{remark}

The main result of this section is given below.

\begin{theorem}
	\label{th:convAlgo}
        Let $\epsilon >0$ and recall that $\shj =\shj_\epsilon$ and
        $\shj^* = \shj ^*_\epsilon$  defined in \eqref{eq:penalizedProblemIntro},
        i.e. 
	$$\shj^* = \underset{(\P, \Q) \in \mathcal{A}}{\inf} \shj(\Q, \P).$$
  Let $\P_0 = \Q_0$
 satisfying Hypothesis \ref{hyp:initialPoint}.
	Assume also that Hypotheses \ref{hyp:coefDiffusion}, \ref{hyp:costFunctionsControl} and \ref{hyp:convexSet} hold.
	Let $(\P_k, \Q_k)_{k\ge 0}$ be given by the recursion \eqref{eq:defQk}
        and just before \eqref{eq:decompPn}.
	Then $\shj(\Q_k, \P_k) \underset{k \rightarrow + \infty}{\searrow} \shj^*$. 

        Moreover, there exists a constant $C > 0$,  which only depends on  $c_\sigma$, $C_{b, \sigma}$, $C_{f, g}$, $d$ and $T$
(and not on $k, \epsilon $), 
          such that $0 \le \shj(\Q_k, \P_k) - \shj^* \le \frac{C}{k}\left(1 + \frac{1}{\epsilon}\right)$, for all $k \ge 1$.
\end{theorem}
Theorem \ref{th:convAlgo} and Proposition \ref{prop:approximateControl} yield  Corollary \ref{coro:estimationError} below.
\begin{coro}
	\label{coro:estimationError}
        Let $\epsilon >0$ and $J^*$
        as defined in
        \eqref{eq:controlProblemIntro}.
	Let $\P_0 = \Q_0\in \Pma_\U$ satisfying Hypothesis \ref{hyp:initialPoint}. Let $(\P_k, \Q_k)_{k\ge 0}$ be given by the recursion \eqref{eq:defQk} and
      just before  \eqref{eq:decompPn}.

Under the assumptions of Theorem \ref{th:convAlgo},
	there exists a constant $C > 0$, which depends only on $c_\sigma, C_{b, \sigma}, C_{f, g}, d$ and $T$
(and not on $k, \epsilon $) 
        such that for all $k \ge 1$,
	\begin{equation}
		\label{eq:estimateError}
		0 \le J(\P_k) - J^* \le \epsilon C + \frac{C}{k}\left(1 + \frac{1}{\epsilon}\right).
	\end{equation}
\end{coro}
\begin{remark} \label{rmk:appr}
	We fix $\epsilon > 0$. By Corollary \ref{coro:estimationError}, approximating $J^*$ with a precision $\epsilon$ requires at most $O(1/\epsilon^2)$ iterations of our alternating minimization procedure.
\end{remark}
\begin{proof}[Proof of Corollary \ref{coro:estimationError}]
  Let $C_1 > 0$ be the constant appearing in the convergence rate in Theorem \ref{th:convAlgo}. Let also $C_2 = C^* > 0$ be the constant provided by Proposition \ref{prop:approximateControl} item $1.$
  We recall that $C_1$ and $C_2$ depend only on $c_\sigma, C_{b, \sigma}, C_{f, g}, d$ and $T$. Let us fix  $\epsilon' = \frac{C_1}{k}\left(1 + \frac{1}{\epsilon}\right)$.
  Theorem \ref{th:convAlgo} states that $(\P_k, \Q_k) \in \sha$ is an $\epsilon'$-solution of the penalized Problem \eqref{eq:penalizedProblemIntro}. Then,
  by Proposition \ref{prop:approximateControl} item 2., we have that
   \begin{equation}     
   		\label{eq:firstStepEstimateError}
   		0 \le J(\P_k) - J^* \le \epsilon e^{\epsilon \E[Y_k]}Var^{\P_k}[Y_k] +
      \frac{C_1}{k}\left(1 + \frac{1}{\epsilon}\right)  \le \epsilon C_2e^{C_2} + \frac{C_1}{k}\left(1 + \frac{1}{\epsilon}\right),
   \end{equation}
	where
	$$
	 Y_k = \int_0^T f(r, X_r, u^k(r,X_r)) dr + g(X_T),
	$$
	and \eqref{eq:estimateError} follows from \eqref{eq:firstStepEstimateError} setting $C = C_1 \vee C_2e^{C_2}$.
\end{proof}

Besides Lemma \ref{lemma:sequenceAlternateDirection},
the proof of Theorem \ref{th:convAlgo} uses the so called
three and four points properties introduced in \cite{CsiszarAlternating}.

\begin{lemma}(Three points property).
	\label{lemma:3Points}
We suppose the validity of the hypotheses of Theorem \ref{th:convAlgo}. For all $\Q \in \Pma(\Omega)$,
	\begin{equation}  \label{eq:3points}
          \frac{1}{\epsilon}H(\Q | \Q_{k+1}) + \shj(\Q_{k+1}, \P_k) \le
          \shj(\Q, \P_k).
	\end{equation}
\end{lemma}
\begin{proof}
	We can suppose  that $H(\Q | \P_{k}) < + \infty$,
	otherwise  $\shj(\Q, \P_k) = + \infty$ and the inequality  holds trivially.
	Let 
	$$
	\varphi : X \mapsto \int_0^T f(r, X_r, u^k(r, X_r))dr + g(X_T),
	$$
	where $u^k$ (and $\P_k$) have been defined in
 \eqref{eq:decompPn} and just before.
	By the definition \eqref{eq:defQk} we have
	$$
	\frac{d\Q_{k + 1}}{d\P_k} = \frac{\exp(-\epsilon \varphi(X))}{\E^{\P_k}[\exp(-\epsilon \vphi(X))]}.
	$$
	Now $d\Q_{k + 1}/d\P_k > 0$ implies $\Q_{k + 1} \sim \P_k$, hence
      taking into account  $H(\Q | \P_{k}) < + \infty$, we have that
      $\Q \ll \Q_{k + 1}$, so that, $\Q$-a.s.,
	\begin{equation*}
			\log \frac{d\Q}{d\P_k} = \log \frac{d\Q}{d\Q_{k + 1}} + \log \frac{d\Q_{k + 1}}{d\P_k}
			= \log \frac{d\Q}{d\Q_{k + 1}} - \epsilon \vphi(X) - \log \E^{\P_k}\left[\exp\left(-\epsilon \vphi(X)\right)\right]. 
	\end{equation*}
	Taking the expectation under $\Q$ in the previous equality and dividing both sides by $\epsilon > 0$,
        yields
	\begin{equation*}
		\begin{aligned}
			\frac{1}{\epsilon} H(\Q \vert \Q_{k + 1}) & = \frac{1}{\epsilon}H(\Q | \P_k) + \frac{1}{\epsilon}\log \E^{\P_k}\left[\exp(-\epsilon \vphi(X))\right] + \E^{\Q}[\vphi(X)]\\
			& = \shj(\Q, \P_k) - \shj(\Q_{k + 1}, \P_k),
		\end{aligned}
	\end{equation*}
	where we have used Lemma \ref{lemma:sequenceAlternateDirection} item $1.$ for the latter equality.
\end{proof}

\begin{remark}
	\label{remark:3points}
	Whenever $H(\Q | \P_{k}) < + \infty$, previous proof shows that \eqref{eq:3points} is indeed an equality.
\end{remark}
      
\begin{lemma}(Four points property).
	\label{lemma:4Points}
We suppose the validity of the hypotheses of Theorem \ref{th:convAlgo}. For all $(\P, \Q) \in \mathcal{A}$,
	\begin{equation} \label{eq:4points}
	\shj(\Q, \P_{k+1}) \le \frac{1}{\epsilon}H(\Q | \Q_{k+1}) + \shj(\Q, \P).
	\end{equation}
\end{lemma}
\begin{proof}

    Let $(\P, \Q) \in \mathcal{A}.$ If $H(\Q | \Q_{k + 1}) = + \infty$ or $\shj(\Q, \P) = + \infty$, the inequality is trivial.
    We then assume until the end of the proof that $H(\Q | \Q_{k+1}) < + \infty$ and $\shj(\Q, \P) < + \infty$.

  We first do some preliminary calculations.
We recall that, by
  \eqref{eq:decompQn},
                  there exists a measurable function $\beta^{k +1} : [0, T] \times \R^d \rightarrow \R^d$ such that under $\Q_{k + 1}$ the canonical process has decomposition
		\begin{equation*} 
			X_t = x + \int_0^t \beta^{k+1}(r, X_r)dr + M_t^{\Q_{k + 1}},
		\end{equation*}
		where $M^{\Q_{k + 1}}$ is a local martingale under $\Q_{k + 1}$ and
		$[ M^{\Q_{k + 1}}] = \int_0^\cdot \sigma \sigma^\top(r, X_r)dr$.
                We provide
                now
                a useful lower bound for                $H(\Q | \Q_{k + 1})$.
		By
		Lemma \ref{lemma:girsanovEntropy} item $1.$ in the Appendix applied
		with $\P = \Q_{k + 1}$
		and the fact that $H(\Q | \Q_{k + 1}) < + \infty$,
		there exists an $(\shf_t)$-progressively measurable process $\alpha = \alpha(\cdot,X)$
		such that,
		under $\Q$, the canonical process has the decomposition
		\begin{equation} \label{eq:Q} 
			X_t = x + \int_0^t \beta^{k+1}(r, X_r)dr + \int_0^t \sigma\sigma^\top(r, X_r)\alpha(r, X) dr + M_t^\Q,
		\end{equation}
		where $M^\Q$ is a local martingale such that $[ M^\Q] = \int_0^\cdot \sigma \sigma^\top(r, X_r)dr$, and
		\begin{equation} \label{eq:Qk1} 
			H(\Q | \Q_{k+1}) \ge \frac{1}{2}\E^\Q\left[\int_0^T |\sigma^\top(r, X_r)\alpha(r, X)|^2dr\right].
		\end{equation}
		We set
		\begin{equation} \label{eq:Q1}
			{\hat \beta}(t, X) := \beta^{k+1}(t, X_t) + \sigma\sigma^\top(t, X_t)\alpha(t, X),
		\end{equation}
		so that \eqref{eq:Qk1} can be rewritten
		\begin{equation}
			\label{eq:lowerBoundEntropy}
			H(\Q | \Q_{k+1}) \ge \frac{1}{2}\E^\Q\left[\int_0^T |\sigma^{-1}(r, X_r)({\hat \beta}(r, X) -
                          \beta^{k+1}(r, X_r))|^2dr\right],
		\end{equation}
where we recall that $\sigma^{-1}$ is the right-inverse of $\sigma$.

We proceed now with the proof of the four points property \eqref{eq:4points}.
Let $u^{k+1}$ and $u^{\P_{k+1}}$ be as in \eqref{eq:decompPn} and just before
so that
$u^{k+1} = u^{\P_{k+1}}$.
  We set
		\begin{equation}
			\label{eq:defYZ4Points}
			\begin{aligned}
				& y_r^\P := b(r, X_r, \nu^\P(r, X)), \quad y^{k + 1}_r := b(r, X_r, u^{\P_{k + 1}}(r, X_r)),\\
				& z_r^\P := f(r, X_r, \nu^\P(r, X)), \quad z_r^{k + 1} :=  f(r, X_r, u^{\P_{k + 1}}(r, X_r)),
			\end{aligned}			
		\end{equation}
		where $\nu^\P$ (resp. $u^{\P_{k + 1}}$) is   associated to $\P$ (resp. $\P_{k + 1}$), according to Definition \ref{def:PU}. Let $\hat F_{\hat \beta}$ defined in Remark \ref{rmk:fBeta}. Then
		\begin{equation}
			\label{eq:4PointsEq1}
			\begin{aligned}
          {\hat F}_{\hat \beta}(r, X, \nu^\P_r) - {\hat F}_{\hat \beta}(r, X, u^{\P_{k + 1}}(r, X_r)) & =
                                  z_r^\P - z_r^{k + 1}
				+ \frac{1}{2\epsilon}|\sigma^{-1}(r, X_r)({\hat \beta}(r, X) - y^\P_r)|^2\\
				& - \frac{1}{2\epsilon}|\sigma^{-1}(r, X_r)({\hat \beta}(r, X) - y^{k + 1}_r)|^2.
                              \end{aligned}
                            \end{equation}
 We focus on the last two terms in the previous inequality. We apply the  algebraic equality
		$|a|^2 - |b|^2 = |a - b|^2 + 2\langle a - b, b\rangle, $
		with
		$
		a = \sigma^{-1}({\hat \beta} - y^\P), \quad b =  \sigma^{-1}({\hat \beta} - y^{k+1}),$
		  		where for conciseness we have omitted the dependencies in $(r, X)$ of all the quantities at hand.
                So
                we have     
		\begin{equation*}
			\begin{aligned}
				\frac{1}{2\epsilon}|\sigma^{-1}({\hat \beta} - y^\P)|^2 - \frac{1}{2\epsilon}|\sigma^{-1}({\hat \beta} - y^{k + 1})|^2=  \frac{1}{2\epsilon}|\sigma^{-1}(y^\P - y^{k + 1})|^2 + \frac{1}{\epsilon}\langle \sigma^{-1}(y^\P - y^{k + 1}), \sigma^{-1}(y^{k + 1} - {\hat \beta}) \rangle.
			\end{aligned}
		\end{equation*}
		On the other hand,
		\begin{equation*}
			\begin{aligned}
				\frac{1}{\epsilon}\langle \sigma^{-1}(y^\P - y^{k + 1}), \sigma^{-1}(y^{k + 1} - {\hat \beta}) \rangle
                          & = \frac{1}{\epsilon}\langle \sigma^{-1}(y^\P - y^{k + 1}), \sigma^{-1}(y^{k + 1} -
                            {\beta}^{k + 1}) \rangle\\
				& + \frac{1}{\epsilon}\langle \sigma^{-1}(y^\P - y^{k + 1}), \sigma^{-1}({\beta}^{k + 1} - {\hat \beta}) \rangle.
			\end{aligned}
		\end{equation*}
		Combining what precedes yields
		\begin{equation*}
			\begin{aligned}
				\frac{1}{2\epsilon}|\sigma^{-1}({\hat \beta} - y^\P)|^2 - \frac{1}{2\epsilon}|\sigma^{-1}({\hat \beta} - y^{k + 1})|^2 & = \frac{1}{2\epsilon}|\sigma^{-1}(y^\P - y^{k + 1})|^2 + \frac{1}{\epsilon}\langle y^\P - y^{k + 1}, (\sigma^{-1})^\top\sigma^{-1}(y^{k + 1} - \beta^{k + 1}) \rangle\\
				& + \frac{1}{\epsilon}\langle \sigma^{-1}(y^\P - y^{k + 1}), \sigma^{-1}({\beta}^{k + 1} - {\hat \beta}) \rangle.
			\end{aligned}
		\end{equation*}
		From the inequality \eqref{eq:4PointsEq1} we then get 
     \begin{equation} \label{eq:ineqsub}
       \begin{aligned}
				& {\hat F}_{\hat \beta}(r, X, \nu_r^\P) - {\hat F}_{\hat \beta}(r, X, u^{k + 1}(r, X_r)) = \frac{1}{2\epsilon}|\sigma^{-1}(r, X_r)(y^\P_r - y^{k+1}_r)|^2\\
				& + \frac{1}{\epsilon}\langle \sigma^{-1}(r, X_r)(\beta^{k+1}(r, X_r)
                                  - {\hat \beta}(r, X)), \sigma^{-1}(r, X_r)(y^{\P}_r - y^{k+1}_r)\rangle\\
                & + z_r^\P - z_r^{k + 1} + \frac{1}{\epsilon} \left\langle (\sigma^{-1})^\top\sigma^{-1}(r, X_r)(y_r^{k + 1} - \beta^{k + 1}(r, X_r)),  y_r^\P - y_r^{k + 1}\right\rangle.
			\end{aligned}
                      \end{equation}
		By  \eqref{eq:defPk}
		$u^{k + 1}(t, x)$ achieves the minimum of $F_{\beta^{k + 1}}(t, x, .)$, for all $(t, x) \in [0, T] \times \R^d$, where the application $F_{\beta^{k + 1}}$ is the one defined in \eqref{eq:fBeta}.
                Taking into account \eqref{eq:defPk},
 Lemma \ref{lemma:infKtx}
$2 (b)$,
                applied for any
                $(t, x) \in [0, T] \times \R^d$ with $u^* = u^{k + 1}(t, x)$ and $\beta = \beta^{k + 1}$, states that
the restriction of the function $\bar F^{t, x}_{\beta^{k + 1}}$ given by \eqref{eq:barFbeta} to the convex set $K(t,x)$, achieves
its minimum  at the point
$
(b(t, x, u^{k + 1}(t, x)), f(t, x, u^{k + 1}(t, x))).
$
Consequently, for the generic probability measure $\P$,
Lemma \ref{lemma:infKtx} item $1.$, applied with
$\beta = \beta^{k + 1}$, $(y^*, z^*) = (y^{k + 1}_r, z^{k + 1}_r)$ and $(y, z) = (y^\P_r, z^\P_r)$,
shows that  $(y^{k + 1}_r, z^{k + 1}_r)$ is the unique minimum of ${\bar F}^{t,x}_\beta$,
so that the term on the third line of the equality \eqref{eq:ineqsub} is non-negative.       
Then \eqref{eq:ineqsub} yields
		\begin{equation}
			\label{eq:firstOrderIneq}
                             \begin{aligned}
                          & {\hat F}_{\hat \beta}(r, X, \nu_r^\P) - {\hat F}_{\hat \beta}(r, X, u^{k + 1}(r, X_r))dr \ge \frac{1}{2\epsilon}|\sigma^{-1}(r, X_r)(y_r^\P - y_r^{k + 1})|^2\\
			& + \frac{1}{\epsilon}\langle \sigma^{-1}(r, X_r)({\beta}^{k+1}(r, X_r)
                          - {\hat \beta}(r, X)), \sigma^{-1}(r, X_r)(y^{\P}_r - y^{k+1}_r)\rangle.
			\end{aligned}
                      \end{equation}
		Next, by the classical inequality $|ab| \le a^2/2 + b^2/2$ for all $(a, b) \in \R^2$, the right-hand side
                term in inequality \eqref{eq:firstOrderIneq} gives
		\begin{equation*}
			\begin{aligned}
				& \frac{1}{\epsilon}\langle \sigma^{-1}(r, X_r)(\beta^{k+1}(r, X_r) - {\hat \beta}(r, X)), \sigma^{-1}(r, X_r)(y^{\P}_r - y^{k+1}_r)\rangle \\
				\ge & - \frac{1}{2\epsilon} |\sigma^{-1}(r, X_r)({\hat \beta}(r, X) - \beta^{k+1}(r, X_r))|^2\\
				& -\frac{1}{2\epsilon}|\sigma^{-1}(r, X_r)(y^\P_r - y^{k+1}_r)|^2, 
			\end{aligned}
		\end{equation*}
		and from inequality \eqref{eq:firstOrderIneq} we get
		$$
		{\hat F}_{\hat \beta}(r, X, \nu_r^\P) + \frac{1}{2\epsilon}|\sigma^{-1}(r, X_r)({\hat \beta}(r, X) - \beta^{k+1}(r, X_r))|^2 \ge {\hat F}_{\hat \beta}(r, X, u^{k + 1}(r, X_r)).
		$$
		Integrating previous inequality with respect to
                $r \in [0, T]$, yields
		\begin{equation}
			\label{eq:4pointLastStep}
			\int_0^T {\hat F}_{\hat \beta}(r, X, \nu^\P_r)dr  + \frac{1}{2\epsilon} \int_0^T|\sigma^{-1}(r, X_r)({\hat \beta}(r, X) - \beta^{k+1}(r, X_r))|^2dr \ge \int_0^T {\hat F}_{\hat \beta}(r, X, u^{k + 1}(r, X_r))dr.
		\end{equation}
              Since  $H(\Q | \P) < + \infty$, by Definition \ref{def:PU},
                Lemma \ref{lemma:girsanovEntropy} item $1.$
                with $\delta(\cdot, X) = b(., X_\cdot, \nu^\P(\cdot, X))$
(writing $\nu^\P(r,X)= \nu^\P_r$),
                states the existence of a predictable process $\tilde \alpha$ such that
		\begin{equation}\label{eq:841bis}
                  X_t = x + \int_0^t b(r, X_r, \nu_r^\P) dr
                  + \int_0^t\sigma\sigma^\top(r, X_r) \tilde \alpha(r, X)dr + \tilde M_t^\Q,
		\end{equation}
		where $\tilde M^\Q$ is a ($\Q, \shf_t$)-local martingale with $[\tilde M^\Q] = \int_0^\cdot \sigma\sigma^\top(r, X_r)dr$.
	By \eqref{eq:Q} and \eqref{eq:Q1}, under $\Q$, the canonical process decomposes as
\begin{equation}
			\label{eq:lastDecompQ}
			X_t = x + \int_0^t {\hat \beta}(r, X)dr + M_t^\Q,
		\end{equation}
		where $M^\Q$ is a local martingale verifying $[M^\Q] = \int_0^{\cdot} \sigma\sigma^\top(r, X_r)dr$.

                Identifying the bounded variation component between \eqref{eq:841bis} and decomposition \eqref{eq:lastDecompQ}
		(under $\Q$), yields $\hat \beta(r, X) - b(r, X_r, \nu^\P_r)
                = \sigma\sigma^\top(r, X_r) \tilde \alpha(r, X)$
		and \eqref{eq:inequalityEntropy} in Lemma \ref{lemma:girsanovEntropy} item $1.$ implies that
		\begin{equation}
			\label{eq:entropQP}
			H(\Q | \P) \ge \frac{1}{2}\E^{\P}\left[\int_0^T |\sigma^{- 1}(r, X_r)(b(r, X_r, \nu_r^\P)-
                          {\hat \beta}(r, X))|^2dr\right].
                      \end{equation}     
Then recalling the definition of $\shj$ in \eqref{eq:penalizedProblemIntro},
previous inequality \eqref{eq:entropQP} yields
		\begin{equation}
			\label{eq:costQP}
			\shj(\Q, \P) \ge \E^{\Q}\left[\int_0^T {\hat F}_{\hat \beta}(r, X, \nu^\P_r)dr + g(X_T) \right].
		\end{equation}
		From \eqref{eq:costQP} and \eqref{eq:lowerBoundEntropy} it holds
		\begin{equation*}
			\begin{aligned}
                          \shj(\Q, \P) + \frac{1}{\epsilon}H(\Q | \Q_{k + 1}) & \ge \E^{\Q}\left[\int_0^T
     {\hat F}_{\hat \beta}(r, X, \nu_r^\P)dr + g(X_T) \right]\\
				& + \frac{1}{2\epsilon}\E^{\Q}\left[\int_0^T|\sigma^{-1}(r, X_r)({\hat \beta}(r, X) - \beta^{k+1}(r, X_r))|^2dr\right],
			\end{aligned}
		\end{equation*}
		and by \eqref{eq:4pointLastStep}
		\begin{equation}
			\label{eq:4pointLastStep2}
			\shj(\Q, \P) + \frac{1}{\epsilon}H(\Q | \Q_{k + 1}) \ge \E^{\Q}\left[\int_0^T
                          {\hat F}_{\hat \beta}(r, X, u^{k + 1}(r, X_r))dr + g(X_T) \right].
		\end{equation}
		In particular, since $g \ge 0$,
     we have $\E^{\Q}\left[\int_0^T {\hat F}_{\hat \beta}(r, X, u^{k + 1}(r, X_r))dr \right] < + \infty$, hence, recalling  the expression \eqref{eq:rmkFbeta} 
		$$
		\E^{\Q}\left[\int_0^T |\sigma^{- 1}(r, X_r)(b(r, X_r, u^{k + 1}(r, X_r)) -
                  {\hat \beta}(r, X))|^2dr\right] < + \infty.
		$$
                We keep in mind \eqref{eq:Q} and  \eqref{eq:Q1}. By Lemma \ref{lemma:sequenceAlternateDirection} item 2., the decomposition \eqref{eq:decompPn} is unique in law. Then,
                by Lemma \ref{lemma:girsanovEntropy} item $2.$ applied to
                $\P = \P_{k + 1} (= \P^{u^{k+1}})$  with $\delta(\cdot, X) = b(\cdot, X_\cdot, u^{k + 1}(\cdot, X_\cdot))$ and $\gamma = {\hat \beta}$, we have
		\begin{equation}
			\label{eq:entropQPk}
			H(\Q | \P_{k + 1}) = \frac{1}{2}\E^{\Q}\left[\int_0^T |\sigma^{- 1}(r, X_r)(b(r, X_r, u^{k + 1}(r, X_r)) - {\hat \beta}(r, X))|^2dr\right],
                      \end{equation}
                      and so
		\begin{equation}
			\label{eq:costQPk}
	\shj(\Q, \P_{k + 1}) = \E^{\Q}\left[\int_0^T {\hat F}_{\hat \beta}(r, X, u^{k + 1}(r, X_r))dr + g(X_T) \right].
		\end{equation}
		Finally, combining \eqref{eq:4pointLastStep2} and \eqref{eq:costQPk}, we get
		$$
		\shj(\Q, \P) + \frac{1}{\epsilon}H(\Q | \Q_{k + 1}) \ge \shj(\Q, \P_{k + 1}).
		$$
		This concludes the proof.
\end{proof}

\begin{lemma}
	\label{lemma:estimateOptimalSolution}
	Let $(\P_\epsilon^*, \Q_\epsilon^*)$ be an optimal solution
                to Problem \eqref{eq:penalizedProblemIntro}, given by Theorem
        \ref{th:existenceSolutionRegProb},
        under the assumptions of the aforementioned theorem.
 Let $(\P_k, \Q_k)_{k\ge 0}$ be given by the recursion \eqref{eq:defQk} and
      just before  \eqref{eq:decompPn}.

        There exists a constant  $C > 0$,  which only depends on $c_\sigma, C_{b, \sigma}, C_{f, g}, d$ and $T$, such that, for all $k \ge 0$, $\shj(\Q_\epsilon^*, \P_k) \le C\left(1 + \frac{1}{\epsilon}\right)$.
\end{lemma}

The proof of the  result above is postponed to Appendix \ref{app:proofEstimate} for clarity.

\begin{proof}[Proof of Theorem \ref{th:convAlgo}.]
  Combining \eqref{eq:3points} in Lemma \ref{lemma:3Points}
  and \eqref{eq:4points} in Lemma \ref{lemma:4Points} we get, for all $k \ge 0$, the so called five points property
	\begin{equation}
		\label{eq:5points}
		\shj(\Q, \P_{k + 1}) + \shj(\Q_{k + 1}, \P_k) \le \shj(\Q, \P_k) + \shj(\Q,\P).
	\end{equation}
	Evaluating \eqref{eq:5points} for $(\P,\Q)$  being
        the solution $(\P_\epsilon^*, \Q_\epsilon^*)$ of the penalized problem given by Theorem \ref{th:existenceSolutionRegProb}, we obtain
	\begin{equation}
		\label{eq:5Points}
		\shj(\Q_\epsilon^*, \P_{k + 1}) + \shj(\Q_{k + 1}, \P_k) \le \shj(\Q_\epsilon^*, \P_k) + \shj(\Q_\epsilon^*,\P_\epsilon^*),
	\end{equation}
	and since $\shj(\Q_\epsilon^*, \P_k) < + \infty$, by Lemma \ref{lemma:estimateOptimalSolution}, the previous inequality rewrites
	\begin{equation}
		\label{eq:interRate}
		\shj(\Q_{k + 1}, \P_k) - \shj^* \le \shj(\Q_\epsilon^*, \P_k) - \shj(\Q_\epsilon^*, \P_{k + 1}),
	\end{equation}
	where we have used the equality $\shj(\Q_\epsilon^*,\P_\epsilon^*) = \shj^*$.

        Let $K$ be a fixed number of iterations of the algorithm. Summing equation \eqref{eq:interRate} between $0$ and $K - 1$ and dividing each member of the inequality by $K$, we get
	\begin{equation}
		\label{eq:interRate2}
		\frac{1}{K}\sum_{k = 0}^{K - 1}\shj(\Q_{k + 1}, \P_k) - \shj^* \le \frac{1}{K}\left(\shj(\Q_\epsilon^*, \P_0) - \shj(\Q_\epsilon^*, \P_K)\right).
	\end{equation}
	By construction of the sequence $(\P_k, \Q_k)_{k\ge 0}$, it holds that
	\begin{equation}
		\label{eq:infSequence}
		\shj(\Q_{k + 1}, \P_k) \ge \shj(\Q_{k + 1}, \P_{k + 1}) \ge \shj(\Q_{k + 2}, \P_{k + 1}) \ge \dots \ge \shj(\Q_K, \P_K) 
              \end{equation}
              for all $k \le K-1$.
              Applying \eqref{eq:infSequence} in \eqref{eq:interRate2} gives
	\begin{equation}
	 	\label{eq:convergenceRate}
	0 \le \shj(\Q_K, \P_K) - \shj^* \le \frac{1}{K}\left(\shj(\Q_\epsilon^*, \P_0) - \shj(\Q_\epsilon^*, \P_K)\right) \le \frac{1}{K}\shj(\Q_\epsilon^*, \P_0).
	\end{equation}
	Finally, by Lemma \ref{lemma:estimateOptimalSolution}, there exists a constant $C > 0$, which only depends on $c_\sigma, C_{b, \sigma}, C_{f, g}, d$ and $T$, such that $\shj(\Q_\epsilon^*, \P_0) \le C\left(1 + \frac{1}{\epsilon}\right)$ and \eqref{eq:convergenceRate} yields
	$$
	0 \le \shj(\Q_K, \P_K) - \shj^* \le \frac{C}{K}\left(1 + \frac{1}{\epsilon}\right).
	$$
	Previous relation proves the convergence of the algorithm and exhibits a convergence rate for fixed $\epsilon$. This concludes the proof.
\end{proof}

\begin{remark} \label{rmk:ConvRate}
  To prove Lemma \ref{lemma:3Points} and Lemma \ref{lemma:4Points}, one
  can relax the continuity assumption on $b$ in Hypothesis \ref{hyp:coefDiffusion} and assume instead that $b(t, x, \cdot)$ is continuous for all $(t, x) \in [0, T] \times \R^d$ and $b(\cdot, \cdot, u)$ is measurable for all $u \in \U$.
  Then $\shj$ verifies the so-called five point property \eqref{eq:5points} and Theorem 2 in \cite{CsiszarAlternating} ensures that $\shj(\Q_k, \P_k) \underset{k \rightarrow + \infty}{\searrow} \shj^*$. However our proof of Theorem \ref{th:existenceSolutionRegProb} strongly relies on the continuity of $b$ in $(t, x, u) \in [0, T] \times \R^d \times \U$, and this stronger regularity allows to exhibit a convergence rate in Theorem \ref{th:convAlgo}.
\end{remark}

We conclude the section by stating a lemma which is a reformulation in our setting of Proposition 3.9 in \cite{EntropyWeighted}.
This allows us to estimate the drift $\beta^k$ in the algorithm via a conditional derivative.
\begin{lemma}
	\label{lemma:regDrift}
	Assume Hypothesis \ref{hyp:coefDiffusion} and
         that a probability $\P_0$ verifies Hypothesis \ref{hyp:initialPoint}.
        Consider the sequence constructed after Remark \ref{rmk:fBeta}. 
	For almost all $0 \le t < T$, it holds that
	\begin{equation}
		\label{eq:approxDrift}
		\lim_{h \downarrow 0} \E^{\Q_k}\left[\frac{X_{t + h} - X_t}{h}~\Big| ~X_t\right] = \beta^k(t, X_t)~\text{in}~L^1(\Q_k).
	\end{equation}
\end{lemma}
\begin{proof}
  We fix some $1 < p < 2.$ By decomposition \eqref{eq:decompQn},
replacing $k+1$ with $k$,
	in order to apply Lemma \ref{lemma:nelsonDerivative}, it is enough to have
	$\| \beta^k\|_{L^p(dt \otimes \Q_k)} < + \infty$, which is guaranteed by item $1.$ of Lemma \ref{lemma:sequenceAlternateDirection}. Consequently Lemma \ref{lemma:nelsonDerivative} and Remark \ref{remark:nelsonDerivative}
	yield the result.
\end{proof}


	\begin{remark} \label{rmk:Algorithm}
          Our algorithm has the advantage of relying on two standard optimization sub-problems that are simpler than the original stochastic control problem: on the one hand, an exponential twist problem \eqref{eq:defQk} and, on the other hand, a convex pointwise optimization problem \eqref{eq:defPk}. From a numerical point of view, each of the subproblems can be solved numerically by specific approaches. For example, solving the exponential twist problem can be reduced
          to computing independent conditional expectations on each time step, as shown in Lemma \ref{lemma:regDrift},
    hence those computations can be easily parallelized.
    Our method has a clear advantage with respect to HJB or BSDEs representation
    of the solution of control problems, which involve nonlinearly nested conditional expectations,
    because of the backward dynamical programming recursion.
In our context 
the conditional expectation computations can be also efficiently addressed by deep learning methods when the dimension is high,
    see e.g. 
    \cite{MLPDEWarin, Germain22, HighDimensionalPDEDL, Hure20}.
    However, in the numerical applications considered in Section \ref{sec:example}, we choose to use a simple polynomial regression Monte-Carlo method,
    since we restrict ourselves to a dimension less than 20.
  This linearization effect has of course the cost of repeating the procedure along $k$ iterations, 
  allowing a convergence rate  $O(\frac{1}{k})$.

  \end{remark}

\subsection{Entropy penalized Monte-Carlo algorithm}

The alternating minimization procedure in Section \ref{sec:convergenceResult} suggests a Monte-Carlo algorithm to approximate a solution to Problem \eqref{eq:controlProblemIntro}. In the following, $0 = t_0 \le t_1 < ... < t_M = T$, is a regular subdivision of the time interval $[0, T]$ with step $\Delta t,$ $N \ge 0$
being the number of particles and $K$ the number of descent steps of the algorithm. $P_r$ will denote the set of $\R^d$-valued polynomials defined on $\R^d$ of degree $\le r$. Recall that for all $\hat u \in \B([0, T] \times \R^d, \U)$, $\P^{\hat u}$ is the probability measure given by Proposition \ref{prop:existencePu}.
The estimation of the drift ${\hat \beta}^k$
in Step 2 of the  algorithm below,
is performed via regression.
This
is inspired by \eqref{eq:approxDrift}
in Lemma \ref{lemma:regDrift}.
The term in the argmin 
is a weighted Monte-Carlo approximation of the
expectation of $\frac{X_{m + 1}^n - X_m^n}{\Delta t} $,
under the exponential twist of the probability measure
$\P^{\hat u^{k-1}}$.

\begin{algorithm}[H]
	\caption{Entropy penalized Monte-Carlo algorithm}
	\begin{algorithmic}
          \State {\bf Parameters initialization:} $M, N, K \in \N^*,~r \in \N, ~\Delta t := \frac{T}{M},~x \in \R^d$,
         $\hat u^0 \in \B([0,T] \times \R^d, \U)$.
         
         \State {\bf Simulate:} $(X^n)_{1 \le n \le N}$,  $N$ iid Monte-Carlo path
          simulations under ${\hat \P}_0 = \P^{{\hat u}^0}$
          on the time-grid $(t_m)_{0 \le m \le M}$, with $X^n = (X^n_m)_{0 \le m \le M}$ and $X^n_0 = x$, for all $1 \le n \le N$.
		\For {$1 \le k \le K$}
		\State \textbf{Step 1. \ }
           Compute the weights $(D_n)_{1 \le n \le N}$ by
		$$
		D_n = \exp\left(-\epsilon \sum_{m = 0}^{M - 1} f(t_m, X_m^n,
                  {\hat u}^{k - 1}(t_m,X_m^n))\Delta t - \epsilon g(X_M^n)\right).
		$$
		\State \textbf{Step 2. \ } Compute
                ${\hat \beta}^k = ({\hat \beta}^k_m)_{0 \le m \le M-1} $ in \eqref{eq:decompQn} by the weighted Monte-Carlo approximation of \eqref{eq:approxDrift}
                $$
		{\hat \beta}_m^k \in \argmin_{\vphi \in P_r}
              \frac{1}{\sum_{\ell= 1}^N D_\ell}  \sum_{n = 1}^ND_n\left|\vphi(X_m^n) -
           \frac{X_{m + 1}^n - X_m^n}{\Delta t} \right|^2.
         $$
		
		\State  \textbf{Step 3. \ } Simulate new iid Monte-Carlo paths $(X^n)_{1 \le n \le N}$ under $\P^{\hat u^k}$, where for $0 \le m \le M - 1$

		\begin{equation}
			\label{eq:approxControl}
			{\hat u}^{k}(t, x) = \underset{a \in \U}{\argmin}~
		f(t_m, x, a) + \frac{1}{2\epsilon}|\sigma^{-1}(t_m, x)
                ({\hat \beta}_m^k( x) - b(t_m, x, a))|^2, \ t \in [t_m,t_{m+1}[.
		\end{equation}
		\EndFor\\
		\Return ${\hat u}^K$
	\end{algorithmic}
	\label{algo:mcEntropy}
\end{algorithm}

\begin{remark} \label{rmk:Storing}
  The algorithm stores the functions $\hat \beta^k$, from which the controls are computed. In our implementation these functions
  are polynomial regressors (whose coefficients are stored at each time steps) but one could also imagine storing them in the
  form of neural networks or any other machine learning models.
  The algorithm actually returns $\hat \beta^k$ after $k$ iterations,
from which the feedback $\hat u^k$ can be evaluated in each point $(t, x)$ by solving the minimization problem \eqref{eq:approxControl}, 
  which defines a measurable function, by Proposition \ref{prop:pointwiseMinimization} below.
  Thus an optimal feedback control is an output of the algorithm.
\end{remark}

An interest of the entropy penalized Monte-Carlo algorithm is that in Lemma \ref{lemma:regDrift}, \eqref{eq:approxDrift} can be independently estimated by regression techniques at each time step $t_m$, $1 \le m \le M$, while in dynamic programming approaches, conditional expectations are recursively computed in time, implying an error accumulation from time $t_M = T$ to $t_m$.
Moreover one can expect that the trajectories simulated under $\P^{\hat u^k}$, localize around the optimally controlled trajectories, when the number of iterations $k$ of the algorithm increases to $+ \infty$. Hence, the computational effort to estimate the optimal control, focuses on this specific region of the state space, whereas standard regression based Monte-Carlo approaches are blindly exploring the state space, with forward Monte-Carlo simulations of the process.

\section{Solving the subproblems}
\label{sec:subProblems}
In this section we aim at describing the two subproblems $\underset{\P \in \Pma_{\U}}{\inf} \shj(\Q, \P)$ and $\underset{\Q \in \Pma_{\Omega}}{\inf} \shj(\Q, \P)$ appearing in the alternating minimization algorithm proposed in Section \ref{sec:algo}.
\subsection{Pointwise minimization subproblem}
\label{sec:pointwiseMinimization}

\setcounter{equation}{0}

Let us first describe the minimization $\underset{\P \in \Pma_{\U}}{\inf} \shj(\Q, \P)$ where the probability $\Q \in \Pma(\Omega)$ is fixed and is
such that, under $\Q$, the canonical process is a fixed Itô process.
In this section, we assume that Hypotheses \ref{hyp:coefDiffusion}, \ref{hyp:costFunctionsControl} and \ref{hyp:convexSet} are fulfilled.
Let  $p \ge 1$ be the real intervening in Hypothesis \ref{hyp:costFunctionsControl} item $1$.
In the sequel of the present section we also make a specific assumption for a given probability $\Q$ on the canonical space.
\begin{hyp}	\label{hyp:QMarkov}
  There is a Borel function  $\beta:[0,T] \times \R^d \rightarrow \R$ for which the canonical process $X$ decomposes as
\begin{equation}
  	\label{eq:decompQ}
 	X_t = x + \int_0^t \beta(r, X_r)dr + M^\Q_t,
 \end{equation}
 where $M^\Q$ is a local martingale verifying $[ M^\Q ] = \int_0^{\cdot} \sigma\sigma^\top(r, X_r)dr$. Moreover, $\E^\Q\left[\underset{0 \le r \le T}{\sup}|X_r|^p\right] < + \infty$.
\end{hyp}
For the proposition below we recall that if $u:[0,T] \times \R^d \rightarrow \R$ is a Borel function then
 $\P^u \in \Pma_\U^{Markov}$ denotes the  associated probability measure given by Proposition \ref{prop:existencePu}.
 \begin{prop} 
	\label{prop:pointwiseMinimization}
        There exists a measurable function $(t, x) \mapsto u(t, x) \in \U$ such that
	\begin{equation} 
          \label{eq:pointwiseMinimization}
		u(t, x) \in \argmin_{a \in \U} F_\beta(t, x, a),
              \end{equation}
              where $F_\beta$ is given by \eqref{eq:fBeta}, which is well-defined and measurable.
              Moreover $\shj(\Q, \P^{u}) = \underset{\P \in \Pma_{\U}}{\inf} \shj(\Q, \P)$.
      \end{prop}

      \begin{proof}[Proof of Proposition \ref{prop:pointwiseMinimization}]
  We will make use of the function $\bar F_\beta^{t, x}$ defined by \eqref{eq:barFbeta},
  defined for all $(t, x) \in [0, T] \times \R^d$.
  We also keep in mind  the definition \eqref{eq:Ktx} of the convex set $K(t, x)$ where one
  will consider the restriction
      of $\bar F^{t, x}_\beta$.
  For all $(t, x) \in [0, T] \times \R^d$ let us consider $(y^*(t, x), z^*(t, x)) \in K(t, x)$ given by Lemma \ref{lemma:infKtx}
  item $1.$ By Theorem \ref{th:measurableSelection} there exists a measurable function $u \in \B([0, T] \times \R^d, \U)$ such that $y^*(t, x) = b(t, x, u(t, x))$ and $z^*(t, x) \ge f(t, x, u(t, x))$. By Lemma \ref{lemma:infKtx} item $2. (a)$,
\begin{equation} \label{eq:argmin}
  u(t, x) \in \argmin_{a \in \U}~F_\beta(t, x, a), \ \forall (t,x).
  \end{equation}
  By  Proposition \ref{prop:existencePu}, there is a probability measure
  $ \P^u$ belonging to $\Pma_\U^{Markov}$.
Let also $\P \in \Pma_\U$.
  In particular
  there exists a progressively measurable process $\nu_r^\P $,
  with values in $\U$ such that under $\P$ the canonical process $X$ has decomposition
	$$
	X_t = x + \int_0^t b(r, X_r, \nu_r^\P)dr + M_t^\P, \ t \in [0,T],
	$$
	where $M^\P$ is a local martingale verifying $[ M^\P ] = \int_0^\cdot \sigma\sigma^\top(r, X_r)dr$. We want to prove that
	\begin{equation}
		\label{eq:ineqJP}
		\shj(\Q, \P) \ge \shj(\Q, \P^u).
	\end{equation}
	If $\shj(\Q, \P) = + \infty$, inequality \eqref{eq:ineqJP} is trivially verified. Assume now that $\shj(\Q, \P) < + \infty$. In particular, $H(\Q | \P) < + \infty$ and by Lemma \ref{lemma:girsanovEntropy} item $1.(a)$, there exists a process $\alpha = \alpha(\cdot, X)$ such that under $\Q$, $X$ decomposes as
	\begin{equation}
		\label{eq:entropDecompQ}
		X_t = x + \int_0^t b(r, X_r, \nu_r^\P)dr + \int_0^t \sigma\sigma^\top(r, X_r)\alpha_rdr + \tilde M^\Q_t,
	\end{equation}
	where the local martingale $\tilde M^\Q$ verifies $[\tilde M^\Q] = \int_0^\cdot \sigma\sigma^\top(r, X_r)dr$,
	and
	\begin{equation}
		\label{eq:ineqEntropQ}
		H(\Q | \P) \ge \frac{1}{2}\E^\Q\left[\int_0^T |\sigma^\top(r, X_r)\alpha(r, X)|^2dr\right].
	\end{equation}
	Identifying the bounded variation and the local martingale parts in \eqref{eq:decompQ} and \eqref{eq:entropDecompQ} yields $\sigma^\top(t, X_t)\alpha(t, X) = \sigma^{-1}(t, X_t)(\beta(t, X_t) - b(t, X_r, \nu_t^\P))$ $d\Q \otimes dt$-a.e. and $\tilde M^\Q = M^\Q$. Replacing in \eqref{eq:ineqEntropQ} we get
	$$
	H(\Q | \P) \ge \frac{1}{2}\E^\Q\left[\int_0^T |\sigma^{-1}(r, X_r)(\beta(r, X_r) - b(r, X_r, \nu_r^\P))|^2dr\right],
	$$
	and the previous inequality yields
	\begin{equation}
		\label{eq:interPointwise}
		\begin{aligned}
			\shj(\Q, \P) & = \E^\Q\left[\int_0^T f(r, X_r, \nu_r^\P)dr + g(X_T) \right] + \frac{1}{\epsilon}H(\Q | \P)\\
			& \ge \E^\Q\left[\int_0^T f(r, X_r, \nu_r^\P)dr + g(X_T) + \frac{1}{2\epsilon}\int_0^T |\sigma^{-1}(r, X_r)(\beta(r, X_r) - b(r, X_r, \nu_r^\P))|^2dr \right].
		\end{aligned}
	\end{equation}
	By assumption, $\E^\Q\left[\underset{0 \le r \le T}{\sup}|X_r|^p\right] < + \infty$ and by \eqref{eq:linearGrowthbSigma} and \eqref{eq:polyGrowthfg} we have
	\begin{equation}
		\label{eq:finiteBF}
		\E^\Q\left[\int_0^T \left(|b(r, X_r, \nu_r^\P)| + |f(r, X_r, \nu_r^\P)|\right)dr\right] < + \infty.
              \end{equation}
An application of Fubini's theorem, the tower property and Jensen's inequality for conditional expectation 
in \eqref{eq:interPointwise} gives
	\begin{equation}
		\label{eq:interPointwise2}
		\begin{aligned}
			\shj(\Q, \P) & \ge \E^\Q\left[\int_0^T \E^\Q\left[f(r, X_r, \nu_r^\P)\middle| X_r\right]dr + g(X_T) \right.\\
			& + \left.\frac{1}{2\epsilon}\int_0^T \left|\sigma^{-1}(r, X_r)\left(\beta(r, X_r) -\E^\Q\left[b(r, X_r, \nu_r^\P) \middle| X_r\right]\right)\right|^2dr \right].
		\end{aligned}
	\end{equation}
	Since \eqref{eq:finiteBF} holds, Lemma \ref{lemma:condExpConvex} applied with $(y_t, z_t) = \left(b(t, X_t, \nu_t^\P), f(t, X_t, \nu_t^\P)\right)$
	gives the existence of a function $v \in \shb([0, T] \times \R^d, \U)$ such that for almost all $t \in [0, T]$, $\P$-a.s.
	\begin{equation}
		\label{eq:markovControlJensen}
		\left\{
		\begin{aligned}
			& \E^\Q\left[b(t, X_t, \nu_t^\P)\middle|X_t\right] = b(t, X_t, v(t, X_t))\\
			& \E^\Q\left[f(t, X_t, \nu_t^\P)\middle|X_t\right] \ge f(t, X_t, v(t, X_t)).
		\end{aligned}
		\right.
	\end{equation}
	Injecting \eqref{eq:markovControlJensen} in \eqref{eq:interPointwise2} we get
	$$
	\shj(\Q, \P) \ge \E^\Q\left[\int_0^T f(r, X_r, v(r, X_r))dr + g(X_T)\right] + \frac{1}{2\epsilon}\E^\Q\left[\int_0^T |\sigma^{-1}(r, X_r)(\beta(r, X_r) - b(r, X_r, v(r, X_r)))|^2dr\right].
	$$
	The previous inequality rewrites
	$$
	\shj(\Q, \P) \ge \E^\Q\left[\int_0^T F_\beta(r, X_r, v(r, X_r))dr + g(X_T)\right],
	$$
        where we recall that $F_\beta$ was defined in \eqref{eq:fBeta}.
        By \eqref{eq:argmin},
        for all $t \in [0, T]$ we have
	$$
	F_\beta(t, X_t, v(t, X_t)) \ge F_\beta(t, X_t, u(t, X_t))~\Q\text{-a.s.},
	$$
	hence
	\begin{equation}
		\label{eq:infJQP}
		\shj(\Q, \P) \ge \E^\Q\left[\int_0^T F_\beta(r, X_r, v(r, X_r))dr + g(X_T)\right] \ge \E^\Q\left[\int_0^T F_\beta(r, X_r, u(r, X_r))dr + g(X_T)\right].
	\end{equation}
	In particular,
	$$
	\E^\Q\left[\int_0^T |\sigma^{- 1}(r, X_r)(b(r, X_r, u(t, X_t)) - \beta(r, X_r))|^2dr\right] < + \infty.
	$$
By Remark \ref{rmk:uniqDecomp} the equation  \eqref{eq:decompPRem}
admits a unique solution. Therefore we can apply item $2.$ of Lemma \ref{lemma:girsanovEntropy}  with $\delta(t, X) = b(t, X_t, u(t, X_t))$
and $\gamma(t, X) = \beta(t, X_t)$, and we have
	$$
	H(\Q | \P) = \frac{1}{2}\E^\Q\left[\int_0^T |\sigma^{- 1}(r, X_r)(b(r, X_r, u(t, X_t)) - \beta(r, X_r))|^2dr\right],
	$$
	hence
	$$
	\E^\Q\left[\int_0^T F_\beta(r, X_r, u(r, X_r))dr + g(X_T)\right] = \shj(\Q, \P^u)
	$$
	and previous inequality along with \eqref{eq:infJQP} yields
	$
	\shj(\Q, \P) \ge \shj(\Q, \P^u).
	$
\end{proof}

\subsection{Exponential twist subproblem} \label{sec:markovDrift}

In this section we focus on the minimization $\underset{\Q \in \Pma(\Omega)}{\inf} \shj(\Q, \P)$, $\P \in \Pma_\U^{Markov}$ being the reference probability.
Let us denote $\Q^*$  the solution of
that problem
given by Proposition \ref{prop:markovExistenceMinimizer}.

	\begin{prop}
		\label{prop:markovianDrift}
		Assume that, under $\P$, the canonical process decomposes as
		$
		X_t = x + \int_0^t b(r, X_r, u(r, X_r))dr + M_t^\P,
		$
		where $M^\P$ is a local martingale such that $[ M^\P] = \int_0^{\cdot} \sigma\sigma^\top(r, X_r)dr$ and $u \in \B([0, T] \times \R^d, \U)$. Then there exists $\beta \in \B([0, T] \times \R^d, \R^d)$ such that, under $\Q^*$, the canonical process decomposes as 
 		\begin{equation}\label{eq:betalambda}
		X_t = x +  \int_0^t \beta(r, X_r)dr + M_t^{\Q^*},
		\end{equation}
		where $M^{\Q^*}$ is a local martingale such that $[ M^{\Q^*}] = \int_0^{\cdot} \sigma\sigma^\top(r, X_r)dr$. Moreover, $X$ is a Markov process under $\Q^*$ and $\beta  \in \B([0, T] \times \R^d;\R^d)$
                such that
            $ \vert \beta \vert    \in L^q(dt\otimes d\Q^*)$,
 for all $1 \le q < 2$.
	\end{prop}
        
	\begin{proof}
		Recall that by Remark \ref{rmk:stroockVaradhan}, $\P$ is a solution in law of the SDE
		$$
		dX_t = b(t, X_t, u(t, X_t))dt + \sigma(t, X_t)dW_t, ~X_0 = x.
		$$
                By Corollary 6.12 in \cite{BORMarkov2023}
                with $\Q = \Q^*$,
                there is $\lambda  \in \B([0, T] \times \R^d;\R^d)$
                belonging to 
                $ L^q(dt\otimes d\Q^*)$, for
                all $1 \le q < 2$, such that
\eqref{eq:betalambda} holds with
$$ \beta(t,x) = b(t,x,u(t,x)) + \lambda(t,x).$$
$ (t,\omega) \mapsto b(t,X_t(\omega),u(t,X_t(\omega)
                \in L^q(dt\otimes d\Q^*)$,
                taking into account
\eqref{eq:linearGrowthbSigma}.
%
The result 
follows again by Remark
		\ref{rmk:stroockVaradhan}.
	\end{proof}

\section{Application to the control of thermostatic loads in power systems} \label{sec:example}

\setcounter{equation}{0}

We consider in this section the problem of controlling a large, heterogeneous population of $N$ air-conditioners in order that their overall consumption tracks a given target profile $r = (r_t)_{0\leq t\leq T}$,  on a given time horizon $[0,T]$. This problem was introduced in \cite{FullyBackward}. Air-conditioners are aggregated in $d$ clusters indexed by $1 \le i \le d$, depending on their characteristics. We denote by $N_i$ the number of air-conditioners in the cluster $i$. Individually, the temperature $X^{i, j}$ in the room with air-conditioner $j$ in cluster $i$, is assumed to evolve according to the dynamics
\begin{equation}
	\label{eq:individualTempSDE}
	dX_t^{i, j} = -\theta^i(X_t^{i, j} - x_{out}^i)dt - \kappa^iP_{max}^iu^{i, j}_t dt + \sigma^{i, j}dW_t^{i, j}, ~X_0^{i, j} = x_0^{i, j}, 1 \le i \le d, 1 \le j \le N_i,
\end{equation} 
where  $x_{out}^i$ is the outdoor temperature, $\theta^i$ is a positive thermal constant,
$\kappa^i$ is the heat exchange constant and $P^i_{max}$ is the maximal power consumption of an air-conditioner in cluster $i$. $W^{i, j}$ are independent Brownian motion that represent random temperature fluctuations inside the rooms, such as a window or a door opening. For each cluster, a \textbf{local controller} decides at each time step, to turn $ON$ or $OFF$ some conditioners in the cluster $i$ by setting $u^{i, j} = 1$ or $0$, in order to satisfy a \textbf{prescribed proportion} of active air-conditioners. We are interested in the global planner problem which consists in computing the prescribed proportion $u^i = \frac{1}{N_i}\sum_{j = 1}^{N_i} u^{i, j}$ of air conditioners ON in each cluster in order to track the given target consumption profile $r = (r_t)_{0\leq t\leq T}$. For each $1 \le i \le d$ the average temperature $X^i = \frac{1}{N}\sum_{j = 1}^{N_i} X^{i, j}$ in the cluster $i$ follows the aggregated dynamics
\begin{equation}
	\label{eq:aggregatedTempSDE}
	dX_t^{i} = -\theta^i(X_t^{i} - x_{out}^i)dt - \kappa^iP_{max}^iu^{i}_t dt + \sigma^{i}dW_t^{i}, ~X_0^{i} = x_0^{i},
\end{equation}
with
$$
W_t^i = \frac{1}{N_i}\sum_{j = 1}^{N_i}W_t^{i, j}, ~\sigma^i = \frac{1}{N_i}\sum_{j = 1}^{N_i} \sigma^{i, j} ~\text{and}~x_0^{i} = \frac{1}{N_i}\sum_{j = 1}^{N_i}x_0^{i, j}.
$$
We consider the stochastic control Problem \eqref{eq:controlProblemIntro} on the time horizon $[0,T]$ with $\U = [0, 1]^d$ and $T = 2h$. The running cost $f$ is defined for any $(t, x, u) \in [0, T] \times \R^d \times \U$, such that
\begin{equation}
	\label{eq:runningCostExample}
	f(t, x, u) := \mu\left(\sum_{i=1}^d \rho_i u_i - r_t\right)^2 + \frac{1}{d}\sum_{i=1}^d\left(\gamma_i(\rho_i u_i)^2 + \eta_i (x_i - x_{max}^i)^2_+ + \eta_i (x_{min}^i - x_i)^2_+\right)
,\end{equation}
where $\rho_i = N_iP_{max}^i/(\sum_{j = 1}^d N_jP_{max}^j)$, the first term in the above cost function penalizes the deviation of the the overall consumption $\sum_i \rho_i u^i_t$ with respect to the target consumption $r_t$, $\gamma_i$ quantifies the penalization for irregular controls in cluster $i$, while $\eta_i$ penalizes the exits of the mean temperatures in the cluster $i$, from a comfort band $[x^i_{min}, x_{max}^i]$. Finally the terminal cost is given by $g(x) = \frac{1}{d}\sum_{i = 1}^d |x^i - x_{target}^i|^2$, where $x_{target}^i$ is a target temperature for cluster $i$. Clearly the cost functions $f$ and $g$ satisfy Hypothesis \ref{hyp:costFunctionsControl}.
To estimate an optimal policy $u^*$ for this problem, we use Algorithm \ref{algo:mcEntropy} with a time step $t_{m+1}-t_m=60s$, for $m=0,\cdots M$. The parameters of the problem are the same as in \cite{FullyBackward}. We perform $N_{grid} = 100$ independent runs
of the algorithm, providing $(\hat u^i)_{1 \le i \le N_{grid}}$ estimations of an optimal control on the whole period $t_0,t_1,\cdots t_M$. For each estimation $\hat u^i$, we simulate $N_{simu} = 1000$ iid
trajectories of the process controlled by $\hat u^i$ and compute the associated costs $(\mathcal{J}_\ell (\hat u^i))_{1 \le \ell \le N_{simu}}$. The average cost is finally estimated by $\mathcal{J} = \frac{1}{N_{grid}N_{simu}}\sum_{i = 1}^{N_{grid}}\sum_{\ell = 1}^{N_{simu}}\mathcal{J}_\ell (\hat u^i)$.

To evaluate the performances of our approach, we compare it with the classical regression-based Monte-Carlo technique, relying on 
the dynamic programming principle in \cite{FullyBackward}.
We underline that we only aim to obtain lower costs compared to the BSDE technique in \cite{FullyBackward},
there are no benchmark costs. The results are reported in Table \ref{tab:resultsSimu}, for dimensions $d = 1, ~2, ~5, ~10, ~15, ~20$. For both methods, $N = 10^3, ~10^4, ~5 \times 10^4, ~10^5$ particles are used to estimate an optimal policy for each dimension $d$. For the entropy penalized Monte-Carlo algorithm, we use a penalization parameter $\epsilon = 70$ and $K = 20$ iterations for dimensions $d = 1, ~2, ~5, ~10$ and $\epsilon = 20$ and $K = 60$ iterations for dimensions $d = 15, ~20$. Concerning the approximation in Step 1 of the Algorithm \ref{algo:mcEntropy}, we limit ourselves to the set $\Pma_0$ of polynomials of degree $0$, since the problem is very localized in space.
On Table \ref{tab:resultsSimu} we can observe very good performances, that seem to be weakly sensitive to the dimensions of the problem. On Figure \ref{fig:illustrationAlternate}, we have reported the cost $\shj(\Q_k,\P_k)$ and $\shj(\P_k,\P_k) =\E^{\P_k}\left[\int_0^T f(r, X_r, u^k(r, X_r))dr + g(X_T)\right]$ as a function of the iteration number $k$, obtained
on one run of the algorithm with $d = 20$ and $N =50000$. Theses costs are compared to a reference cost obtained with a run of our algorithm with $N = 100000$ particles. As expected, $\shj(\Q_k,\P_k)$ is decreasing and converging to a limiting value. It is interesting to notice that $\shj(\P_k,\P_k)$ is also decreasing and very close to $\shj(\Q_k,\P_k)$. Hence, it seems that the parameter $\epsilon$ does not need to be as small to obtain a good approximation of the original control Problem \eqref{eq:controlProblemIntro}. 

\begin{table}[H]
	\captionsetup{font=scriptsize}
	\resizebox{\textwidth}{!}{
	\begin{tabular}{c|cc|cc|cc|cc|}
		\cline{2-9}
		& \multicolumn{2}{c|}{$\bm{N = 10^3}$}                        & \multicolumn{2}{c|}{$\bm{N = 10^4}$}                       & \multicolumn{2}{c|}{$\bm{N = 5 \times 10^4}$}              & \multicolumn{2}{c|}{$\bm{N = 10^5}$}                                          \\ \hline
		\multicolumn{1}{|c|}{\bf Method}   & \multicolumn{1}{c|}{\bf{Entropy}}          & \bf{BSDE}            & \multicolumn{1}{c|}{\bf{Entropy}}         & \bf{BSDE}            & \multicolumn{1}{c|}{\bf{Entropy}}         & \bf{BSDE}           & \multicolumn{1}{c|}{\bf{Entropy}}                            & \bf{BSDE}          \\ \hline
		\multicolumn{1}{|c|}{$\bm{d = 1}$}  & \multicolumn{1}{c|}{$7.60 (1e^{-6})$} & $7.61(6e^{-4})$ & \multicolumn{1}{c|}{$7.59(1e^{-6})$} & $7.60(3e^{-4})$ & \multicolumn{1}{c|}{$7.59(1e^{-6})$} & $7.60(3e^{-4})$ & \multicolumn{1}{c|}{$7.59(1e^{-6})$}                    & $7.60(3e^{-4})$ \\
		\multicolumn{1}{|c|}{$\bm{d = 2}$}  & \multicolumn{1}{c|}{$7.82(2e^{-6})$}  & $8.24(7e^{-2})$ & \multicolumn{1}{c|}{$7.79(5e^{-7})$} & $7.77(1e^{-3})$ & \multicolumn{1}{c|}{$7.78(5e^{-7})$} & $7.79(2e^{-4})$ & \multicolumn{1}{c|}{$7.78(5e^{-7})$}                    & $7.78(1e^{-4})$ \\
		\multicolumn{1}{|c|}{$\bm{d = 5}$}  & \multicolumn{1}{c|}{$7.34(2e^{-6})$}  & $14.83(0.64)$   & \multicolumn{1}{c|}{$7.30(5e^{-7})$} & $7.69(6e^{-2})$ & \multicolumn{1}{c|}{$7.30(3e^{-7})$} & $7.28(2e^{-3})$ & \multicolumn{1}{c|}{$7.30(3e^{-7})$} & $7.27(8e^{-4})$ \\
		\multicolumn{1}{|c|}{$\bm{d = 10}$} & \multicolumn{1}{c|}{$5.96(2e^{-6})$}  & $28.14(0.64)$   & \multicolumn{1}{c|}{$5.88(8e^{-7})$} & $16.06(0.38)$   & \multicolumn{1}{c|}{$5.87(5e^{-7})$} & $7.96(0.25)$    & \multicolumn{1}{c|}{$5.87(4e^{-7})$}                    & $6.12(0.08)$    \\
		\multicolumn{1}{|c|}{$\bm{d = 15}$} & \multicolumn{1}{c|}{$9.15(7e^{-5})$}  & $37.91(0.60)$   & \multicolumn{1}{c|}{$8.32(2e^{-5})$} & $32.20(0.63)$   & \multicolumn{1}{c|}{$8.11(5e^{-6})$} & $26.69(0.65)$   & \multicolumn{1}{c|}{$8.08(3e^{-6})$}                    & $22.54(0.56)$   \\
		\multicolumn{1}{|c|}{$\bm{d = 20}$} & \multicolumn{1}{c|}{$8.80(4e^{-5})$}  & $34.83(0.45)$   & \multicolumn{1}{c|}{$7.91(1e^{-5})$} & $30.66(0.59)$   & \multicolumn{1}{c|}{$7.71(3e^{-6})$} & $26.21(0.69)$   & \multicolumn{1}{c|}{$7.68(2e^{-6})$}                    & $23.26(0.59)$   \\ \hline
	\end{tabular}}
	\caption{\label{tab:resultsSimu} Simulated costs (within parenthesis, standard deviation) for the relative entropy penalization scheme and a classical BSDE scheme.}
\end{table}

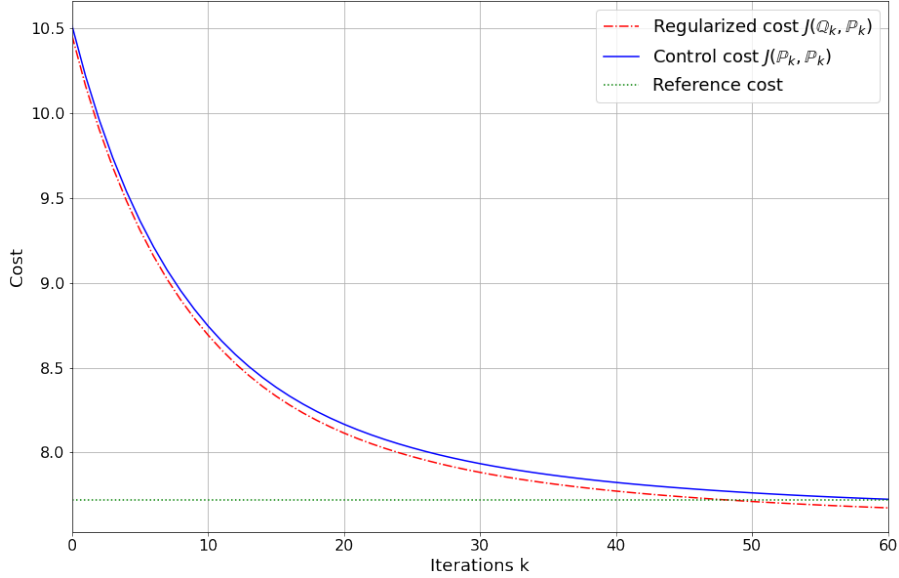
\begin{figure}[H]
	\captionsetup{font=scriptsize}
	\centering
\begin{tikzpicture}

\definecolor{darkgray176}{RGB}{176,176,176}
\definecolor{green01270}{RGB}{0,127,0}
\definecolor{lightgray204}{RGB}{204,204,204}

\begin{axis}[
legend cell align={left},
legend style={fill opacity=0.8, draw opacity=1, text opacity=1, draw=lightgray204},
tick align=outside,
tick pos=left,
x grid style={darkgray176},
xlabel={Iterations k},
xmajorgrids,
xmin=0, xmax=60,
xtick style={color=black},
y grid style={darkgray176},
ylabel={Cost},
ymajorgrids,
ymin=7.5320138, ymax=10.6586582,
ytick style={color=black},
width=13cm,
height=8.7cm
]
\addplot [thick, red, dash pattern=on 1pt off 3pt on 3pt off 3pt]
table {%
0 10.457565
1 10.159737
2 9.9037
3 9.680162
4 9.483708
5 9.309996
6 9.156863
7 9.02008
8 8.898537
9 8.790036
10 8.6922
11 8.604355
12 8.525188
13 8.454178
14 8.390079
15 8.332146
16 8.279795
17 8.232452
18 8.189526
19 8.150406
20 8.114646
21 8.082044
22 8.052214
23 8.024855
24 7.999684
25 7.97637
26 7.95478
27 7.934791
28 7.916261
29 7.899079
30 7.883161
31 7.868419
32 7.854724
33 7.842043
34 7.830177
35 7.819062
36 7.808625
37 7.798831
38 7.789609
39 7.780926
40 7.772792
41 7.765153
42 7.75792
43 7.75108
44 7.744588
45 7.738418
46 7.73254
47 7.726944
48 7.721608
49 7.716513
50 7.71164
51 7.707044
52 7.702682
53 7.698519
54 7.694547
55 7.69074
56 7.68711
57 7.683641
58 7.680325
59 7.677165
60 7.674134
};
\addlegendentry{Regularized cost $\mathcal{J}(\mathbb{Q}_k, \mathbb{P}_k)$}
\addplot [thick, blue]
table {%
0 10.516538
1 10.217568
2 9.960647
3 9.736388
4 9.539332
5 9.365114
6 9.211569
7 9.074421
8 8.952563
9 8.843786
10 8.745694
11 8.657619
12 8.578245
13 8.507057
14 8.4428
15 8.384725
16 8.332247
17 8.284786
18 8.241752
19 8.202535
20 8.166683
21 8.133998
22 8.10409
23 8.076655
24 8.051415
25 8.028031
26 8.006377
27 7.986324
28 7.967735
29 7.950496
30 7.934527
31 7.919738
32 7.906001
33 7.893284
34 7.881384
35 7.870236
36 7.859769
37 7.849945
38 7.840694
39 7.831983
40 7.823825
41 7.816162
42 7.808906
43 7.802044
44 7.795531
45 7.789339
46 7.783439
47 7.777822
48 7.772465
49 7.767349
50 7.762456
51 7.757842
52 7.753464
53 7.749285
54 7.745298
55 7.741477
56 7.737832
57 7.73435
58 7.73102
59 7.727848
60 7.724804
};
\addlegendentry{Controlled cost $\mathcal{J}(\mathbb{P}_k, \mathbb{P}_k)$}
\addplot [thick, green01270, dotted]
table {%
0 7.72
60 7.72
};
\addlegendentry{Optimal cost}
\end{axis}

\end{tikzpicture}
	\caption{Costs associated with the iterates, generated by the entropy penalized Monte-Carlo algorithm in dimension $d = 20$ with $N = 50000$.}
	\label{fig:illustrationAlternate}
\end{figure}

\section{Conclusion and perspectives}
\label{S6}

\setcounter{equation}{0}
In this paper we have proposed an original approach to treat
stochastic optimal control problems, regarded  as optimization programs
on the space of probability measures,
based on an entropy penalized formulation. In particular this has allowed us to design an alternating
minimization procedure to tackle those problems.
One additional interest of this entropy penalized formulation is that it can be naturally extended to treat control problems with more complex constraints
of the form
\begin{equation}
	\label{eq:genProblem}
	\inf_{\P \in \sha \cap \shb} \E^\P\left[\int_0^T f(r, X_r, \nu_r^\P)dr + g(X_T)\right],
\end{equation}
with a general admissible set of the form $\sha \cap \shb$ where
$\sha$ is a convex subset of $\shp(\Omega)$ and
 $\shb$ is a subset of $\shp(\Omega),$
describing a class of controlled dynamics, which fulfills some technical conditions.
A typical example appears when
$ \sha = \{ \P \in \shp(\Omega) : \P_T = \mu_T\},$
where $\mu_T$ is a prescribed (terminal) law
and $\shb$ imposes an initial law $\mu_0$.
Problem \eqref{eq:genProblem}
then corresponds in this example to a stochastic control problem with prescribed
initial and terminal distributions, typically encountered in the fields of martingale optimal transport or Schr\"odinger Bridge problems. 
We remark that this formulation  covers in particular the one of the present paper
setting $\sha =  \shp(\Omega)$ and $\shb = \shp_\U$.

The idea is then to extend the splitting approach of our entropy penalized method,
leading us to two simpler subproblems,
each one taking into account separately the constraints sets $\sha$ and $\shb$.
This is the object of a paper in preparation.


\appendix
\section*{Appendices}

\label{appendix}

\section{Relative entropy related results}
\setcounter{equation}{0}
\renewcommand\theequation{A.\arabic{equation}}
Let $(\Omega, \F, (\F_t)_{t \in [0, T]}, \P)$ be a filtered probability
	space.  Let $\delta = (\delta_t)_{t \in [0,T]}$ (resp. $a = (a_t)_{t \in [0,T]}$)
	be a progressively measurable process with values in $\R^d$
	(resp. in the set of square $d \times d$
	non-negative defined
	symmetric matrices $S_d^+$).
	Let $X$ be a continuous process, which decomposes as
	\begin{equation}  \label{eq:decXP}
		X_t = x + \int_0^t \delta_r dr + M^\P_t, ~0 \le t \le T,
	\end{equation}
	where $M^\P$ is a continuous $((\F_t), \P)$-local martingale such that $[ M^\P ] = \int_0^{\cdot} a_rdr$.

The theorem below is the Girsanov's theorem under a finite
relative entropy assumption.
\begin{theorem}
  \label{th:girsanovEntropy}
        Let $\Q$ be a probability measure on $(\Omega, \F)$ such that $H(\Q | \P) < + \infty$.
 Then there exists an $\R^d$-valued
	progressively measurable process $\alpha$ such that
	\begin{equation} \label{eq:alphaInt}
		\E^{\Q}\left[\int_0^T \alpha_r^\top a_r \alpha_rdr\right] < + \infty,
	\end{equation}
	and such that, under $\Q$, the process $X$ is still a continuous semimartingale with decomposition
	\begin{equation} \label{eq:alpha}
		X_t = x + \int_0^t \delta_rdr + \int_0^t a_r\alpha_rdr + M^\Q_t, ~0 \le t \le T,
	\end{equation}
	where $M^\Q$ is a continuous $\Q$-local martingale and $[ M^\Q ] = \int_0^{\cdot} a_rdr$. Furthermore,
	\begin{equation}
		\label{eq:subBoundEntropy}
		\frac{1}{2}\E^{\Q}\left[\int_0^T \alpha_r^\top a_r \alpha_rdr\right] \le H(\Q | \P).
	\end{equation}
\end{theorem}
\begin{proof}
  The fact that $H(\Q | \P)  < + \infty$ implies in particular that $\Q \ll \P$. Let then $Z_T := d\Q/d\P$ and $(Z_t)_{t \in [0, T]}$ be
  the c\`adl\`ag $\P$-modification of the martingale $\left(\E^\P\left[Z_T | \shf_t\right]\right)_{t \in [0, T]}$. By Theorem 3.24, Chapter III in \cite{JacodShiryaev}, there exists a progressively measurable  process $\alpha$ such that decomposition \eqref{eq:alpha} holds and
	\begin{equation}
		\label{eq:QFinite}
		\int_0^\cdot \alpha_r^\top a_r \alpha_r dr < + \infty \quad \Q\text{-a.s.}
              \end{equation}
	as well as
	\begin{equation}
		\label{eq:bracketQ}
		[Z, M^\P] = \int_0^\cdot a_r\alpha_rZ_{r-}dr,
              \end{equation}
              with respect to $\P$, so also with respect to $\Q$.
            
              Let then  $\tau_k := \inf\left\{t \in [0, T]~:~\int_0^t\alpha_r^\top a_r \alpha_rdr > k\right\}$, with the convention that $\inf \emptyset = + \infty$.
              Setting $M^k := \int_0^{.\wedge \tau_k}\alpha_r^\top dM^\P_r$ and $Z^k$ the Doléans exponential $ \mathcal{E}(M^k)$, we define $d\Q_k := Z^k_Td\P$. By Novikov's criterion (see Proposition 1.15, Chapter VIII in \cite{RevuzYorBook}),
              $Z^k$ is a martingale, therefore $\Q_k$ is a probability measure on $(\Omega, \shf)$ equivalent to $\P$ since $Z^k_T$ is strictly positive $\P$-a.s. As $\Q \ll \P$ and $\Q_k \sim \P$, we have $\Q \ll \Q_k$.
              It follows that $\P$-a.s., with the notation $\log (0) = -\infty$,
              and later $0 \log (0) = 0$.
	\begin{equation}
		\label{eq:logDensity}
		\begin{aligned}
			\log \frac{d\Q}{d\P} & = \log \frac{d\Q}{d\Q_k} + \log \frac{d\Q_k}{d\P}\\
			& = \log \frac{d\Q}{d\Q_k} + \log Z_T^k\\
			& = \log \frac{d\Q}{d\Q_k} + \int_0^{T \wedge \tau_k}\alpha^\top_rdM_r^\P - \frac{1}{2}\int_0^{T \wedge \tau_k}\alpha_r^\top a_r \alpha_r dr.
		\end{aligned}
              \end{equation}
              Previous equality can be of course considered also $\Q$-a.s. since
              $\Q$ is ''rougher'' than $\P$.
Setting $\bar M := M^\P - \int_0^\cdot a_r\alpha_rdr$,
equality  \eqref{eq:logDensity} rewrites
\begin{equation}
		\label{eq:logDensity2}
		\log \frac{d\Q}{d\P} = \log \frac{d\Q}{d\Q_k} + \int_0^{T \wedge \tau_k}\alpha^\top_rd\bar M_r + \frac{1}{2}\int_0^{T \wedge \tau_k}\alpha_r^\top a_r \alpha_r dr \quad \Q\text{-a.s.}
              \end{equation}
              Taking into account  \eqref{eq:bracketQ},
              Theorem 3.11, Chapter III in \cite{JacodShiryaev} states that, $\bar M$ is a
              $\Q$-local martingale.
              %
              Since,
              still with respect to $\Q$,
$\left[\int_0^\cdot \alpha_r^\top d\bar M_r\right] = \int_0^\cdot \alpha_r^\top a_r \alpha_rdr$, by definition of $\tau_k$, the process $\int_0^{\cdot \wedge \tau_k}\alpha_r^\top d\bar M_r$ is a genuine $\Q$-martingale. Consequently,
taking the expectation under $\Q$ in \eqref{eq:logDensity2} gives
	$$
	H(\Q|\P) = H(\Q|\Q_k) + \frac{1}{2}\E^\Q\left[\int_0^{T \wedge \tau_k}\alpha^\top a_r \alpha_rdr\right] \ge \frac{1}{2}\E^\Q\left[\int_0^{T \wedge \tau_k}\alpha^\top a_r \alpha_rdr\right].
	$$
	Since $\tau_k \underset{k \rightarrow + \infty}{\longrightarrow} + \infty$ increasingly $\Q$-a.s. by \eqref{eq:QFinite}, a direct application of the monotone convergence theorem then yields
	\begin{equation*}
		H(\Q|\P) \ge \frac{1}{2}\E^{\Q}\left[\int_0^T \alpha_r^\top a_r \alpha_rdr\right].
	\end{equation*}
\end{proof}
For the following lemma let again $X$ be a process, as at the beginning of the section
fulfilling
\eqref{eq:decXP}, this time
with $a_t = \sigma\sigma^\top(t, X_t)$.
Then by Theorem \ref{th:girsanovEntropy} there is
a progressively measurable process $\alpha$
such that \eqref{eq:alpha} holds.
For that we have the following estimates.
\begin{lemma}
  \label{lemma:estimateBetaEntrop}
 	We suppose the existence of $1 < p < 2$
	such that
	\begin{equation*} 
	C_p := \E^{\P}\left[\int_0^T \|\sigma(r, X_r)\|^{2p/(2 - p)}dr\right] < + \infty.
	\end{equation*}
 Let $\Q$ be a probability measure on $(\Omega, \F)$ such  that $H(\Q | \P) < + \infty$.

        \begin{enumerate}
		\item If $C_\infty := \|d\Q/d\P\|_\infty < + \infty$, there exists a constant $L > 0$, which depends only on $C_p$ and $C_\infty$, such that
		\begin{equation} \label{eq:L2bis}
		\E^{\Q}\left[\int_0^T |\sigma\sigma^\top(r, X_r)\alpha_r|^pdr\right] \le L(1 + H(\Q | \P)).
		\end{equation}
		\item Suppose
		moreover $H(\P | \Q) < + \infty$. Then it holds that
		\begin{equation} \label{eq:reverseEntropy}
		\frac{1}{2}\E^{\P}\left[\int_0^T |\sigma^\top(r, X_r)\alpha_r|^2 dr\right] \le H(\P |\Q),
		\end{equation}
		and $L$ can be chosen such that both \eqref{eq:L2bis} and  
		\begin{equation} \label{eq:L2}
			\E^{\P}\left[\int_0^T |\sigma\sigma^\top(r, X_r)\alpha_r|^pdr\right] \le L(1 + H(\P | \Q)).
		\end{equation}
	\end{enumerate}
\end{lemma}
\begin{proof}
	\begin{enumerate}
		\item We recall that $H(\Q |\P) < \infty.$
		By H\"older's inequality applied on the measure space $([0, T] \times \Omega, \B([0, T])\otimes \F, dt\otimes d\Q)$, it holds that
		\begin{equation}
			\label{eq:ineqLambdaIntermediate}
			\begin{aligned}
				\E^{\Q}\left[\int_0^T |\sigma\sigma^\top(r, X_r)\alpha_r|^pdr\right] & \le 	\E^{\Q}\left[\int_0^T \|\sigma(r, X_r)\|^p|\sigma^\top(r, X_r)\alpha_r|^pdr\right]\\
				& \le \left(\E^{\Q}\left[\int_0^T \|\sigma(r, X_r)\|^{2p/(2 - p)}\right]\right)^{1 - p/2}\left(\E^{\Q}\left[\int_0^T |\sigma^\top(r, X_r)\alpha_r|^2dr\right]\right)^{p/2}.
			\end{aligned}
		\end{equation}
		On the one hand,
		\begin{equation}
			\label{eq:ineqLambda}
			\E^{\Q}\left[\int_0^T \|\sigma(r, X_r)\|^{2p/(p - 2)}dr\right] = \E^{\P}\left[\frac{d\Q}{d\P}\int_0^T \|\sigma(r, X_r)\|^{2p/(p - 2)}dr\right] \le C_\infty C_p.
		\end{equation}
		On the other hand, by \eqref{eq:subBoundEntropy} in Theorem \ref{th:girsanovEntropy},
		\begin{equation}
			\label{eq:ineqEntrop}
			\E^{\Q}\left[\int_0^T |\sigma^\top(r, X_r)\alpha_r|^2dr\right] \le 2H(\Q | \P).
		\end{equation}
		Combining \eqref{eq:ineqLambda} and \eqref{eq:ineqEntrop} with \eqref{eq:ineqLambdaIntermediate}, we get
		$$
		\E^{\Q}\left[\int_0^T |\sigma\sigma^\top(r, X_r)\alpha_r|^pdr\right] \le 2^{p/2}(C_\infty C_p)^{1 - p/2}H(\Q | \P)^{p/2},
		$$
		and as $p < 2$, using the inequality
		\begin{equation} \label{eq:aq}
			\vert a\vert^q \le (1 + \vert a \vert), \
			{\rm if} \ q \in ]0,1],
		\end{equation}
		with $q = 1 - \frac{p}{2}$ and $q = \frac{p}{2}$, we have
		$$
		\E^{\Q}\left[\int_0^T |\sigma\sigma^\top(r, X_r)\alpha_r|^pdr\right] \le
		2 (1 + C_\infty C_p)(1 + H(\Q | \P)).
		$$
		Setting
\begin{equation}\label{eq:L}
  L := 2 (1 + C_p(C_\infty \vee 1)),
\end{equation}
  one concludes the proof of item 1.

              \item 
Applying Theorem \ref{th:girsanovEntropy}, we recall the decomposition
  \eqref{eq:alpha},
  where the local martingale $M^\Q$ verifies $[ M^\Q] = \int_0^\cdot \sigma\sigma^\top(r, X_r)dr$,
  under $\Q$.
		As $H(\P | \Q) < + \infty$, interchanging $\P$ and $\Q$, 
 again  Theorem \ref{th:girsanovEntropy}  yields the existence of a progressively measurable
		process $\tilde \alpha$ such that under $\P$ the process $X$ decomposes as
                \begin{equation} \label{eq:P}
                  X_t = x + \int_0^t \delta_rdr + \int_0^t \sigma\sigma^\top(r, X_r)\alpha_rdr + \int_0^t \sigma\sigma^\top(r, X_r)\tilde \alpha_rdr + \tilde M_t,
                  \end{equation}
		where $\tilde M$ is a $\P$-local martingale such that $[\tilde M] = \int_0^\cdot \sigma\sigma^\top(r, X_r)dr$
		and
		$$
		\frac{1}{2}\E^{\P}\left[\int_0^T |\sigma^\top(r, X_r)\tilde \alpha_r|^2dr\right] \le H(\P | \Q).
		$$
		Identifying the bounded variation and the martingale
                components of $X$ under $\P$, in \eqref{eq:P} and \eqref{eq:decXP},
                we get that $\tilde M = M^\P$ and $\sigma\sigma^\top(t, X_t)\tilde \alpha_t = - \sigma\sigma^\top(t, X_t) \alpha_t$ $dt \otimes d\P$-a.e. In particular, 
		\eqref{eq:reverseEntropy} holds.
		Then,		as in the proof of item $1.$, H\"older's inequality,
		\eqref{eq:ineqLambdaIntermediate} with $\Q$ replaced by $\P$,   
		and  \eqref{eq:reverseEntropy}
		yield
		\begin{equation*}
			\begin{aligned}
				\E^{\P}\left[\int_0^T |\sigma\sigma^\top(r, X_r)\alpha_r|^pdr\right] & \le \left(\E^{\P}\left[\int_0^T \|\sigma(r, X_r)\|^{2p/(2 - p)}\right]\right)^{1 - p/2}\left(\E^{\P}\left[\int_0^T |\sigma^\top(r, X_r)\alpha_r|^2dr\right]\right)^{p/2}\\
				& \le 2^{p/2}C_p^{1 - p/2}H(\P | \Q)^{p/2} \le 2 (1 + C_p)(1 + H(\P | \Q)),
			\end{aligned}
		\end{equation*}
		where, for the latter inequality we have used again \eqref{eq:aq}
                with $q = 1 - \frac{p}{2}$ and $q =\frac{p}{2}$
                together with  \eqref{eq:reverseEntropy}.
		This finally  also implies the result \eqref{eq:L2}
                with $L$ defined in \eqref{eq:L}.
              \end{enumerate}
\end{proof}
\begin{remark}
	\label{rmk:integrabilityQ}
	Let
	$
	\tilde C_p := \E^{\Q}\left[\int_0^T \|\sigma(r, X_r)\|^{2p/(2 - p)}dr\right].
	$
	Item $1.$ of Lemma \ref{lemma:estimateBetaEntrop} is still valid if one assumes that $\tilde C_p < + \infty$ instead of $C_p < + \infty$ and $\|d\Q/d\P\|_\infty$. One only has to replace $C_\infty C_p$ by $\tilde C_p$ in the estimates in the proof.
\end{remark}
The results of Theorem \ref{th:girsanovEntropy} can be specified if one considers probability measures on the canonical space $\Omega = C([0, T], \R^d)$. In the following, $\delta, \gamma : [0, T] \times C([0, T], \R^d)  \mapsto \R^d$ are progressively measurable functions
w.r.t. their corresponding Borel $\sigma$-fields.
A consequence of Theorem \ref{th:girsanovEntropy} in this setting is the following.
\begin{lemma}
	\label{lemma:girsanovEntropy}
	Let $\P \in \Pma(\Omega)$ such that, under $\P$ the canonical process
        can be decomposed as
	\begin{equation}
		\label{eq:lemmaSDE}
		X_t = x + \int_0^t\delta(r, X)dr + M_t^\P,
	\end{equation}
	where $M^\P$ is a martingale with
        $[ M^\P ] = \int_0^{\cdot} \sigma\sigma^\top(r, X_r)dr$, where $\sigma$ verifies item
        $(ii)$ of Hypothesis \ref{hyp:coefDiffusion}.
        Let $\Q \in \Pma(\Omega)$.
        \begin{enumerate}
        \item 
          Assume that $H(\Q |\P) < + \infty.$
          Then we have the following.
          \begin{enumerate}
        \item
      There exists a progressively measurable
      process $\alpha$,
      w.r.t. the natural filtration of $X$ (in particular
      of the form $\alpha = \alpha(\cdot, X)$) such that, under
      $\Q$, 
           $X$ decomposes as
\begin{equation}\label{eq:841}
		X_t = x + \int_0^t \delta(r, X)dr + \int_0^t\sigma\sigma^\top(r, X_r) \alpha(r, X)dr + M_t^\Q,
	\end{equation}
		where $M^\Q$ is a martingale with
                $[ M^\Q] = \int_0^\cdot \sigma\sigma^\top(r, X_r)dr$
        and 
		\begin{equation}
			\label{eq:inequalityEntropy}
			H(\Q | \P) \ge \frac{1}{2}\E^\Q\left[\int_0^T |\sigma^\top(r, X_r)\alpha(r, X)|^2dr\right].
                      \end{equation}
\item
      	If moreover uniqueness in law holds for the SDE \eqref{eq:lemmaSDE}, equality holds in \eqref{eq:inequalityEntropy}.
      \end{enumerate}
      \item Assume that under $\Q$ the canonical process writes
    \begin{equation} \label{eq:842}
		X_t = x + \int_0^t\gamma(r, X)dr + M_t^\Q,
\end{equation}
where $M^\Q$ is a martingale with $[ M^\Q ] = \int_0^{\cdot} \sigma\sigma^\top(r, X_r)dr$ and that uniqueness in law holds for the SDE \eqref{eq:lemmaSDE}.
Let $\sigma^{-1}$ be again the generalized
                      right-inverse of $\sigma$.
If 
		$$
		\E^{\Q}\left[\int_0^T |\sigma^{- 1}(r, X_r)(\delta(r, X) - \gamma(r, X))|^2dr\right] < + \infty,
		$$
                then	$H(\Q | \P) < + \infty$ and
	\begin{equation}
			\label{eq:equalityEntropy}
			H(\Q|\P) = \frac{1}{2}\E^{\Q}\left[\int_0^T |\sigma^{- 1}(r, X_r)(\delta(r, X) - \gamma(r, X))|^2dr\right].
		\end{equation}
                
	\end{enumerate}
\end{lemma}
\begin{proof}
  Part $(a)$ of item $1.$ of Lemma  \ref{eq:lemmaSDE} is constituted by
  Theorem \ref{th:girsanovEntropy} applied to the canonical space equipped
  with the natural filtration of
  the canonical process.
  Item 2. is the object of
  Lemma 4.4 $(iii)$ in \cite{LackerHierarchies}.

  As far as item $1.(b)$ is concerned, we apply item $2.$ 
  with  $\gamma(r,X) = \delta(r, X) + \sigma\sigma^\top(r, X_r) \alpha(r, X)$
  in \eqref{eq:842} so that
  $(\gamma-\delta)(r,X) = \sigma\sigma^\top(r, X_r) \alpha(r, X)$.
  So $\sigma^{-1}(r,X_r) (\delta-\gamma)(r,X)$
  and the equality in \eqref{eq:inequalityEntropy} holds because of
    \eqref{eq:equalityEntropy}.
\end{proof}
\begin{remark} \label{remark:R81}
  By
  Hypothesis \ref{hyp:coefDiffusion} on the diffusion coefficient $\sigma$, uniqueness in law for the SDE \eqref{eq:lemmaSDE} holds e.g. if $\delta$ is bounded, or if $\delta(r, X) = b(r, X_r, u(r, X_r))$, where $b$ has linear growth in $(t, x)$ independently of $u$. This follows from Theorem 10.1.3 of \cite{stroock}.
\end{remark}

\section{Measurable selection}
\setcounter{equation}{0}
\renewcommand\theequation{B.\arabic{equation}}

The following measurable selection theorem is a direct consequence of
Theorem A.9 in \cite{HaussmanLepeltierOptimal},
setting $y= (t,x),  \phi = b, i = 1, \psi_1 = f$. 
\begin{theorem}
  \label{th:measurableSelection}
  Suppose the validity of item $1.$ of Hypothesis \ref{hyp:coefDiffusion}
  and item $2.$ of Hypothesis \ref{hyp:costFunctionsControl}.
  Let $K(t, x)$ be given by \eqref{eq:Ktx}.
Let $y \in \B([0, T] \times \R^d, \R^d)$ and $z \in \B([0, T] \times \R^d, \R)$ be two functions such that $(y(t, x), z(t, x)) \in K(t, x)$ for all $(t, x) \in [0, T] \times \R^d$.
  Then there exists a
        (measurable) function $u \in \B([0, T] \times \R^d, \U)$ such that
	\begin{equation*}
		\left\{
		\begin{aligned}
			& y(t, x) = b(t, x, u(t, x))\\
			& z(t, x) \ge f(t, x, u(t, x))
		\end{aligned}
		\right.~\text{for all}~(t, x) \in [0, T] \times \R^d.
	\end{equation*}
\end{theorem}
The result below is a simple consequence of Theorem \ref{th:measurableSelection}

\begin{lemma}
	\label{lemma:condExpConvex}
	Let $\Omega$ be a Polish space, and let $\shf := \shb(\Omega)$ be its Borel $\sigma$-field. Let $X : [0, T] \times \Omega \rightarrow \R^d$ and $(y, z) : [0, T] \times \Omega \rightarrow \R^{d + 1}$ be two processes on $(\Omega, \shf)$. Let $\P$ be a probability measure on $(\Omega, \shf)$.  Assume Hypothesis \ref{hyp:convexSet} and that $\E^\P\left[\int_0^T|y_r|dr\right] < + \infty$ and $\E^\P\left[\int_0^T |z_r| dr\right] < + \infty$. Assume moreover that $(y_t, z_t) \in K(t, X_t)$ for almost all $t \in [0, T]$, $\P$-a.s. Then there exists a function $u \in \shb([0, T] \times \R^d, \U)$ such that for almost all $t \in [0, T]$, $\P$-a.s.,
	\begin{equation}
		\label{eq:fromYZToU}
		\left\{
		\begin{aligned}
			& \E^\P[y_t | X_t] = b(t, X_t, u(t, X_t))\\
			& \E^\P[z_t | X_t] \ge f(t, X_t, u(t, X_t)).
		\end{aligned}
		\right.
	\end{equation}
\end{lemma}
\begin{proof}
  \begin{enumerate}
  \item We set $\Phi_t:= (y_t,z_t)$ which belongs a.s. to $K(t,X_t)$.
    We prove below that, for almost all $t \in [0, T]$, 
	\begin{equation}
		\label{eq:condExpInKtx}
\E^\P([\Phi_t | X_t]) \in K(t, X_t) \quad \P\text{-a.s.}
	\end{equation}
	Indeed, let $t \in [0, T]$ such that $(y_t, z_t) \in K(t, X_t)$ $\P$-a.s.
        We set $\mu := \shl^\P(X_t)$. By Theorem 1.1.6 and Theorem 1.1.8 in \cite{stroock} there exists a measurable family $(\P_x)_{x \in \R^d}$ of probability measures on $(\Omega, \shf)$ such that $\P_x(X_t = x) = 1$ for $\mu$-almost all $x \in \R^d$ and $\P = \int_{\R^d}\P_x\mu(dx)$. On the one hand, since $\P(\Phi_t \in K(t, X_t)) = 1$,
	$$
	1 = \P(\Phi_t \in K(t, X_t)) = \int_{\R^d}\P_x(\Phi_t \in K(t, X_t))\mu(dx) =  \int_{\R^d}\P_x(\Phi_t \in K(t, x))\mu(dx),
	$$
	hence $\P_x(\Phi_t \in K(t, x)) = 1$ for $\mu$-almost all $x \in \R^d$.
        Consequently, since $K(t,x)$ is a convex closed set, by Theorem 1 in \cite{ConvCondExp},
\begin{equation}
		\label{eq:PxKtx}
		\P_x\left(\E^{\P_x}[\Phi_t | X_t] \in K(t, x)\right) = \P_x\left(\E^{\P_x}[\Phi_t] \in K(t, x)\right) = 1.
	\end{equation}
	On the other hand, by definition of the conditional expectation, $\E^\P[\Phi_t | X_t] = \left(\E^{\P_x}[\Phi]\right)\circ X_t$. Consequently,
	\begin{equation*}
		\begin{aligned}
			\P\left(\E^\P[\Phi_t | X_t] \in K(t, X_t)\right) & = \int_{\R^d}\P_x\left(\E^{\P}[\Phi_t | X_t] \in K(t, X_t)\right)\mu(dx)\\
			& = \int_{\R^d}\P_x\left(\E^{\P_x}[\Phi_t] \in K(t, x)\right)\mu(dx)\\
			& = 1,
		\end{aligned}
	\end{equation*}
	where we used \eqref{eq:PxKtx} to conclude. \eqref{eq:condExpInKtx} is proved.
	
	\item
	It remains to prove \eqref{eq:fromYZToU}. Proposition 5.1 in \cite{MimickingItoGeneral} provides two measurable functions $\shy, \shz$ such that for all $t \in [0, T]$, $\P$-a.s.
	\begin{equation}
		\label{eq:defYZ}
		\left(\E^\P\left[y_t\middle|X_t\right], \E^\P\left[z_t\middle|X_t\right]\right) = (\shy(t, X_t), \shz(t, X_t)).
	\end{equation}
	Let then
	$$
	N := \{(t, x) \in [0, T] \times \R^d~|~(\shy(t, x), \shz(t, x)) \notin K(t, x)\}.
	$$
	The set $N$ is a Borel set,
	and we now modify the functions $\shy$ and $\shz$ on $N$ and obtain two Borel functions $\hat \shy, \hat \shz$ defined by
	\begin{equation}
		\left\{
		\begin{aligned}
			& (\hat \shy(t, x), \hat \shz(t, x)) = (\shy(t, x), \shz(t, x))~\text{if}~(t, x) \notin N\\
			& (\hat \shy(t, x), \hat \shz(t, x)) = (b(t, x, u_0), f(t, x, u_0))~\text{if}~(t, x) \in N,
		\end{aligned}
		\right.
	\end{equation}
	where $u_0 \in \U$ is fixed. In particular, $(\hat \shy(t, x), \hat \shz(t, x)) \in K(t, x)$ for all $(t, x) \in [0, T] \times \R^d$. Then by Theorem \ref{th:measurableSelection}there exists a Borel function $u \in \shb([0, T] \times \R^d, \U)$ such that 
	\begin{equation}
		\label{eq:measurableSelecU}
		\left\{
		\begin{aligned}
			&\hat \shy(t, x) = b(t, x, u(t, x))\\
			&\hat \shz(t, x) \ge f(t, x, u(t, x))
		\end{aligned}
		\right.~ \text{for all}~(t, x) \in [0, T] \times \R^d.
	\end{equation}
	Combining \eqref{eq:condExpInKtx}, \eqref{eq:defYZ} and \eqref{eq:measurableSelecU} yields \eqref{eq:fromYZToU} for almost all $t \in [0, T]$, $\P$-a.s.
        \end{enumerate}
\end{proof}

\section{Proof of Theorem \ref{th:existenceSolutionRegProb}} \label{app:proofThEx}
\setcounter{equation}{0}
\renewcommand\theequation{C.\arabic{equation}}

  To simplify the formalism of the proof we will assume that $\epsilon = 1$ and $g = 0$.
  In the whole section, we can choose $1 \le p < p'$ as power constants
  appearing in Hypotheses
  \ref{hyp:costFunctionsControl}. We start by some definitions.
\begin{definition}
	\label{def:wasserstein}
	(Wasserstein space).
	Let $(E, d)$ be a metric space. We denote $\shp^p(E)$ the set of probability measures $\P \in \shp(E)$ such that $\int_{E}(d(x, x_0))^p\P(dx) < + \infty$ for some (and thus for any) $x_0 \in E$. We endow $\shp^p(E)$ with the Wasserstein metric
	\begin{equation}
		\label{eq:wassersteinDistance}
		d_p(\P, \Q) := \inf\left\{\int_{E \times E} (d(x, y))^p\rho(dx, dy) ~:~\rho \in \shp(E \times E),~\rho(\cdot \times E) = \P, ~\rho(E \times \cdot) = \Q\right\}^{1/p}.
	\end{equation}
\end{definition}
\begin{definition}
	\label{def:relaxedControl}
	(Relaxed controls).
	We denote $\shv$ the set of relaxed controls, that is the set of non-negative measures $q$ on $[0, T] \times \U$ such that we have the following.
	\begin{enumerate}
		\item $q(\cdot \times \U)$ is the Lebesgue measure on $[0, T]$, and $q([0, T] \times \cdot)/T$ is a probability measure on $(\U, \shb(\U))$.
		\item $\int_{[0, T] \times \U}|u|^pq(dr, du) < + \infty$.
	\end{enumerate}
	The space $\shv$ is endowed with the distance $d_\shv(q_1, q_2) := d_p(q_1/T, q_2/T)$ where $d_p$ is given by \eqref{eq:wassersteinDistance}.
\end{definition}
\begin{definition}
	\label{def:relaxedSpace}
	(Extended space).
	Let $\bar \Omega := C([0, T], \R^d) \times \shv$ and we denote $(X, \Lambda)$ its canonical process.
        The space $\bar \Omega$ is endowed with the filtration $(\bar \shf_t)_{t \in [0, T]}$ defined for all
        $t \in [0, T]$ by $\bar \shf_t := \shf_t^X \otimes \shf_t^\Lambda$ where
        $\shf_t^X := \sigma(X_r, 0 \le r \le t)$ and $\shf_t^\Lambda := \sigma(\Lambda(A), A \in \shb([0, t] \times \U))$. $\bar \Omega$ is equipped with the distance $d_{\bar \Omega}$ given by $d_{\bar \Omega}((x_1, q_1), (x_2, q_2)) := |x_1 - x_2|_\infty + d_{\shv}(q_1, q_2)$.
\end{definition}
\begin{definition}
	\label{def:relaxedSet}
	(Relaxed admissible set).
	Let $\bar \sha$ be the subset of $(\shp(\bar \Omega))^2$ such that $(\bar \P, \bar \Q) \in \bar \sha$ if the following holds.
	\begin{enumerate}
		\item $H(\bar \Q | \bar \P) < + \infty$.
		\item Under $\bar \P$ the process $X$ decomposes as
		\begin{equation}
			\label{eq:decompRelaxed}
			X_t = x + \int_{[0, t] \times \U} b(r, X_r, u)\Lambda(dr, du) + M_t^{\bar \P},
		\end{equation}
		where $M^{\bar \P}$ is a $(\bar \shf_t)$-local martingale verifying $[ M^{\bar \P}] = \int_0^\cdot \sigma\sigma^\top(r, X_r)dr$.
	\end{enumerate}
      \end{definition}
     We will denote $\bar \shp_\U$ the set of elements of $\shp(\bar \Omega)$ such that decomposition \eqref{eq:decompRelaxed} holds.

For $(\bar \P, \bar \Q) \in \bar \sha$ we introduce a relaxed problem defined by
\begin{equation}
  \label{eq:relaxedProblem}
	\bar \shj^* := \inf_{(\bar \P, \bar \Q) \in \bar \sha} \bar \shj(\bar \Q, \bar \P) \quad \text{where} \quad \bar \shj(\bar \Q, \bar \P) := \E^{\bar \Q}\left[\int_{[0, T] \times \U}f(r, X_r, u)\Lambda(dr, du) + g(X_T)\right] + H(\bar \Q | \bar \P).
\end{equation}
\begin{remark}
	\begin{enumerate}
        \item The notion of relaxed control in Definition \ref{def:relaxedControl} extends the notion of (strict) control $\nu = \nu^\P$
          as introduced in Definition \ref{def:PU}. Indeed, a
          control $\nu : [0, T] \times \Omega \rightarrow \U$ induces a measure on $[0, T] \times \U$ by setting $q^\nu := dt\delta_{\nu}(du) \in \shv$.
          
          \item The set of relaxed controls has two main advantages : it is convex and there exist very convenient tightness criteria to identify its precompact sets using Prokhorov's theorem. This allows to easily prove the existence of a solution to the relaxed Problem \eqref{eq:relaxedProblem}. Under the convexity Hypothesis \ref{hyp:convexSet}, it is then possible to deduce the existence of a solution to the original Problem \eqref{eq:penalizedProblemIntro}.  
	\end{enumerate}
\end{remark}
The strategy of the proof of Theorem \ref{th:existenceSolutionRegProb} is the following. We first prove in Proposition \ref{prop:relaxedSolution} that Problem \eqref{eq:relaxedProblem} admits a solution $(\bar \P^*, \bar \Q^*)$ on $\bar \sha$. We then use Lemma \ref{lemma:backToOmega} to compute an optimal solution $(\P^*, \Q^*)$ to the penalized Problem \eqref{eq:penalizedProblemIntro} derived from $(\bar \P^*, \bar \Q^*)$. We start by a useful technical result, which is Lemma 3.2 in \cite{LackerMarkovianControl}.
\begin{lemma}
  \label{lemma:representationMeasure}
  There exists a $\shf_t^\Lambda$-predictable process
  $\bar \Lambda : [0, T] \times \shv \rightarrow \shp(\U)$ such that
   for each $q \in \shv$,
  $\Lambda(q)(dt, du) =  dt \bar \Lambda_t(q)(du)$.
\end{lemma}
Based on Lemma \ref{lemma:representationMeasure}, we can now write the canonical process $(X, \Lambda)$ on $\bar \Omega$ as $(X, dt\bar \Lambda_t(du))$.
\begin{remark} \label{rmk:bfLambdaEst}
	We list below some facts that will be useful to prove Theorem \ref{th:existenceSolutionRegProb}.
	\begin{enumerate}
		\item We immediately deduce from Hypothesis \ref{hyp:coefDiffusion} item 2. and Hypothesis \ref{hyp:costFunctionsControl} item 1. that, for all $t \in [0, T]$,
		\begin{equation}
			\label{eq:linearGrowthRelaxed}
			\left|\int_\U b(t, X_t, u)\bar \Lambda_t(du)\right| \le \int_\U|b(t, X_t, u)|\bar \Lambda_t(du) \le C_{b, \sigma}(1 + |X_t|),
		\end{equation}
		and
		\begin{equation}
			\label{eq:polyGrowthRelaxed}
			\left|\int_\U f(t, X_t, u)\bar \Lambda_t(du)\right| \le \int_\U |f(t, X_t, u)|\bar \Lambda_t(du) \le C_{f, g}(1 + |X_t|^p).
		\end{equation}
		
              \item Let $\bar \P \in \bar \shp_\U$.
                Taking into account decomposition \eqref{eq:decompRelaxed},
 \eqref{eq:linearGrowthRelaxed} as well as
         linear growth of the diffusion coefficient $\sigma$ in Hypothesis \ref{hyp:coefDiffusion} item $2.$, 
         we can apply Lemma \ref{lemma:classicalEstimates}.
This yields that
  for all $q \ge 1$, there exists a constant $C(q)$ which only depends on $C_{b, \sigma}$, $T$ and $q$ such that
		\begin{equation}
			\label{eq:momentBarP}
			\E^{\bar \P}\left[\left(\sup_{0 \le t \le T} |X_t|\right)^q\right] \le C(q) < + \infty.
		\end{equation}
		
		\item Hypothesis \ref{hyp:coefDiffusion} item 2. implies in particular that
		$$
		|b(t, x, u)| \le C(1 + |x|^p + |u|^p),
		$$
		for some constant $C > 0$. Since $b$ is continuous in $(t, x, u) \in [0, T] \times \R^d \times \U$ by Hypothesis \ref{hyp:coefDiffusion} item 1., by
		Corollary A.5 in \cite{LackerMarkovianControl} applied with $A = \U$, $E = \R^d$ and $\phi = b$, the map
		$$
		(X, \Lambda) \mapsto \int_{[0, t] \times \U} b(r, X_r, u)\Lambda(dr, du),
		$$
		is continuous for $d_{\bar \Omega}$. Similarly, Hypothesis \ref{hyp:costFunctionsControl} implies by Corollary A.5 in \cite{LackerMarkovianControl} that the map
		$$
		(X, \Lambda) \mapsto \int_{[0, T] \times \U}f(r, X_r, u)\Lambda(dr, du)
		$$
		is continuous for $d_{\bar \Omega}$.
	\end{enumerate}	
\end{remark}

We will need the following simple two technical observations.

\begin{lemma}
	\label{lemma:weakConvUnbounded}
	Let $(\P_n)_{n \ge 1}$ be a sequence of Borel probability measures on
        a Polish space $E$ that weakly converges towards a probability measure $\P_{\infty}.$ Let $\phi : E \rightarrow \R$ be a continuous function. Assume that there exists $\alpha, C > 0$ such that
	\begin{equation}
		\label{eq:uniformInt}
		\sup_{n\ge 1} \int_{E} |\phi(e)|^{1 + \alpha}\P_n(de) \le C.
	\end{equation}
	Then
	$$
	\int_{E}\phi(e)\P_n(de) \underset{n \rightarrow + \infty}{\longrightarrow} \int_{E}\phi(e)\P_{\infty}(de).
	$$ 
\end{lemma}

\begin{proof}
	By Skorokhod's representation theorem, there exists a probability space $( \Omega, \F, \mathbb{Q})$, a sequence of random variable $(X_n)_{n \ge 1}$ on $ \Omega$ and a random variable $X$ such that $\mathcal{L}^{\mathbb{Q}}(X_n)
	= \P_n$ and $X_n \rightarrow X$ $\mathbb{Q}$-a.s. Condition \eqref{eq:uniformInt} implies that the sequence $(\phi(X_n))_{n \ge 1}$ is uniformly integrable. Furthermore, by continuity of $\phi$, $\phi(X_n) \underset{n \rightarrow + \infty}{\longrightarrow} \phi(X)$ $\mathbb{Q}$-a.s. Thus
	$$\E^{\mathbb{Q}}[\phi(X_n)] \underset{n \rightarrow + \infty}{\longrightarrow} \E^{\mathbb{Q}}[\phi(X)]$$
	or equivalently
	$$
	\int_{E}\phi(e)\P_n(de) \underset{n \rightarrow + \infty}{\longrightarrow} \int_{E}\phi(e)\P_{\infty}(de).
	$$ 
\end{proof}

\begin{lemma}
	\label{lemma:relaxedBoundDensity}
	Let $\bar \P \in \bar\shp_\U$. Let $\bar \Q \in \shp(\bar \Omega)$ be defined by
	$$
	d\bar \Q := \frac{\exp\left(-\int_0^T \int_\U f(r, X_r, u)\bar \Lambda_r(du)dr\right)}{\E^{\bar \P}\left[\exp\left(-\int_0^T \int_\U f(r, X_r, u)\bar \Lambda_r(du)dr\right)\right]}d\bar \P.
	$$
	There exists a constant $C > 0$ only depending
        on  $C_{b, \sigma}, C_{f,g}, T$ and $p$
        such that $\|d\bar \Q/d\bar \P\|_\infty \le C < + \infty$.
\end{lemma}
\begin{proof}
	On the one hand, since $f \ge 0$,
	\begin{equation}
		\label{eq:boundExp}
		\exp\left(-\int_0^T \int_\U f(r, X_r, u)\bar \Lambda_r(du)dr\right) \le 1.
	\end{equation}
	On the other hand, from \eqref{eq:polyGrowthRelaxed} and \eqref{eq:momentBarP} in Remark \ref{rmk:bfLambdaEst}, there exists a constant $C(p)$ which only depends on $C_{b, \sigma}$, $T$ and $p$ such that
	\begin{equation}
		\E^{\bar \P}\left[\int_0^T \int_\U f(r, X_r, u)\bar
		\Lambda_r(du)dr\right] \le C_{f, g}T(1 + C(p)).
	\end{equation}
	Then by Jensen's inequality we have
	\begin{equation}
		\label{eq:jensenEntropy}
		\begin{aligned}
			\E^{\bar \P}\left[\exp\left(-\int_0^T\int_\U f(r, X_r, u)\bar \Lambda_r(du)dr \right)\right]
			& \ge \exp\left( \E^{\bar \P} \left[-\int_0^T\int_\U f(r, X_r, u)\bar \Lambda_r(du)dr\right]\right)\\
			& \ge \exp(- C_{f, g}T(1 + C(p))).
		\end{aligned}
	\end{equation}
	Combining \eqref{eq:boundExp} and \eqref{eq:jensenEntropy} we get $\left\|d\bar \Q_n/d\bar \P_n\right\|_\infty \le C$ by setting $C := \exp\left(C_{f, g}T(1 + C(p))\right)$.
\end{proof}
We can now start the proof of Theorem \ref{th:existenceSolutionRegProb}.
\begin{lemma}
	\label{lemma:tightness}
	There exists a minimizing sequence $(\bar \P_n, \bar \Q_n)_{n \ge 1}$ of $\bar \shj$ verifying the following.
	\begin{enumerate}
		\item $\underset{n \ge 1}{\sup}\left\|\frac{d\bar \Q_n}{d\bar \P_n}\right\|_\infty < + \infty$ and $\underset{n \ge 1}{\sup}~H(\bar \Q_n | \bar \P_n) < + \infty$.
		\item $(\bar \P_n, \bar \Q_n)_{n \ge 1}$ is relatively compact in $(\shp^p(\bar \Omega))^2$.
	\end{enumerate}
\end{lemma}
\begin{proof}
	In this proof, $C$ denotes a generic non-negative constant. Let $(\bar \P_n, \tilde \Q_n)_{n \ge 1}$ be a minimizing sequence of $\bar \shj$. Setting
	\begin{equation}
		\label{eq:densityRelaxed}
		d\bar \Q_n := \frac{\exp\left(-\int_0^T \int_\U f(r, X_r, u)\bar \Lambda_r(du)dr\right)}{\E^{\bar \P_n}\left[\exp\left(-\int_0^T \int_\U f(r, X_r, u)\bar \Lambda_r(du)dr\right)\right]}d\bar \P_n,
	\end{equation}
	by Proposition \ref{prop:markovExistenceMinimizer},
    $  \inf_{\Q} \bar \shj(\Q, \bar \P_n) = {\bar \shj}(\bar \Q_n, \bar \P_n)$
    so that $\bar \shj(\bar \Q_n, \bar \P_n) \le \bar \shj(\tilde \Q_n, \bar \P_n)$. Hence $(\bar \P_n, \bar \Q_n)_{n \ge 1}$ is still a minimizing sequence of $\bar \shj$. Since $\bar \Q_n$ is defined by \eqref{eq:densityRelaxed},
    $\underset{n \ge 1}{\sup}\left\|\frac{d\bar \Q_n}{d\bar \P_n}\right\|_\infty < + \infty$
    and $\underset{n \ge 1}{\sup}~H(\bar \Q_n | \bar \P_n) < + \infty$ by Lemma \ref{lemma:relaxedBoundDensity}.
    This establishes item 1.
    
    Let us now prove that the sequence $(\bar \P_n, \bar \Q_n)_{n \ge 1}$ is relatively compact in $(\shp^p(\bar \Omega))^2$, i.e. item 2. Notice first
    that since $(\bar \Q_n, \bar \P_n)_{n \ge 1}$ is a minimizing sequence of $\bar \shj$, we have
	\begin{equation}
		\label{eq:supBounded}
		\sup_{n \ge 1} \bar \shj(\bar \Q_n, \bar \P_n) < + \infty.
	\end{equation}
	Let then $q' > 1$. Since \eqref{eq:linearGrowthRelaxed} and \eqref{eq:linearGrowthbSigma} hold, by Problem 3.15, Chapter 5 in \cite{karatshreve}
        applied with $b(t, y) = \int_\U b(t, y, u)\bar \Lambda_r(du)$, there exists a constant $C > 0$ which only depends on $C_{b, \sigma}$, $T$, $q'$ and $d$ such that
	$$
	\E^{\bar \P_n}\left[|X_t - X_s|^{2q'}\right] \le C|t - s|^{q'},
	$$
	hence $\left(\shl^{\bar \P_n}(X)\right)_{n \ge 1}$ is a tight sequence by Kolmogorov criteria, see e.g. Problem 4.11, Chapter 2 in \cite{karatshreve}.
	Moreover, by Hypothesis \ref{hyp:costFunctionsControl} item 3. and \eqref{eq:supBounded},
	\begin{equation}
		\label{eq:momentRelaxedControl}
		\begin{aligned}
                  \sup_{n \ge 1} \E^{\bar \P_n}\left[\int_0^T \int_\U |u|^{p'}\bar \Lambda_r(du)dr\right] & \le C'\left(1 + \sup_{n \ge 1} \E^{\bar \P_n}\left[\int_0^T \int_\U f(r, X_r, u) \bar \Lambda_r(du)dr\right]\right)\\
			& \le C'\left(1 + \sup_{n \ge 1} \bar \shj(\bar \Q_n, \bar \P_n)\right) < + \infty.
		\end{aligned}
	\end{equation}
	Using again \eqref{eq:momentBarP} and by \eqref{eq:momentRelaxedControl} we have 
	\begin{equation}
		\label{eq:normP}
		\sup_{n \ge 1} \E^{\bar \P_n}\left[\left(\sup_{0 \le r \le T}|X_r|\right)^{p'} + \int_0^T \int_\U |u|^{p'}\bar \Lambda_r(du)dr\right] < + \infty,
	\end{equation}
	where we recall that $p' > p \ge 1$ as fixed at the beginning
        of Appendix \ref{app:proofThEx}.
        Since $\left(\shl^{\bar \P_n}(X)\right)_{n \ge 1}$ is tight in $\shp(\Omega)$ and \eqref{eq:normP} holds,
	by Proposition B.3 in \cite{LackerMarkovianControl},
        the sequence $(\bar \P_n)_{n \ge 1}$ is relatively compact in $\shp^p(\bar \Omega)$. Now since $\underset{n \ge 1}{\sup}\left\|\frac{d\bar \Q_n}{d\bar \P_n}\right\|_\infty < + \infty$ by item 1.,  $\left(\shl^{\bar \Q_n}(X)\right)_{n \ge 1}$ is also tight and \eqref{eq:normP} is also verified replacing $\bar \P_n$ by $\bar \Q_n$. Hence $(\bar \Q_n)_{n \ge 1}$ is also relatively compact in $\shp^p(\bar \Omega)$. This concludes the proof.
      \end{proof}

\begin{lemma}
	\label{lemma:limitInBarA}
	Let $(\bar \P_n, \bar \Q_n)_{n \ge 1}$ be a minimizing sequence of $\bar \shj$
        fulfilling items 1. and 2. of Lemma \ref{lemma:tightness} statement.
        Any limit point $(\bar \P, \bar \Q)$ of $(\bar \P_n, \bar \Q_n)_{n \ge 1}$ belongs to $\bar \sha$. 
\end{lemma}
\begin{proof}
  Up to a subsequence, we can assume that the whole sequence $(\bar \P_n, \bar \Q_n)_{n \ge 1}$ converges in $(\shp^p(\bar \Omega))^2$ towards $(\bar \P, \bar \Q)$. Let us prove that $(\bar \P, \bar \Q)$ verifies all items of Definition \ref{def:relaxedSet}. We first check item 1.
  We recall that $E:= \bar \Omega$ is a Polish space.
  By  Remark \ref{rmk:relativeEntropy},
 $(\Q, \P) \mapsto H(\Q | \P)$ is lower semicontinuous with 
 respect to the weak-star convergence on $E^*$.
 Since, the convergence in $\shp^p(\bar \Omega)$ implies the weak convergence, we have
	$$
	H(\bar \Q | \bar \P) \le \liminf_{n \rightarrow + \infty} H(\bar \Q_n | \bar \P_n) < + \infty.
	$$
	where we used item 1. of Lemma \ref{lemma:tightness} to prove the finiteness in  previous inequality.

        We now verify item 2. of Definition \ref{def:relaxedSet}.
        Let $h$ belonging to the space $ C_c^{\infty}(\R^d)$ of real-valued smooth functions with compact support on $\R^d$. We set
	$$
	Y_\cdot := \int_{[0, \cdot] \times \U} b(r, X_r, u)\Lambda(dr, du).
	$$
	By \eqref{eq:decompRelaxed}, under $\bar \P_n$ we have $X = x + Y + M^{\bar \P_n},$
        where  $M^{\bar \P_n}$ is a $(\bar \shf_t)$-local martingale
        verifying $[ M^{\bar \P_n}] = \int_0^\cdot \sigma\sigma^\top(r, X_r)dr$.
         Then by It\^o's formula applied to \eqref{eq:decompRelaxed} under $\bar \P_n$, the process
	$$
	N [h] := h(X_{\cdot} - Y_{\cdot}) - h(x) - \frac{1}{2}\int_0^{\cdot} Tr[\sigma\sigma^\top(r, X_r)\nabla_x^2h(X_r - Y_r)]dr
	$$
	is a local martingale under $\bar \P_n$. Moreover, since $h$ and $\nabla_x^2 h$ are bounded, \eqref{eq:linearGrowthbSigma} and \eqref{eq:momentBarP} implies that
	$$
	\E^{\bar \P_n}\left[\sup_{0 \le t \le T}|N[h]_t|\right] \le 2\|h\|_\infty + T\|\nabla_x^2 h\|_\infty C_{b, \sigma}^2\left(1 + \E^{\bar \P_n}\left[\sup_{0 \le t \le T} |X_t|^2\right]\right) < + \infty,
	$$
	hence $N[h]$ is a genuine $(\bar \P_n, \bar \shf_t)$-martingale.
	We then want to prove that $N[h]$ is also a martingale under $\bar \P$. 
	Let $ 0 \le s < t \le T$. Let $\psi : C([0, s], \R^d) \times \shv_s \rightarrow \R$ be a bounded continuous function, where $\shv_s$ is
the set of the elements of $\shv$ according to 
       Definition \ref{def:relaxedControl}
       where we have replaced $T$ with $s$.
        Then
	\begin{equation}
		\label{eq:martingaleEquality}
		\E^{\bar\P_n}\left[\psi\left(\1_{[0, s]}X, \1_{[0, s]}\Lambda\right)
		N[h]_t\right] =
		\E^{\bar\P_n}\left[\psi\left(\1_{[0, s]}X, \1_{[0, s]}\Lambda\right)
		N[h]_s\right].
	\end{equation}
	On the one hand by Remark \ref{rmk:bfLambdaEst} item 3., the map
    $$
    (X, \Lambda) \mapsto \int_{[0, t] \times \U} b(r, X_r, u)\Lambda(dr, du)
    $$
    is continuous for $d_{\bar \Omega}$, that is $Y = Y(X, \Lambda)$ is continuous for $d_{\bar \Omega}$. Since $\psi$ and $h$ are bounded continuous, the function $(X, \Lambda) \mapsto \psi\left(\1_{[0, s]}X, \1_{[0, s]}\Lambda\right)(h(X_s - Y_s) - h(x))$ is bounded continuous for $d_{\bar \Omega}$ and since $\bar \P_n \rightarrow \bar \P$ weakly,
	\begin{equation}
		\label{eq:id1}
		\E^{\bar \P_n}\left[\psi\left(\1_{[0, s]}X, \1_{[0, s]}\Lambda\right)(h(X_s - Y_s) - h(x))\right] \underset{n \rightarrow + \infty}{\longrightarrow} \E^{\bar \P}\left[\psi\left(\1_{[0, s]}X, \1_{[0, s]}\Lambda\right)(h(X_s - Y_s) - h(x))\right].
	\end{equation}
	On the other hand, since $\nabla_x^2h$ is bounded, \eqref{eq:linearGrowthbSigma} yields for all $r \in [0, T]$
	$$
	\left|Tr[\sigma\sigma^\top(r, X_r)\nabla_x^2h(X_r - Y_r)]\right|
	\le 2C_{b, \sigma}^2\|\nabla_x^2 h\|_\infty(1 + |X_r|^2).
	$$
	Combining the previous inequality with \eqref{eq:momentBarP} we get that for some $\alpha > 0$,
	\begin{equation}
		\label{eq:estimateWeakConv}
		\sup_{n \ge 1}\sup_{r \in [0, T]} \E^{\bar \P_n}\left[\left|Tr[\sigma\sigma^\top(r, X_r)\nabla_x^2h(X_r - Y_r)]\right|^{1 + \alpha}\right] < + \infty.
	\end{equation}
	Hence it holds
	$$
	\sup_{n \ge 1} \E^{\bar \P_n}\left[\left|\psi\left(\1_{[0, s]}X, \1_{[0, s]}\Lambda\right)\int_0^s Tr[\sigma\sigma^\top(r, X_r)\nabla_x^2h(X_r - Y_r)]dr\right|^{1 + \alpha}\right] < + \infty,
	$$
	and by Lemma \ref{lemma:weakConvUnbounded} with $E = C([0, s], \R^d) \times \shv_s$, we get
	\begin{equation}
		\label{eq:id2}
		\begin{aligned}
			&\E^{\bar \P_n}\left[\psi\left(\1_{[0, s]}X, \1_{[0, s]}\Lambda\right)\int_0^s Tr[\sigma\sigma^\top(r, X_r)\nabla_x^2h(X_r - Y_r)]dr\right]\\
			& \underset{n \rightarrow + \infty}{\longrightarrow} \E^{\bar \P}\left[\psi\left(\1_{[0, s]}X, \1_{[0, s]}\Lambda\right)\int_0^s Tr[\sigma\sigma^\top(r, X_r)\nabla_x^2h(X_r - Y_r)]dr\right].
		\end{aligned}
	\end{equation}
	Combining \eqref{eq:id1} and \eqref{eq:id2} and letting $n \rightarrow + \infty$ in \eqref{eq:martingaleEquality} yields
	$$
	\E^{\bar\P}\left[\psi\left(\1_{[0, s]}X, \1_{[0, s]}\Lambda\right)N[h]_t\right] = \E^{\bar\P}\left[\psi\left(\1_{[0, s]}X, \1_{[0, s]}\Lambda\right)N[h]_s\right].
	$$
	Hence the process $N[h]$ is an $((\bar \F_t), \bar \P)$-martingale for all $h \in C_c^{\infty}(\R^d)$.
	By standard stochastic calculus arguments, this implies that under $\bar \P$ the process  writes $X_t = x + Y_t + M_t^{\bar \P}$, where $M^{\bar \P}$ is a $(\bar \F_t)$-local martingale verifying $[ M^{\bar \P}] = \int_0^\cdot \sigma\sigma^\top(r, X_r)dr$. Item 2. of Definition \ref{def:relaxedSet} is verified and we conclude that $(\bar \P, \bar \Q) \in \bar \sha$.
\end{proof}
\begin{prop}
	\label{prop:relaxedSolution}
	The Problem \eqref{eq:relaxedProblem} admits a solution $(\bar \P^*, \bar \Q^*) \in \bar \sha$, in the sense that $\bar \shj^* = \bar \shj(\bar \Q^*, \bar \P^*)$ which verifies $\|d\bar \Q^*/d\bar \P^*\|_\infty < + \infty$.
\end{prop}
\begin{proof}
	Let $(\bar \P_n, \bar \Q_n)_{n \ge 1}$ be the minimizing sequence given by Lemma \ref{lemma:tightness} and let $(\bar \P, \bar \Q)$ be any limit point of the sequence $(\bar \P_n, \bar \Q_n)_{n \ge 1}$. Up to a subsequence we can assume that the whole sequence $(\bar \P_n, \bar \Q_n)_{n \ge 1}$ converges towards $(\bar \P, \bar \Q)$ in $(\shp^p(\bar \Omega))^2$. Recall that by Remark \ref{rmk:bfLambdaEst} the map
	$$
	(X, \Lambda) \mapsto \int_{[0, T] \times \U}f(r, X_r, u)\Lambda(dr, du)
	$$
	is continuous for $d_{\bar \Omega}$. Now by \eqref{eq:polyGrowthRelaxed}, we have
	$$
	\left|\int_{[0, T] \times \U}f(r, X_r, u)\Lambda(dr, du)\right| \le C_{f, g}\left(1 + \sup_{0 \le r \le T}|X_r|^p\right),
	$$
	and by \eqref{eq:momentBarP}, we deduce that
	$$
	\sup_{n \ge 1}\E^{\bar \P_n}\left[\left|\int_{[0, T] \times \U}f(r, X_r, u)\Lambda(dr, du)\right|^{1 + \alpha}\right] < + \infty
	$$
	for any $\alpha > 0$. Since $\underset{n \ge 1}{\sup}\left\|\frac{d\bar \Q_n}{d\bar \P_n}\right\|_\infty$ by item 1. of Lemma \ref{lemma:tightness}, it also holds that
	$$
	\sup_{n \ge 1}\E^{\bar \Q_n}\left[\left|\int_{[0, T] \times \U}f(r, X_r, u)\Lambda(dr, du)\right|^{1 + \alpha}\right] < + \infty.
	$$
	Then by Lemma \ref{lemma:weakConvUnbounded} applied with $E = \bar \Omega$, we have
	\begin{equation}
		\label{eq:convCost1}
		\E^{\bar \Q_n}\left[\int_{[0, T] \times \U}f(r, X_r, u)\Lambda(dr, du)\right] \underset{n \rightarrow + \infty}{\longrightarrow} \E^{\bar \Q}\left[\int_{[0, T] \times \U}f(r, X_r, u)\Lambda(dr, du)\right].
              \end{equation}
Again by  Remark \ref{rmk:relativeEntropy},
 $(\Q, \P) \mapsto H(\Q | \P)$ is lower semicontinous with 
 respect to the weak-star convergence on ${\bar \Omega}$, and
we have
	\begin{equation}
		\label{eq:convCost2}
		H(\bar \Q | \bar \P) \le \liminf_{n \rightarrow + \infty} H(\bar \Q_n | \bar \P_n).
	\end{equation}
	Combining \eqref{eq:convCost1} and \eqref{eq:convCost2}, we get
	$$
	\bar \shj^* = \lim_{n \rightarrow + \infty} \bar \shj(\bar \Q_n, \bar \P_n) = \liminf_{n \rightarrow + \infty} \bar \shj(\bar \Q_n, \bar \P_n) \ge \bar \shj(\bar \Q, \bar \P).
	$$
	By Lemma \ref{lemma:limitInBarA}, $(\bar \P, \bar \Q) \in \bar \sha$ and we conclude that $(\bar \P, \bar \Q)$
        achieves the minimum of $\bar \shj$. Moreover, we set $\bar \P^* := \bar \P$ and
	$$
	d\bar \Q^* := \frac{\exp\left(-\int_0^T \int_\U f(r, X_r, u)\bar \Lambda_r(du)dr - g(X_T)\right)}{\E^{\bar \P}\left[\exp\left(-\int_0^T \int_\U f(r, X_r, u)\bar \Lambda_r(du)dr - g(X_T)\right)\right]}d\bar \P^*.
	$$
        By Proposition \ref{prop:markovExistenceMinimizer} we have that
        $\shj(\bar \Q, \bar \P) \ge \shj(\bar \Q^*, \bar \P^*)$.
         Since $(\bar \P^*, \bar \Q^*) \in \bar \sha$, $\bar \shj(\bar \Q^*, \bar \P^*) = \bar \shj(\bar \Q, \bar \P)$ and $(\bar \P^*, \bar \Q^*)$ also achieves the minimum of $\bar \shj$.
         Finally $\|d\bar \Q^*/d\bar \P^*\|_\infty < + \infty$ by Lemma \ref{lemma:relaxedBoundDensity}.
      \end{proof}
\begin{lemma}
	\label{lemma:backToOmega}
	Let $(\bar \P, \bar \Q) \in \bar \sha$ such that $\|d\bar \Q/d\bar \P\|_\infty < + \infty$. There exists $(\P, \Q) \in \sha$ with $\P \in \shp_\U^{Markov}$ such that $\bar \shj(\bar \Q, \bar \P) \ge \shj(\Q, \P)$.
\end{lemma}
\begin{proof}
  Since $H(\bar \Q | \bar \P) < + \infty$, by Theorem \ref{th:girsanovEntropy} applied on the space $\bar \Omega$ equipped with the probability measures $\bar \P$ and $\bar \Q$ with $\delta_r = \int_\U b(r, X_r, u)\bar \Lambda_r(du)$, $a_r = \sigma\sigma^\top(r, X_r)$, there exists a $(\bar \shf_t)$-progressively measurable process $\bar \alpha$ such that under $\bar \Q$,
 the canonical process decomposes as
	\begin{equation} \label{eq:DecompQ}
	X_t = x + \int_0^t \int_\U b(r, X_r, u)\bar \Lambda_r(du)dr + \int_0^t \sigma\sigma^\top(r, X_r)\bar \alpha_rdr + M_t^{\bar \Q},
	\end{equation}
	where the local martingale $M^{\bar \Q}$ verifies $[ M^{\bar \Q} ] = \int_0^\cdot \sigma\sigma^\top(r, X_r)dr$ and
	\begin{equation}
		\label{eq:entropyBar}
		H(\bar \Q | \bar \P) \ge \frac{1}{2}\E^{\bar \Q}\left[\int_0^T |\sigma^\top(r, X_r)\bar \alpha_r|^2dr\right].
	\end{equation}
	The proof consists in two parts. In the first part we establish some useful estimates related to $\bar \Q$ and to the previous decomposition. In the second part we introduce a probability measure $\Q \in \shp(\Omega)$ mimicking the time marginals of $\bar \Q$ for all $t \in [0, T]$ and another probability measure $\P \in \shp_\U^{Markov}$ such that $\bar \shj(\bar \Q, \bar \P) \ge \shj(\Q, \P)$.
	\begin{enumerate}
        \item Note first that since $\|d\bar \Q/d\bar \P\|_\infty < + \infty$, for all $q \ge 1$,  by \eqref{eq:momentBarP} in Remark \ref{rmk:bfLambdaEst} we have
		\begin{equation}
			\label{eq:momentBarQ}
			\E^{\bar \Q}\left[\left(\sup_{0 \le r \le T}|X_r|\right)^q\right] \le \left\|\frac{d\bar \Q}{d\bar \P}\right\|_\infty\E^{\bar \P}\left[\left(\sup_{0 \le r \le T}|X_r|\right)^q\right] < + \infty.
		\end{equation}
		It immediately follows from \eqref{eq:linearGrowthRelaxed} and \eqref{eq:momentBarQ} that
		\begin{equation}
			\label{eq:momentDrift}
			\E^{\bar \Q}\left[\int_0^T \left|\int_\U b(r, X_r, u)\bar \Lambda_r(du)\right|dr\right] < + \infty,
		\end{equation}
		and from \eqref{eq:polyGrowthRelaxed} and \eqref{eq:momentBarQ} that
		\begin{equation}
			\label{eq:momentCost}
			\E^{\bar \Q}\left[\int_0^T \left|\int_\U f(r, X_r, u)\bar \Lambda_r(du)\right|\right] < + \infty.
		\end{equation}
		Finally being $\sigma$ of linear growth because of Hypothesis \ref{hyp:coefDiffusion} item $2.$ and \eqref{eq:momentBarP} in Remark \ref{rmk:bfLambdaEst}, it holds that
		\begin{equation}
			\label{eq:momentSigma}
			\E^{\bar \P}\left[\int_0^T \|\sigma(r, X_r)\|^qdr\right] < + \infty
		\end{equation}
		for all $q \ge 1$. Then we can apply Lemma \ref{lemma:estimateBetaEntrop} item $1.$ which implies that for
                any $1 < q < 2$
		\begin{equation}
			\label{eq:momentEntropyDrift}
			\E^{\bar \Q}\left[\int_0^T |\sigma\sigma^\top(r, X_r)\bar \alpha_r|^qdr\right] < + \infty.
		\end{equation}

      \item We set $\beta_t := \int_\U b(t, X_t, u) \bar \Lambda_t(du) + \sigma\sigma^\top(t, X_t)\bar \alpha_t$ so that,
        taking into account \eqref{eq:DecompQ},
        $X$ decomposes as $X_t = x + \int_0^t \beta_rdr + M_t^{\bar \Q},$
        where $M_t^{\bar \Q}$ is a local martingale such that 
$[ M^{\bar \Q} ] = \int_0^\cdot \sigma\sigma^\top(r, X_r)dr$
   under $\bar \Q$ and \eqref{eq:entropyBar} rewrites

        \begin{equation}
			\label{eq:entropyBeta}
			H(\bar \Q | \bar \P) \ge \frac{1}{2}\E^{\bar \Q}\left[\int_0^T
                          \left|\sigma^{-1}(r, X_r)\left(\beta_r - \int_\U b(r, X_r, u)\bar \Lambda_r(du)\right)\right|^2dr\right],
                      \end{equation}
                      where $\sigma^{-1}$ denotes again the generalized
                      right-inverse of $\sigma$.
                      
        It follows from \eqref{eq:momentDrift}, \eqref{eq:momentEntropyDrift}, and \eqref{eq:momentSigma}
        together with the assumption
        $\|d\bar \Q/d\bar \P\|_\infty < + \infty$,
        that $\E^{\bar \Q}\left[\int_0^T (|\beta_r| + \|\sigma(r, X_r)\|)dr \right] < + \infty$. Then by Corollary 3.7 in \cite{MimickingItoGeneral} there exists a measurable function $\Gamma : [0, T] \times \R^d \mapsto \R^d$ and a probability measure $\Q$ on $(\Omega, \shf)$ such that the following holds.
		\begin{itemize}
                \item For all $0 \le t \le T$,
                \begin{equation}
                	\label{eq:Gamma}
                	\Gamma(t, X_t) = \E^{\bar \Q}[\beta_t~|~X_t]
                	\quad dt \otimes d\bar \Q\text{-a.s.}
                \end{equation}
			\item Under $\Q$ the canonical process can be expressed as
			$
			X_t = x + \int_0^t \Gamma(r, X_r)dr + M^\Q_t,
			$
			where $M^\Q$ is a $(\F_t)$-local martingale with
                        $[ M^\Q] = \int_0^{\cdot}\sigma\sigma^\top(r, X_r)dr$.
			\item $\shl^\Q(X_t) = \shl^{\bar \Q}(X_t),$ for all $t \in [0, T]$.
		\end{itemize}
                Since the estimates \eqref{eq:momentDrift} and \eqref{eq:momentCost} hold, Lemma \ref{lemma:condExpConvex} applied with $\Omega = \bar \Omega, \P = \bar \Q$ and
                $$
                (y_t, z_t) = \left(\int_\U b(t, X_t, u)\bar \Lambda_t(du), \int_\U f(t, X_t, u)\bar \Lambda_t(du)\right)
                $$
                gives
                the existence of a measurable function $\bar u \in \shb([0, T] \times \R^d, \U)$ such that for almost all $t \in [0, T]$, $\bar \Q$-a.s.,
		\begin{equation}
			\label{eq:controlBar}
			\left\{
			\begin{aligned}
				& \E^{\bar \Q}\left[\int_\U b(t, X_t, u)\bar \Lambda_t(du) \middle | X_t\right] = b(t, X_t, \bar u(t, X_t))\\
				& \E^{\bar \Q}\left[\int_\U f(t, X_t, u)\bar \Lambda_t(du) \middle | X_t \right] \ge f(t, X_t, \bar u(t, X_t)).
			\end{aligned}
			\right.
		\end{equation}
		\eqref{eq:controlBar} together with Fubini's theorem
                then gives
		\begin{equation}
			\label{eq:inter1}
			\begin{aligned}
			\E^{\bar \Q}\left[\int_0^T \int_\U f(r, X_r, u)\bar \Lambda_r(du) + g(X_T)\right] & = \E^{\bar \Q}\left[\int_0^T f(r, X_r, \bar u(r, X_r))dr + g(X_T)\right]\\
			& = \E^{\Q}\left[\int_0^T f(r, X_r, \bar u(r, X_r))dr + g(X_T)\right].
			\end{aligned}
                      \end{equation}
                      \eqref{eq:controlBar} together with Fubini's theorem and Jensen's inequality for the conditional expectation applied to \eqref{eq:entropyBeta} yields
		\begin{equation}
			\label{eq:inter2bis}
			\begin{aligned}
				H(\bar \Q | \bar \P) & \ge \frac{1}{2}\int_0^T\E^{\bar \Q}\left[ \left|\sigma^{-1}(r, X_r)\E^{\bar \Q}\left[\left(\beta_r - \int_\U b(r, X_r, u)\bar \Lambda_r(du)\right) \middle | X_r\right] \right|^2\right]dr\\
				& = \frac{1}{2}\int_0^T \E^{\bar \Q}\left[|\sigma^{-1}(r, X_r)(\Gamma(r, X_r) - b(r, X_r, \bar u(r, X_r)))|^2\right]dr,\\
			\end{aligned}
		\end{equation}
		where we used \eqref{eq:Gamma} and \eqref{eq:controlBar} in the last equality. Since $\bar \Q$ and $\Q$ have the same time marginals, we deduce from \eqref{eq:inter2bis} and Fubini's theorem that
		\begin{equation}
			\label{eq:inter2}
			\begin{aligned}
				H(\bar \Q | \bar \P) & \ge \frac{1}{2}\int_0^T \E^{\Q}\left[|\sigma^{-1}(r, X_r)(\Gamma(r, X_r) - b(r, X_r, \bar u(r, X_r)))|^2\right]dr\\
				& = \frac{1}{2}\E^{\Q}\left[\int_0^T |\sigma^{-1}(r, X_r)(\Gamma(r, X_r) - b(r, X_r, \bar u(r, X_r)))|^2dr\right].
			\end{aligned}
		\end{equation}
		Finally, let $\P := \P^{\bar u} \in \shp_\U^{Markov}$ be the unique probability measure given
                by Proposition \ref{prop:existencePu}.
                We recall that, by Remark \ref{rmk:uniqDecomp}, the SDE
                \begin{equation*}
 		X_t = x + \int_0^t b(r, X_r, \bar u(r, X_r))dr + M_t^\P,
 	\end{equation*}
        where $M^\P$ is a local martingale with
        $[M^\P]  = \int_0^\cdot \sigma\sigma^\top(r, X_r)dr$,
       admits uniqueness in law.

                As $H(\bar \Q | \bar \P) < \infty$, \eqref{eq:inter2}
                and Lemma \ref{lemma:girsanovEntropy} $2.$
                implies that $H(\Q | \P) < + \infty$ and that
		\begin{equation}
			\label{eq:inter3}
			H(\Q | \P) = \frac{1}{2}\E^{\Q}\left[\int_0^T |\sigma^{-1}(r, X_r)(\Gamma(r, X_r) - b(r, X_r, \bar u(r, X_r)))|^2dr\right].
		\end{equation}
		In particular, $(\P, \Q) \in \sha$ and combining \eqref{eq:inter1}, \eqref{eq:inter2} and \eqref{eq:inter3} yields $\bar \shj(\bar \Q, \bar \P) \ge \shj(\Q, \P)$. This concludes the proof.
	\end{enumerate}
\end{proof}
We are now ready to prove Theorem \ref{th:existenceSolutionRegProb}.
\begin{proof}[Proof of Theorem \ref{th:existenceSolutionRegProb}.]
	Let $(\P, \Q) \in \sha$. Let $\bar \P$ (resp. $\bar \Q$) be the law of $(X, dt\delta_{\nu^{\P}_t}(du))$ under $\P$ (resp. $\Q$). Then $(\bar \P, \bar \Q) \in (\shp(\bar \Omega))^2$ and $X$ has clearly the decomposition \eqref{eq:decompRelaxed} under $\bar \P$. Furthermore, one has $d\bar \Q/d\bar \P = d\Q/d\P \circ \pi_X $, where $\pi_X$ is the first coordinate projection on $\bar \Omega$, and this yields
	$$
	H(\bar \Q | \bar \P) = \E^{\bar \Q}\left[\log \frac{d\bar \Q}{d\bar \P}\right] = \E^{\bar \Q}\left[\log \frac{d\Q}{d\P} \circ \pi_X  \right] = \E^{\Q}\left[\log \frac{d\Q}{d\P}(X)\right] = H(\Q | \P).
	$$
	Hence $H(\bar \Q | \bar \P) < + \infty$, $(\bar \P, \bar \Q) \in \bar \sha$ and since
	$$
	\E^{\bar \Q}\left[\int_{[0, T] \times \U}f(r, X_r, u)\Lambda(dr, du) + g(X_T)\right] = \E^{\Q}\left[\int_0^Tf(r, X_r, \nu^{\P}_r)dr + g(X_T)\right],
	$$
	we get $\bar \shj(\bar \Q, \bar \P) = \shj(\Q, \P)$. Previous computations then show that $\bar \shj^* \le \shj^*$. Let now $(\bar \P^*, \bar \Q^*) \in \bar \sha$ be the solution of \eqref{eq:relaxedProblem} given by Proposition \ref{prop:relaxedSolution}. In particular, $\bar \shj(\bar \Q^*, \bar \P^*) = \bar \shj^* \le \shj^*$. Let also $(\P^*, \Q^*) \in \sha$ be given by Lemma \ref{lemma:backToOmega} applied to $(\bar \P^*, \bar \Q^*)$. We have $\bar \shj(\bar \Q^*, \bar \P^*) \ge \shj(\Q^*, \P^*)$, hence $\shj^* \ge \shj(\Q^*, \P^*)$, that is $\shj^* = \shj(\Q^*, \P^*)$. This implies that $(\P^*, \Q^*)$ is a solution of Problem \eqref{eq:penalizedProblemIntro}.
\end{proof}

\section{Strong and weak controls}
\setcounter{equation}{0}
\renewcommand\theequation{D.\arabic{equation}}

\label{app:equiControl}
Let $(\tilde \Omega, \tilde \F, (\tilde \F_t)_{t \in [0, T]}, \tilde \P)$ be a filtered probability space endowed with a Brownian motion $W$. Let $\mathcal{V}$ be the set of $(\tilde \F_t)$-progressively measurable processes $\nu$ on $(\tilde \Omega, \tilde \F, \tilde \P)$ taking values in $\U$ such that equation \eqref{eq:nu} has a unique strong solution.
We give here some details on the equivalence between a strong formulation of our stochastic optimal control
\eqref{eq:strongControlIntro} formulated on the generic probability space $(\tilde \Omega, \tilde \F, \tilde \P)$, and our optimization problem \eqref{eq:controlProblemIntro}. 
We have the following result.
\begin{prop}
	\label{prop:equivControl}
	Assume Hypotheses \ref{hyp:costFunctionsControl} and \ref{hyp:coefDiffusion}. Recall the definition \eqref{eq:strongControlIntro} of $J^*_{strong}$ and \eqref{eq:controlProblemIntro} of $J^*$. Then $J^{*}_{strong} = J^*$.
\end{prop}
\begin{proof}
	\begin{enumerate}[label = (\roman*)]
        \item We first prove that $J^*_{strong} \ge J^*$. Let $(\nu^n)_{n \ge 0}$ be a minimizing sequence of elements of $\mathcal{V}$ for Problem \eqref{eq:strongControlIntro}. For any $n \in \N$,
          Lemma \ref{lemma:classicalEstimates}
          and \eqref{eq:polyGrowthfg} yields 
		\begin{equation}
			\label{eq:momentBFStrong}
			\E^{\tilde \P}\left[\int_0^T \left(|b(r, X_r^{\nu^n}, \nu_r^n)| + |f(r, X_r^{\nu^n}, \nu_r^n)| + \|\sigma\sigma^\top(r, X_r^{\nu^n})\|\right)dr\right] < + \infty.
                      \end{equation}
		Then by Corollary 3.7 in \cite{MimickingItoGeneral} there exist a measurable function $\Gamma \in \shb([0, T] \times \R^d, \R^d)$ and a probability measure $\P \in \shp(\Omega)$ such that
		\begin{itemize}
			\item For all $0 \le t \le T$, $\Gamma(t, X_t) = \E^{\tilde \P}[b(t, X_t^{\nu^n}, \nu_t^n)~|~X_t^{\nu^n}]$ $d\tilde \P \otimes dt$-a.e.
			
			\item Under $\P$ the canonical process can be expressed as
			$
			X_t = x + \int_0^t \Gamma(r, X_r)dr + M^\P_t,
			$
			where $M^\P$ is a $(\F_t)$-local martingale
			with $[M^\P] = \int_0^{\cdot}\sigma\sigma^\top(r, X_r)dr$.
			\item $\shl^\P(X_t) = \shl^{\tilde \P}(X_t), \ \forall t \in [0,T].$
		\end{itemize}
		Since \eqref{eq:momentBFStrong} holds, by Lemma \ref{lemma:condExpConvex} applied with $\Omega = \tilde \Omega, \P = \tilde \P, X = X^{\nu^n}$ and $(y_t, z_t) = \left(b(t, X_t^{\nu^n}, \nu_t^n), f(t, X_t^{\nu^n}, \nu_t^n)\right)$, there exists a function $u^n \in \shb([0, T] \times \R^d, \U)$ such that for almost all $t \in [0, T]$, $\P$-a.s.,
		\begin{equation}
			\label{eq:uStrong}
			\left\{
			\begin{aligned}
				& \E^{\tilde \P}\left[b(t, X_t^{\nu^n}, \nu_t^n) \middle | X_t^{\nu^n}\right] = b(t, X_t^{\nu^n}, u^n(t, X_t^{\nu^n}))\\
				& \E^{\tilde \P}\left[f(t, X_t^{\nu^n}, \nu_t^n)\middle | X_t^{\nu^n}\right] \ge f(t, X_t^{\nu^n}, u^n(t, X_t^{\nu^n})).
			\end{aligned}
			\right.
		\end{equation}
		By Fubini's theorem and Jensen's inequality for the conditional expectation, by \eqref{eq:uStrong}
we have		\begin{equation} \label{eq:WeSt}
			\begin{aligned}
				\E^{\tilde \P}\left[\int_0^T f(r, X_r^{\nu^n}, \nu_r^n)dr + g(X_T^{\nu^n})\right] & \ge \E^{\tilde \P}\left[\int_0^T f(r, X_r^{\nu^n}, u^n(r, X_r^{\nu^n}))dr + g(X_T^{\nu^n})\right]\\
				& = \E^{\P}\left[\int_0^T f(r, X_r, u^n(r, X_r))dr + g(X_T)\right]\\
				& \ge \inf_{\bar \P \in \Pma_\U}\E^{\bar \P}\left[\int_0^T f(r, X_r, \nu_r^{\bar \P})dr + g(X_T)\right],
			\end{aligned}
		\end{equation}
		where, for the latter inequality,  we have used the fact that $\P \in \shp_\U$. From \eqref{eq:WeSt},
                for all $n \in \N$, we have
		$$
		\E^{\tilde \P}\left[\int_0^T f(r, X_r^{\nu^n}, \nu_r^n)dr + g(X_T^{\nu^n})\right] \ge J^*,
		$$
		and letting $n \rightarrow + \infty$ yields $J^*_{strong} \ge J^*$.
		
              \item We now prove that $J^* \ge J^*_{strong}$. Let us consider a minimizing sequence $(\P_n)_{n \ge 0}$ of elements of $\shp_\U$ for Problem \eqref{eq:controlProblemIntro}. Notice that, taking into account
      Lemma \ref{lemma:classicalEstimates},
 the estimate \eqref{eq:momentBFStrong} still holds if we replace $(X^{\nu^n}, \nu^n, \tilde \P)$ by $(X, \nu^{\P_n}, \P_n)$. Then for all $n \in \N$, again by Corollary 3.7 in \cite{MimickingItoGeneral} together with Lemma \ref{lemma:condExpConvex} applied with $\P = \P_n, (y_t, z_t) = (b(t, X_t, \nu_t^{\P_n}), f(t, X_t, \nu_t^{\P_n}))$, there exist a function $u^n \in \shb([0, T] \times \R^d, \U)$ and a probability measure $\hat \P_n$ on $(\Omega, \F)$ such that the following holds.
		\begin{itemize}
			\item For almost all $t \in [0, T]$, $\P_n$-a.s.
			\begin{equation*}
				\left\{
				\begin{aligned}
					& \E^{\P_n}\left[b(t, X_t, \nu_t^{\P_n}) \middle | X_t\right] = b(t, X_t, u^n(t, X_t))\\
					& \E^{\P_n}\left[f(t, X_t, \nu_t^{\P_n})\middle|X_t\right] \ge f(t, X_t, u^n(t, X_t)).
				\end{aligned}
				\right.
			\end{equation*}
			\item Under $\hat \P$ the canonical process decomposes as
			$$
			X_t = x + \int_0^t b(r, X_r, u^n(t, X_t))dr + M_t^{\hat \P_n},
			$$
			where $M^{\hat \P_n}$ is an $(\F_t)$-local martingale such that $[ M^{\hat \P_n}] = \int_0^{\cdot} \sigma\sigma(r, X_r)dr$.
			\item $\mathcal{L}^\P(X_t) = \mathcal{L}^{\hat \P_n}(X_t)$.
		\end{itemize}
		On the one hand, Fubini's theorem and Jensen's inequality for conditional expectation yield
		\begin{equation}
			\label{eq:strongToWeak}
			\E^{\P_n}\left[\int_0^T f(r, X_r, \nu_r^{\P_n})dr + g(X_T)\right] \ge \E^{\hat \P_n}\left[\int_0^T f(r, X_r, u^n(r, X_r))dr + g(X_T)\right].
		\end{equation}
		On the other hand, Theorem 1.1 in \cite{ZhangStrong} ensures the existence of a unique (strong)
                solution $X = X^{\nu^n}$ (on the space
		$(\tilde \Omega, \tilde \F, (\tilde \F_t)_{t \in [0, T]}, \tilde \P)$
		to the SDE
		$$
		dX_t = b(t, X_t, u^n(t, X_t))dt + \sigma(t, X_t)dW_t, ~X_0 = x.
		$$
		In particular the process $\nu^n := u^n(., X^{\hat u}_.)$ is an element of $\shv$, and we get by \eqref{eq:strongToWeak} that
		$$
		J(\P_n) \ge \E^{\tilde \P}\left[\int_0^T f(r, X_r^{\nu^n}, \nu^n_r)dr + g(X_T^{\nu^n})\right] \ge \inf_{\nu \in \mathcal{V}} \E^{\tilde \P}\left[\int_0^T f(r, X_r^\nu, \nu_r)dr + g(X_T^\nu)\right] = J^*_{strong}.
		$$
		The previous expression gives $J(\P_n) \ge J^*_{strong}$ for all $n \in \N$,
		and letting $n \rightarrow + \infty$ yields $J^* \ge J^*_{strong}$.
	\end{enumerate}
	By item $(i)$, we have $J^*_{strong} \ge J^*$, whereas by item $(ii)$, $J^* \ge J^*_{strong}$. Hence $J^* = J_{strong}$, and this concludes the proof.
\end{proof}

\section{Proofs of two technical lemmata}

\setcounter{equation}{0}
\renewcommand\theequation{E.\arabic{equation}}
\label{app:proofEstimate}
   
\begin{proof}[Proof of Lemma \ref{lemma:infKtx}]
	\begin{enumerate}
        \item The function $\bar F^{t, x}_\beta$ is coercive on $K(t, x)$ in the sense of Definition 2.13 in \cite{BeckFirstOrder}. Since $K(t, x)$ is closed (see Remark \ref{rmk:KtxClosed}), Theorem 2.14 in \cite{BeckFirstOrder} gives the existence of a minimum $(y^*, z^*)$ to $\bar F^{t, x}_\beta$ on $K(t, x)$, which is unique since $\bar F^{t, x}_\beta$ is strictly convex. 
		Let then $(y, z) \in K(t, x)$. Since $K(t, x)$ is convex, $(\lambda y + (1 - \lambda)y^*, \lambda z + (1 - \lambda)z^*) \in K(t, x)$, for any $\lambda \in ]0, 1]$. By definition of $(y^*, z^*)$ we then have
		$$
		\frac{\bar F_\beta^{t, x}(\lambda y + (1 - \lambda)y^*, \lambda z + (1 - \lambda)z^*) - \bar F_\beta^{t, x}(y^*, z^*)}{\lambda} \ge 0, \quad \text{for all}~\lambda \in ]0, 1],
		$$
		and since $\bar F_\beta^{t, x}$ is of class $\shc^1$ on $\R^d \times \R$, letting $\lambda \rightarrow 0$ in the previous inequality yields $\langle \nabla_{(y, z)}\bar F_\beta^{t, x}(y^*, z^*), (y, z) \rangle \ge 0$, which rewrites as \eqref{eq:firstOrder}.
		\item
                  We first observe that
                  $u^* \in \underset{a \in \U}{\argmin}~F_\beta(t, x, a)$
                  is equivalent to
    \begin{equation} \label{eq:ustar}
                    F_\beta(t, x, u^*) \le   F_\beta(t, x, a), \
                    \forall a \in \U 
                  \end{equation}
                  and \eqref{eq:ustarB}
                  is equivalent to
\begin{equation} \label{eq:ustarA}
	\bar F^{t, x}_\beta(y, z) \ge \bar F^{t, x}_\beta(y^*, z^*), \ \forall (y,z) \in K(t,x).
		\end{equation}
                For any $a \in \U$ we set now
                $ (y(a),z(a)):= (b(t,x,a),f(t,x,a)).$
                Clearly $(y(a),z(a)) \in K(t,x)$ and
                $(y,z) \in K(t,x)$ if and only if
                there is $a \in U$ with
                $(y,z) = (y(a),z)$ and $z \ge z(a)$.

                In fact we have
                \begin{equation} \label{eq:A1}
                  {\bar F}^{t,x}(y(a),z(a)) = F_\beta(t,x,a).
                  \end{equation}
               \begin{enumerate}
                \item
                  Let $u^* \in \U$ such that $y^* = y(u^*)$
                  and
                  $z^* \ge z(u^*)$
                  and we prove \eqref{eq:ustar}.
                 By \eqref{eq:A1},  for all $a \in U$, we have
                  \begin{eqnarray*}
 F_\beta(t,x,a) &=& {\bar F}^{t,x}(y(a),z(a)) \ge {\bar F}^{t,x}(y^*,z^*)) =  {\bar F}^{t,x}(y(u^*),z^*)) \\
&\ge& {\bar F}^{t,x}(y(u^*),z(u^*)) = F_\beta(t,x,u^*)
       \end{eqnarray*}
       and \eqref{eq:ustarA} follows.
     \item Let $u^*$ such that \eqref{eq:ustar} holds and
       $a \in \U$ such that
       $y = y(a), z \ge z(a)$.
       Then, using again \eqref{eq:A1} we get
          \begin{eqnarray*}
            {\bar F}^{t,x}(y,z)) &\ge& {\bar F}^{t,x}(y(a),z(a))) = F(t,x,a) \\
         &\ge& F(t,x,u^*)    
= {\bar F}^{t,x}(y(u^*),z(u^*)),
       \end{eqnarray*}
       and \eqref{eq:ustarB} holds.
  \end{enumerate}
\end{enumerate}
\end{proof}

\begin{proof}[Proof of Lemma \ref{lemma:estimateOptimalSolution}]


By Remark \ref{rmk:OptVal} and
\eqref{eq:optimalSolutionDensity} in Theorem
\ref{th:existenceSolutionRegProb},
the quantity
 $C_\infty := \|d\Q^*_\epsilon/d\P^*_\epsilon\|_\infty$ is finite, so that
\begin{equation}
	\label{eq:estimateStar}
	\E^{\Q^*_\epsilon}\left[\sup_{0 \le r \le T}|X_r|^q\right] \le
        C_\infty \E^{\P^*_\epsilon}\left[\sup_{0 \le r \le T}|X_r|^q\right] \le C_\infty C(q) < + \infty,
\end{equation}
where $C(q)$ is given by Lemma \ref{lemma:classicalEstimates}. As $H(\Q^*_\epsilon | \P^*_\epsilon) < + \infty$, by Theorem \ref{th:girsanovEntropy}
there exists a progressively measurable process $\alpha$ such that under $\Q_\epsilon^*$ the canonical process decomposes as
$$
X_t = x + \int_0^t b(r, X_r, u^*_\epsilon(r, X_r))dr + \int_0^t \sigma\sigma^\top(r, X_r)\alpha_rdr + M^*_t, \ t \in [0,T],
$$
where $M^* := M^{\Q^*_\epsilon}$ is a local martingale verifying $[ M^*] = \int_0^\cdot \sigma\sigma^\top(r, X_r)dr$
and $ u^*_\epsilon$ is the Borel function introduced in
Theorem \ref{th:existenceSolutionRegProb}.

Moreover,
\begin{equation}
	\label{eq:entropStar}
	H(\Q^*_\epsilon | \P^*_\epsilon) \ge \frac{1}{2}\E^{\Q_\epsilon^*}\left[\int_0^T |\sigma^\top(r, X_r)\alpha_r|^2dr\right].
\end{equation}
We set $\beta_t := b(t, X_t, u^*_\epsilon(t, X_t)) + \sigma\sigma^\top(t, X_t)\alpha_t$

Let $k \ge 1$. On the one hand, combining \eqref{eq:linearGrowthbSigma}, \eqref{eq:sigmaElliptic} \eqref{eq:estimateStar} and \eqref{eq:entropStar}
and taking into account \eqref{eq:decompQn} for $k+1$ replaced
with $k$, 
we have
\begin{equation}
	\label{eq:interEstimateOpti1}
	\begin{aligned}
		\E^{\Q_\epsilon^*}\left[\int_0^T |\sigma^{-1}(r, X_r)(\beta_r - b(r, X_r, u^k(r, X_r)))|^2dr\right] & \le 4\E^{\Q_\epsilon^*}\left[\int_0^T |\sigma^{-1}(r, X_r)b(r, X_r, u^*_\epsilon(r, X_r))|^2dr\right]\\
		& + 4\E^{\Q_\epsilon^*}\left[\int_0^T |\sigma^{-1}(r, X_r)b(r, X_r, u^k(r, X_r))|^2dr\right]\\
		& + 4\E^{\Q_\epsilon^*}\left[\int_0^T |\sigma^\top(r, X_r)\alpha_r|^2dr\right]\\
		& \le 8c_\sigma C_{b, \sigma}^2\left(T + \int_0^T \E^{\Q_\epsilon^*}[|X_r|^2]dr\right) + 8H(\Q_\epsilon^* | \P_\epsilon^*)\\
		& \le 8Tc_\sigma C_{b, \sigma}^2(1 + C_\infty C(2)) + 8H(\Q_\epsilon^* | \P_\epsilon^*).
	\end{aligned}
\end{equation}

  We recall that, by Remark \ref{rmk:uniqDecomp}, the SDE
                \begin{equation*}
 		X_t = x + \int_0^t b(r, X_r, u^k(r, X_r))dr + M_t^{\P_k},
 	\end{equation*}
        where $M^{\P_k}$ is a local martingale with
        $[M^{\P_k}]  = \int_0^\cdot \sigma\sigma^\top(r, X_r)dr$,
       admits uniqueness in law.

The inequality \eqref{eq:interEstimateOpti1} implies by Lemma \ref{lemma:girsanovEntropy} $2$.
that
$$H(\Q_\epsilon^* | \P_k) = \frac{1}{2}	\E^{\Q_\epsilon^*}\left[\int_0^T |\sigma^{-1}(r, X_r)(\beta_r - b(r, X_r, u^k(r, X_r)))|^2dr\right] < + \infty,
$$
hence
\begin{equation}
	\label{eq:interEstimateOpti2}
	H(\Q_\epsilon^* | \P_k) \le 4Tc_\sigma C_{b, \sigma}^2(1 + C_\infty C(2)) + 4H(\Q_\epsilon^* | \P_\epsilon^*).
\end{equation}
On the other hand, by \eqref{eq:polyGrowthfg} and \eqref{eq:estimateStar},
\begin{equation}
	\label{eq:interEstimateOpti3}
	\E^{\Q_\epsilon^*}\left[\int_0^T f(r, X_r, u^k(r, X_r))dr + g(X_T)\right] \le (T + 1)C_{f, g}(1 + C_\infty C(p)).
      \end{equation}
Taking into account \eqref{eq:penalizedProblemIntro} and
combining \eqref{eq:interEstimateOpti2} and \eqref{eq:interEstimateOpti3} yields
\begin{equation}
	\label{eq:interEstimateOpti4}
	\begin{aligned}
		\shj(\Q_\epsilon^*, \P_k) & = \E^{\Q_\epsilon^*}\left[\int_0^T f(r, X_r, u^k(r, X_r))dr + g(X_T)\right] + \frac{1}{\epsilon}H(\Q_\epsilon^* | \P_k)\\
		& \le (T + 1)C_{f, g}(1 + C_\infty  C(p)) + \frac{4Tc_\sigma C_{b, \sigma}^2(1 + C_\infty C(2))}{\epsilon} + \frac{4}{\epsilon}H(\Q_\epsilon^* | \P_\epsilon^*).
	\end{aligned}
\end{equation}
Finally, by \eqref{eq:optimalSolutionDensity} and Jensen's inequality,
\begin{equation}
	\label{eq:interEstimateOpti5}
	\begin{aligned}
		\frac{1}{\epsilon}H(\Q_\epsilon^* | \P_\epsilon^*) & \le - \frac{1}{\epsilon}\log\left(\E^{\Q_\epsilon^*}\left[\exp\left(-\epsilon \int_0^T f(r, X_r, u_\epsilon^*(r, X_r))dr - \epsilon g(X_T)\right)\right]\right)\\
		& \le \E^{\Q_\epsilon^*}\left[\int_0^T f(r, X_r, u_\epsilon^*(r, X_r))dr + g(X_T)\right]\\
		& \le (T + 1)C_{f, g}(1 + C_\infty  C(p)),
	\end{aligned}
\end{equation}
where we have used \eqref{eq:polyGrowthfg} and \eqref{eq:estimateStar} for the last inequality. Injecting \eqref{eq:interEstimateOpti5} in \eqref{eq:interEstimateOpti4} yields the desired result by setting $C := 5(T + 1)C_{f, g}(1 + C_\infty  C(p)) + 4Tc_\sigma C_{b, \sigma}^2(1 + C_\infty C(2))$.
\end{proof}

\section{Miscellaneous}
\setcounter{equation}{0}
\renewcommand\theequation{F.\arabic{equation}}

We gather in this section two useful technical results. In the following, all the random variables are defined on a filtered probability space $(\Omega, \F, (\F_t)_{t \in [0, T]}, \P)$.
\begin{lemma}
	\label{lemma:squareIntVar}
	Let $\eta$ be a square integrable, non-negative random variable. Then for all $\epsilon > 0$,
	$$
	0 \le \E[\eta] - \left(-\frac{1}{\epsilon}\log\E\left[\exp(-\epsilon \eta)\right]\right) \le
       \epsilon Var[\eta] e^{\epsilon \E[\eta]}.
	$$
        
\end{lemma}
\begin{proof}

  
	For all  $ b \in \R$, it holds by Taylor's formula with integral remainder that
	$$
	e^{- b} = 1 - b  +  b^2\int_0^1 (1-t) e^{-t b }dt. $$
	 A direct application of this formula with 
	$b = \epsilon(\eta(\omega) - \E[\eta])$ for all $\omega \in \Omega$,
	yields
	$$
	e^{-\epsilon(\eta - \E[\eta])} = 1 - \epsilon(\eta - \E[\eta])
        + \epsilon^2
          (\eta - \E[\eta])^2
        \int_0^1 (1-t) e^{-t \epsilon(\eta - \E[\eta])} dt \le
       1 - \epsilon (\eta - \E[\eta]) + \epsilon^2 (\eta - \E[\eta])^2
         e^{\epsilon \E[\eta]},
         $$
taking into account that $\eta \ge 0$.

         	Taking the expectation in previous inequality  we get
	\begin{equation*}
	      \E\left[e^{-\epsilon(\eta - \E[\eta])}\right]
     \le 1  +  \epsilon^2Var[\eta] e^{\epsilon \E[\eta]}.
      	\end{equation*}
   	Since $\log(1 + x) \le x$ for all $x > -1,$ we have
	$$
	\frac{1}{\epsilon}\log \E\left[e^{-\epsilon(\eta - \E[\eta])}\right] \le
        \epsilon Var[\eta] e^{\varepsilon \E[\eta]}.
	$$
	Notice that $\E[\eta]$ is a constant, hence $\frac{1}{\epsilon}\log \E\left[e^{-\epsilon(\eta - \E[\eta])}\right] = \E[\eta] - \left(-\frac{1}{\epsilon}\log \E\left[e^{-\epsilon\eta}\right]\right)$. We then have
	$$
	0 \le \E[\eta] - \left(-\frac{1}{\epsilon}\log \E\left[e^{-\epsilon\eta}\right]\right) \le \epsilon Var[\eta] e^{\epsilon \E[\eta]},
	$$
	where the first inequality follows from Jensen's inequality.

      \end{proof}
      
\begin{lemma}
	\label{lemma:nelsonDerivative}
	Let $(X_t)_{t \in [0, T]}$ be an $(\F_t)$-adapted process of the form
		$$
		X_t = x + \int_0^t b_rdr + M_t,
		$$
	where $\E\left[\int_0^T |b_r|^pdr\right] < + \infty$ for some $p > 1$ and where $M$ is a martingale. For Lebesgue almost all $0 \le t < T$
	$$
	\lim_{h \downarrow 0}\E\left[\frac{X_{t + h} - X_t}{h}~\Big|~\F_t\right] = b_t~\text{in}~L^1(\P).
	$$
\end{lemma}
\begin{proof}
  In this proof we extend the process $X$ by continuity after $T$ and $b_t$ by zero for $t > T$.
	Let $0 < h \le 1$. Notice first that
	$$
	\E\left[\int_0^T \left|\E\left[\frac{X_{t + h} - X_t}{h}~\Big|~\F_t\right] - b_t\right|dt\right] \le \E\left[\int_0^T\left|\frac{1}{h}\int_t^{t + h}b_rdr -b_t\right|dt\right],
	$$
	and that for all $\omega \in \Omega$, for almost all $0 \le t < T$, by Lebesgue differentiation theorem,
	\begin{equation}
		\label{eq:lebesgueDiff}
		\frac{1}{h}\int_t^{t + h}b_r(\omega)dr \underset{n \rightarrow + \infty}{\longrightarrow} b_t.
    \end{equation}
	To conclude by a uniform integrability argument w.r.t. $dt \otimes d\P$ we need to prove that
	$$
	\sup_{0 < h \le 1} \E\left[\int_0^T \left|\frac{1}{h}\int_t^{t + h}b_rdr\right|^pdt\right] < + \infty.
	$$
	Previous expectation, by H\"older inequality, is upper bounded  by
	\begin{equation*}
          \E\left[\int_0^T \frac{1}{h}\int_t^{t + h}|b_r|^pdrdt\right] = \E\left[\int_0^T |b_r|^p \frac{1}{h}\int_{(r - h)_+}^{r}dtdr\right]
          \le \E\left[\int_0^T |b_r|^pdr\right] < + \infty,
	\end{equation*}
	where interchanging the integral inside the expectation is justified by Fubini's theorem.           
	The family $\left(\frac{1}{h}\int_t^{t + h}b_rdr\right)_{0 < h \le 1}$ is uniformly integrable with respect to $dt \otimes d\P$ and we conclude using the Lebesgue's dominated convergence theorem.
\end{proof}
\begin{remark} 	\label{remark:nelsonDerivative}
  If $b_t$ is a.e. $\sigma(X_t)$-measurable then
  the statement of Lemma \ref{lemma:nelsonDerivative} still holds 
  replacing the $\sigma$-field $\shf_t$ with $\sigma(X_t)$.
  This is an obvious property of the tower property of the conditional expectation.
\end{remark}

\section*{Acknowledgments}


The research of the first named author is supported by a doctoral fellowship
PRPhD 2021 of the Région Île-de-France.
The research of the second and third named authors was partially
supported by the  ANR-22-CE40-0015-01 project SDAIM.

\bibliographystyle{plain}
\bibliography{../../../../../BIBLIO_FILE/ThesisBourdais}

\end{document}